\documentclass{article}

\usepackage{layout} %introduit la gestion des marges
\usepackage[top=3cm, bottom=3cm, left=2cm, right=2cm]{geometry} %gestion des marges
\usepackage[english]{babel}
\usepackage[utf8]{inputenc}
\usepackage[T1]{fontenc}
\usepackage{amsmath}
\usepackage{amssymb}
\usepackage{mathrsfs}
\usepackage{amsthm}
\usepackage{enumerate} %(permet les énumérations)
\usepackage[colorlinks=true,urlcolor=blue, citecolor=red,linkcolor=blue]{hyperref}
\usepackage{enumitem}
\usepackage{comment}
\usepackage{mathtools}
\usepackage{esint}
\usepackage{faktor}
\usepackage{color}
\usepackage[capitalize]{cleveref}

\newtheorem{theoremIntro}{Theorem}
\newtheorem{propositionIntro}[theoremIntro]{Proposition}

\newtheorem{theorem}{Theorem}[section]
\newtheorem{proposition}[theorem]{Proposition}
\newtheorem{corollary}[theorem]{Corollary}
\newtheorem{lemma}[theorem]{Lemma}
\newtheorem{remark}[theorem]{Remark}
\newtheorem{defi}[theorem]{Definition}
\newtheorem{claim}[theorem]{Claim}

\newcommand{\scal}[2]{\left\langle #1,#2 \right\rangle}
\newcommand{\g}{\nabla}

\newcommand{\di}{\mathrm{div}}
\newcommand{\lap}{\Delta}
\newcommand{\dr}{\partial}
\newcommand{\vol}{\mathrm{vol}}

\newcommand{\dist}{\mathrm{dist}}
\newcommand{\Span}{\mathrm{Span}}
\newcommand{\Ind}{\mathrm{Ind}}
\newcommand{\tr}{\mathrm{tr}}

\newcommand{\Riem}{\mathrm{Riem}}
\newcommand{\Conf}{\mathrm{Conf}}
\newcommand{\Diff}{\mathrm{Diff}}
\newcommand{\osc}{\mathrm{osc}}

\newcommand{\Imm}{\mathrm{Imm}}

\newcommand{\R}{\mathbb{R}}
\newcommand{\C}{\mathbb{C}}
\newcommand{\N}{\mathbb{N}}

\newcommand{\s}{\mathbb{S}}

\newcommand{\Z}{\mathbb{Z}}
\newcommand{\RP}{\mathbb{RP}}
\newcommand{\Hd}{\mathbb{H}}

\newcommand{\Lr}{\mathcal{L}}
\newcommand{\Dr}{\mathcal{D}}
\newcommand{\Ar}{\mathcal{A}}
\newcommand{\Er}{\mathcal{E}}
\newcommand{\Cr}{\mathcal{C}}
\newcommand{\Ur}{\mathcal{U}}

\newcommand{\Arond}{\mathring{A}}

\newcommand{\inter}[2]{[\![#1,#2]\!]}
\newcommand{\ust}{\underset}

\newcommand{\vp}{\varphi}
\newcommand{\ve}{\varepsilon}

\title{Energy quantization for Willmore surfaces with bounded index}
\author{Dorian Martino\footnote{Partial support through ANR BLADE-JC ANR-18-CE40-002 is acknowledged.\\Institut de Mathématiques de Jussieu, Université Paris Cité, Bâtiment
			Sophie Germain, 75205 Paris Cedex 13, France\\ Email : dorian.martino@imj-prg.fr
		}}

\date{}

\begin{document}
	
	\maketitle
	
	\begin{abstract}
		We prove an energy quantization result for Willmore surfaces with bounded index, whether the underlying Riemann surfaces degenerates in the moduli space or not. To do so, we translate the question on the conformal Gauss map's point of view. In particular, we prove that in a neck or a collar region, the conformal Gauss map converges to a light-like geodesic in the De Sitter space.
	\end{abstract}

	\section{Introduction}

	The Willmore energy has been introduced in the XIX century by Sophie Germain in non linear mechanics as being a modelization of the free energy of a bent two dimensional elastic membrane. It was then independently introduced in geometry by Wilhelm Blaschke \cite{blaschke1955} in an effort to merge minimal surfaces theory and conformal invariance. If $\Phi$ denotes the immersion of an abstract closed surface $\Sigma$ into $\R^3$, the Willmore energy of $\Phi$ is given by
	\begin{align*}
		\Er(\Phi) &:= \int_\Sigma \left| \Arond_\Phi\right|^2_{g_\Phi} d\vol_{g_\Phi},
	\end{align*}
	where $g_{\Phi}$ is the first fundamental form of the immersion, i.e. the induced metric by $\Phi$, $d\vol_{g_\Phi}$ is the
	associated volume form and $\Arond_\Phi$ is the traceless part of the second fundamental form $A_\Phi$ of $\Phi$. An other definition can be
	\begin{align*}
		W(\Phi) := \int_\Sigma |H_\Phi|^2 d\vol_{g_\Phi},
	\end{align*}
	where $H_\Phi$ is the mean curvature of $\Phi(\Sigma)$. Thanks to Gauss-Bonnet, if $\chi(\Sigma)$ is the Euler characteristic of $\Sigma$, we can link $W$ and $\Er$ by $\Er(\Phi) = 2W(\Phi) - 4\pi \chi(\Sigma)$.
	
	Blaschke proved that the quantity $\left|\Arond_\Phi\right|^2_{g_\Phi} d\vol_{g_\Phi}$ is pointwise conformal invariant. That is to say, for any element $\xi\in \Conf(\s^3)$, the Möbius group of conformal
	transformations of $\R^3 \cup \{ \infty\} \simeq \s^3$,
	\begin{align*}
		\left|\Arond_{\xi\circ\Phi} \right|^2_{g_{\xi\circ\Phi}} d\vol_{g_{\xi\circ\Phi}} &= \left| \Arond_\Phi \right|^2_{g_\Phi} d\vol_{g_\Phi}.
	\end{align*}
	Therefore, the lagrangian $\Er$ is invariant under conformal transformations. However, the lagrangian $W$ is not invariant under \textit{every} conformal transformation of $\R^3$, but only under conformal transformations that preserve the topology. In this work, we will focus on $\Er$ rather than $W$.\\
	
	One motivation for the study of the Willmore energy, is to apply Morse theory to understand the topology of the space $\Imm(\Sigma,\R^3)$ of immersions of $\Sigma$ into $\R^3$. For instance, a basic problem would be to understand min-max procedures for $W$ : if $\gamma \in \pi_k(\Imm(\Sigma,\R^3))$ is a generator of regular homotopy of immersions, consider the quantity
	\begin{align*}
		\beta_\gamma & := \inf\left\{  \sup_{t\in\s^k} W(\Phi_t)  :  (\Phi_t)_{t\in\s^k} \in \gamma \right\}.
	\end{align*}
	One can ask few natural questions : can we bound these numbers ? Does there exists any immersion realising these optimization problems ? A first issue is
	that the map $W : \Imm(\Sigma,\R^3) \to \R$ cannot be a Morse function, partly because of the conformal invariance. To overcome this, Tristan Rivière \cite{riviere2020} adopted a viscosity approach. Instead of considering $W$ alone, he adds a "smoother" times a small parameter $\sigma>0$ and then let $\sigma\to 0$. He applied successfully this method for a well choosen smoother, in the case $\Sigma = \s^2$, and obtained that the values of min-max procedures for $W$ on $\s^2$ is the sum of the energy of some Willmore sphere plus the energy of some inverted Willmore spheres. In particular, this result can be applied to the sphere eversion and constitute a step towards the $16\pi$ conjecture. For more information about the viscosity method for geodesics or minimal surfaces, see \cite{riviere2017,riviere2021,michelat2016}. It would be interesting to generalize the study of min-max procedures to surfaces of genus greater or equal to 1. This is one of the motivation of the present work.\\
	
	During a long time the minimal surfaces\footnote{Surfaces with $H=0$ are minimizers of $W$.} and their conformal transformations were the only known critical points of $\Er$. In 1965, Tom Willmore \cite{willmore1965} relaunched the interest for the lagrangian by raising the question of finding compact surface minimizing $W$. He showed that round spheres are absolute minimizers, and conjectured that the Clifford torus should be the unique minimizer among the class of immersed tori, up to conformal transformations. This conjecture has been proved by Fernando Marques and André Neves \cite{marques2014b} in 2014. In 1985, Ulrich Pinkall \cite{pinkall1985} constructed a large class of Willmore tori and proved that inversions of minimal surfaces are not the only Willmore surfaces. In 1984, Robert Bryant \cite{bryant1984} enlightened a duality between Willmore surfaces in $\s^3$ and minimal surfaces in the De Sitter space
	\begin{align*}
		& \s^{3,1} := \{ x\in \R^5 : |x|^2_\eta = 1 \}, \\
		& \eta := (dx^1)^2 + \cdots + (dx^4)^2 - (dx^5)^2.
	\end{align*}
	To each immersion $\Phi : \Sigma\to\R^3$, one can consider its conformal Gauss map $Y :\Sigma\to \s^{3,1}$, see \Cref{section_CGM} for a precise definition, which represent at each $x\in \Sigma$, the sphere in $\R^3$ which is tangent to $\Phi(\Sigma)$ at $\Phi(x)$, with same mean curvature. Bryant showed that $\Phi$ is a smooth Willmore immersion if and only if $Y$ is a smooth conformal harmonic map, in other words a minimal surface. This result can be seen as a generalization of a well known property of constant mean curvature surfaces, namely the Gauss map of a constant mean curvature surface is harmonic. Since conformal transformations of $\R^3$ preserve the set of spheres, the notion of conformal Gauss map is more adapted to the study of Willmore surfaces than the notion of Gauss map.\\
	
	In order to understand the space of Willmore surfaces, one can consider sequences of immersions $\Phi_k$ of a given closed surface $\Sigma$ into $\R^3$ and which are below a given energy level. From an analysis point of view, one of the main difficulty to study Willmore surfaces is the fact that the Euler-Lagrange equation of $W$ is a non-linear supercritical elliptic equation of order four. Nevertheless, it has been proven that Willmore surfaces satisfy an $\ve$-regularity property, either extrinsically \cite{kuwert2001} by Ernst Kuwert and Reiner Schätzle or intrinsically \cite{riviere2008,riviere2014} by Tristan Rivière. This crucial step explains that below a threshold of energy, if the metrics $g_k := g_{\Phi_k}$ are controlled, then up to a subsequence, $(\Phi_k)_{k\in\N}$ converges in the strong $C^\ell$-topology, for any $\ell\in\N$. The $\ve$-regularity leads to concentration compactness phenomena, first developed by Jonathan Sacks and Karen Uhlenbeck \cite{sacks1981} in 1981. Once the $\ve$-regularity is established, a covering argument identifies the points where the energy concentrates. At these points, some "bubbles" are formed and we say that the sequence $(\Phi_k)_{k\in\N}$ satisfy a weak bubble convergence.
	
	A difficult question is then to understand if the whole energy concentrating at these points is given exclusively by the sum of the energies of these bubbles or if there is some additional energy needed to glue these bubbles to the rest of the limit. The regions separating the bubbles between themselves or separating the bubbles with the macroscopic solution are annuli of degenerating conformal classes, called neck region. If the energy inside the necks vanishes, we say that the bubble convergence is strong, see \Cref{bubble_compactness}.
	
	By the uniformisation theorem, each $g_k$ is conformal to a smooth metric $h_k$ of constant curvature. Under the assumption that the sequence $(h_k)_{k\in\N}$ is controlled, Yann Bernard and Tristan Rivière \cite{bernard2014} proved that the sequence, modulo the action of $\Conf(\s^3)$, satisfy a strong bubble convergence.\\
	
	The hypothesis on the conformal class of $(g_k)_{k\in\N}$ can be removed in some cases. Thanks to Duligne-Mumford's description, we have a precise picture of surfaces degenerating in the moduli space. In particular, the regions where the metrics $g_k$ degenerate are given by degenerating cylinders, called collar region. They are conformally equivalent to necks. Paul Laurain and Tristan Rivière \cite{laurain2018} established an estimate on the Green function of the Laplace operator of any degenerating sequence of constant Gauss curvature metric, which permits to extend the weak bubble convergence for Willmore surfaces to the case where the underlying conformal classes degenerate. They also proved the strong bubble convergence in \cite{laurain2018a} with some additional assumptions, see below.\\
	
	From the conformal Gauss map's point of view, we have a sequence of conformal harmonic maps $Y_k : (\Sigma,g_k) \to (\s^{3,1},\eta)$ with bounded energy. Harmonic maps defined on degenerating surfaces with values into Riemannian manifolds have been intensively studied \cite{zhu2008,li2017,chen2012,laurain2014}. However, in the case of harmonic maps with values into Lorentz spaces, only few results are known \cite{zhu2013,bernard2020}.  \\
	
	In the case of harmonic maps defined on a degenerating cylinder, coming as a collar or a neck, with values into a compact Riemannian manifold, Miaomiao Zhu \cite{zhu2008} proved that the loss of energy in collars and necks is explicitly given by the residue coming from the Hopf differential. This residue always vanishes in neck regions, so the energy inside necks always vanishes. However, it is not true in collars and some energy can disappear, see for instance \cite{parker1996}. A geometric explanation is given by Li Chen, Yuxiang Li and Youde Wang \cite{chen2012}. They showed that a sequence of harmonic maps converges to a geodesic on necks and collars. They proved that the behaviour of the parametrization is entirely defined by the behaviour of the residue coming from the Hopf differential, and that the energy loss can be quantified by the asymptotic behaviour of the length of the cylinder shrinking to this geodesic. Yuxiang Li, Lei Liu and Youde Wang \cite{li2017} proved that, if the target is positively curved, the length of a geodesic can be controlled by its index. In particular, they showed that if the harmonic maps have bounded index, then the length of the limit geodesic is bounded, and there is also a strong bubble convergence.  \\
	
	One can also find residues for Willmore immersions, but of different nature. The residue of harmonic maps, generated by the Hopf differential, is given by conformal invariance on the domain. In the Willmore case, the conformal invariance is on the target. Four residues are given by conservation laws coming from Noether theorem, see \cite{bernard2016}. Despite this difference, Laurain-Rivière \cite{laurain2018a} proved that only one the residue can obstruct the strong bubble convergence, and that if this residue converges to zero fast enough, then there is strong bubble convergence.\\
	
	In this work, we study the behaviour of the conformal Gauss maps in necks and collars, see \Cref{section_main_results} for precise statements. First, we extend the result of \cite{chen2012} by proving that on a domain with given oscillations, the conformal Gauss maps converges to a geodesic up to a isometry of $\s^{3,1}$ and reparametrization. The following is the \Cref{bubbling_CGM} :
	\begin{theoremIntro}\label{theoremI}
		Let $(\Phi_k)_{k\in\N}$ be a sequence of Willmore immersions $\Sigma \to \R^3$ satisfying $\sup_k \Er(\Phi_k) <\infty$. Let $(Y_k)_{k\in\N}$ be their conformal Gauss maps. Consider $\Cr_k := [-T_k,T_k]\times \s^1$ a degenerating cylinder arising as a neck or a collar region. There exists a sequence of paths $(M_{t,k})_{t\in[-T_k,T_k],k\in\N}\subset SO(4,1)$ satisfying the following : if there exists $\delta>0$ and a cylinder $\tilde{\Cr}_k :=[t_k-1,t_k+\kappa_k]\times \s^1 \subset \Cr_k$, for any $k\in\N$, such that
		\begin{align*}
			\forall k\in\N,\ \ \ \ \osc_{\tilde{\Cr}_k} M_{t_k,k} Y_k = \delta,
		\end{align*}
		Then there exists a nonconstant light-like straight line $\sigma$ and a parametrization $\alpha_k : (0,s_k)\times \s^1 \to \tilde{\Cr}_k$ such that $(M_{t_k,k} Y_k \circ \alpha_k)_{k\in\N}$ converges to $\sigma$ in the $C^2_{loc}$-topology.
	\end{theoremIntro}
	
	The starting point consists in proving an $\ve$-regularity result for the conformal Gauss map in the spirit of \cite{bernard2020}. However, to prove the convergence to a geodesic, the technics of \cite{chen2012} don't apply here since we work in a Lorentz space. In particular we don't have Cauchy-Schwarz inequality or Poincaré inequality. We developp the approach of \cite{zhu2008}, which is more suitable in our case.\\
	
	The next step to prove strong bubble convergence for sequence of Willmore immersions with bounded energy and index, is to adapt the notion of index on the conformal Gauss map. We define a notion of index $\Ind_\Ar$ for space-like critical points of the area functional $\Ar$, see \Cref{section_CGM} equation \eqref{definition_indexY}. The key proposition, see \cref{equality_index}, shows that the Willmore index of a Willmore immersion $\Phi$ is the area index of its conformal Gauss map $Y$, when we restrict the variations $(Y_t)_{t\in(-1,1)}$ as being conformal Gauss maps of a variation $(\Phi_t)_{t\in(-1,1)}$ of $\Phi$.
	
	\begin{propositionIntro}
		For any Willmore immersion $\Phi : \Sigma\to\R^3$ not totally umbilic, we have $\Ind_\Er(\Phi) = \Ind_\Ar(Y)$.
	\end{propositionIntro}
	
	The stability of Willmore immersions from the conformal Gauss map's point of view have been studied in \cite{palmer1991}. Bennett Palmer found conditions on the conformal Gauss map that leads to stability of the Willmore immersion, for instance if Bryant's quartic is large enough, the Willmore immersion must be unstable. Index of inverted minimal surfaces has been studied in \cite{michelat2021,hirsch2019,hirsch2021}.\\
	
	If we consider a sequence $(\Phi_k)_{k\in\N}$ with bounded energy and index, then we recover a setting similar as \cite{li2017}. But in our case, the limit geodesic is light-like so its length always vanishes. Therefore we need to understand the behaviour of the conformal Gauss maps before the limit. According to the \Cref{theoremI}, the image of the conformal Gauss maps is a space-like cylinder shrinking to a light-like geodesic. If some energy remains in this cylinder, we show that the number of Jacobi fields along $\Phi_k$ goes to infinity. Our main result, see \Cref{quantization_bounded_index}, is the following energy quantization result :
	
	\begin{theoremIntro}
		Let $(\Sigma,h_k)$ be a sequence of closed surfaces with fixed genus, constant curvature and normalized volume if needed. We assume that this sequence converges to a nodal surfaces $(\tilde{\Sigma},\tilde{h})$. Let $(\tilde{\Sigma}^l)_{l\in\inter{1}{q}}$ be the connected components of $\tilde{\Sigma}$. Let $(\Phi_k)_{k\in\N}$ a sequence of conformal Willmore immersion $(\Sigma,h_k) \to \R^3$ satisfying
		\begin{align*}
			\sup_k \Big( \Er(\Phi_k) + \Ind_\Er (\Phi_k)\Big) <\infty.
		\end{align*}
		Then, there exists $q$ branched immersions $\Phi^l_\infty : \tilde{\Sigma}^l \to \R^3$ and a finite number of possibly branched immersions $\omega_j :\s^2\to \R^3$, $j\in\inter{1}{p}$, which are Willmore away from possibly finitely many points, and such that, up to a subsequence,
		\begin{align*}
			\lim_{k\to \infty} \Er(\Phi_k) &= \sum_{l=1}^q \Er(\tilde{\Phi}^l_\infty) + \sum_{j=1}^p \Er(\omega_j).
		\end{align*}
	\end{theoremIntro}
	
	A well-known fact from specialists, but never written explicitly, is that this quantization result for the Lagrangian $\Er$ is equivalent to the quantization for the Lagrangian $W$, see \Cref{quantization_E_W}.\\
	
	One difficulty to study the space $\text{Imm}(\Sigma,\R^3)$ is to prove an energy quantization result on Palais-Smale sequences. An easier question is to know whether there is energy quantization for Willmore surfaces or not. When $\Sigma = \s^2$, any conformal class of metrics is bounded, so there is always quantization by \cite{bernard2014}. Then Rivière \cite{riviere2020} proved it also for Palais-Smale sequences coming from the viscosity method. Our main result can be seen as a step towards the study of Palais-Smale sequences when the underlying surface $\Sigma$ is not a sphere.\\
	
	\textbf{Organization of the paper :}\\
	In the \cref{section_preliminaries}, we introduce the main notations for the rest of the paper. We recall the correspondance between $\Conf(\s^3)$ and $SO(4,1)$. We sketch the construction of the conformal Gauss map and recall the relation with the Willmore energy. We also define degenerating surfaces. In the  \cref{section_bounded_sequence}, we define the bubble convergence, the neck and collar regions. We precise the setting and estimates we are working with. In the \cref{section_main_results}, we states the precise statements of our results and give a sketch of the proofs. In the \cref{proof_theoremA}, \cref{proof_equality_index} and \cref{proof_quantization} we prove the \cref{bubbling_CGM}, the \cref{equality_index} and the  \cref{quantization_bounded_index} respectively.\\

	\textbf{Acknowledgement :} \\
	I would like to thank Paul Laurain for his constant support and advices, and Nicolas Marque for an enlightening conversation and for pointing to me some references.
	
	\section{Preliminaries and notations}\label{section_preliminaries}
	
	We state the theorems for immersions in $\R^3$, but since formulas are easier to handle for immersions in $\s^3$, we will manly work with immersions in $\s^3$, except in \cref{gauge_neck_region}, where some formulas will be recalled. \\
	
	\subsection{Willmore energy}

	Consider $\Sigma$ a closed compact manifold of dimension 2 of genus greater or equal to 1, and $\Psi :\Sigma\to \s^3$ a smooth immersion. Let $\lambda$ be the conformal factor : $g := \Psi^*(\xi_{|\s^3}) = e^{2\lambda} h$, where $\xi$ is the flat metric on $\R^4$ and $h$ is a metric on $\Sigma$ conformal to $g$ with constant curvature. Let $N$ be the Gauss map of $\Psi$, $A$ its second fundamental form, $H$ its mean curvature and $\Arond$ the traceless part of the second fundamental form. They are given by the formulas :
	\begin{align*}
		A_{ij} &:= -\scal{\dr_i \Psi}{\dr_j N}_\xi, \\
		H &:= \frac{1}{2} \tr_g (A) = \frac{e^{-2\lambda} }{2} \scal{\lap_h\Psi}{N}_\xi, \\
		\Arond &= A - Hg,\\
		\g N &= -H\g \Psi - e^{-2\lambda} \Arond\g \Psi.
	\end{align*}
	We consider the energy
	\begin{align*}
		\Er(\Psi) &:= \int_\Sigma \left| \Arond\right|^2_g d\vol_g .
	\end{align*}
	By Gauss-Bonnet, we have
	\begin{align*}
		\Er(\Psi) &= \frac{1}{2}\int_\Sigma |dN|^2_g d\vol_g - \chi(\Sigma).
	\end{align*}
	Therefore, a bound on $\Er(\Psi)$ is equivalent to a bound on the $L^2$-norm of the whole second fundamental form.\\
	Willmore surfaces are smooth according to the $\ve$-regularity, see theorem 1.5 in \cite{riviere2008}. \\
	
	The index of an immersion is defined by
	\begin{align}\label{def_index_Psi}
		\Ind_\Er(\Psi) &:= \dim\Span\left\{ (\dr_t \Psi_t)_{|t=0} : \begin{array}{l}
			(\Psi_t)_{t\in(-1,1)}\ \text{is a smooth familly of immersions }\Sigma\to \s^3,\\
			\Psi_0=\Psi,\\
			(\dr_t \Psi_t)_{|t=0} \in \Span(N),\ \frac{d^2}{dt^2}\Er(\Psi_t)_{|t=0}<0
		\end{array}  \right\} .
	\end{align}
	
	\subsection{Correspondance $\Conf(\R^3) \simeq SO(4,1)$}
	
	The correspondance $\Conf(\s^3) \simeq SO(4,1)$ is known for a long time, see for instance \cite{akivis1996,hertrich-jeromin2003}. The explicit computations can be found in \cite{marque2021}. There is an isomorphism $\vp : \Conf(\s^3) \to SO(4,1)$ given by the following relations, where we identify $\s^3\simeq \R^3\cup\{\infty\}$ :
	\begin{itemize}
		\item If $t_a : \R^3 \to \R^3$ is the translation by a vector $a\in \R^3$, then
		\begin{align*}
			\vp(t_a) = \begin{pmatrix}
				I_3 & -a & a \\
				a^T & 1 - \frac{|a|^2_\xi}{2} & \frac{|a|^2_\xi}{2} \\
				a^T & - \frac{|a|^2_\xi}{2} & 1+ \frac{|a|^2_\xi}{2}
			\end{pmatrix},
		\end{align*}
		\item If $d_\lambda : \R^3 \to \R^3$ is the dialation by a factor $e^\lambda$, then
		\begin{align*}
			\vp(d_\lambda) &= \begin{pmatrix}
				I_3 & 0 & 0 \\
				0 & \cosh \lambda & \sinh \lambda \\
				0 & \sinh \lambda & \cosh \lambda
			\end{pmatrix},
		\end{align*}
		\item If $R_\Theta : \R^3 \to \R^3$ is the rotation by a matrix $\Theta \in SO(3)$ :
		\begin{align*}
			\vp(R_\Theta) &= \begin{pmatrix}
				\Theta & 0 & 0\\
				0 & 1 & 0\\
				0 & 0 & 1
			\end{pmatrix},
		\end{align*}
		\item If $\iota : \R^3 \to \R^3$ is the inversion with respect to $B_1(0)$ :
		\begin{align*}
			\vp(\iota) = \begin{pmatrix}
				-I_3 & 0 & 0\\
				0 & 1 & 0\\
				0 & 0 & -1
			\end{pmatrix}.
		\end{align*}
	\end{itemize}
	Given an arbitrary $\xi\in \Conf(\s^3)$, we will denote $M_\xi := \vp(\xi)$ in the rest of the paper.
	
	\subsection{Conformal Gauss map}\label{section_CGM}
	
	An introduction to conformal Gauss map can be found in \cite{eschenburg,hertrich-jeromin2003,marque2021}. Here we sketch the construction and recall the main properties.\\
	
	\textbf{Construction of the conformal Gauss map :}\\
	Consider an immersion $\Psi : \Sigma \to \s^3$ with mean curvature $H$, let $x\in\Sigma$ and consider the 2-sphere $S_x\subset \s^3$ which is tangent to $\Psi(\Sigma)$ at the point $\Psi(x)$ with mean curvature $H(x)$. Let $\hat{S}_x\subset \R^4$ be the 3-sphere that intersect $\s^3$ orthogonally along $S_x$. If $H(x)=0$, then $\hat{S}_x$ is a flat 3-plane. If $H(x)\neq 0$, the center of $\hat{S}_x$ is given by $c_x := \Psi(x) + \frac{N(x)}{H(x)}$, with radius $\frac{1}{|H(x)|}$. In order to have a description independant of the vanishing of the mean curvature, we can consider the homogeneous vector
	\begin{align}\label{homogeneous_Y}
		\left[ \begin{pmatrix}
			\Psi(x) + \frac{N(x)}{H(x)}\\ 1
		\end{pmatrix} \right] \in \RP^4 .
	\end{align}
	If $\eta = (dx^1)^2 + \cdots + (dx^4)^2 - (dx^5)^2$ is the Minkowski metric on $\R^5$ and $|\cdot|_\eta$ is the associated norm, we note that
	\begin{align*}
		\left| \begin{pmatrix}
			\Psi(x) - \frac{N(x)}{H(x)}\\ 1
		\end{pmatrix} \right|^2_\eta = \frac{1}{H(x)^2} >0.
	\end{align*}
	We can normalize the homogeneous vector \eqref{homogeneous_Y} in order to have a constant norm equal to $1$. Therefore, we obtain a representation of $S_x$ with the quantity 
	\begin{align}\label{definition_Y}
		Y(x) := \begin{pmatrix}
			H(x)\Psi(x) + N(x)\\ H(x)
		\end{pmatrix}.
	\end{align}
	This definition is independant of the vanishing of $H$. By construction, $Y(x)$ belongs to the De Sitter space
	\begin{align*}
		\s^{3,1} &:= \{ p\in \R^5 : |p|^2_\eta = 1 \}.
	\end{align*}
	We note $\R^{4,1}$ the space $(\R^5,\eta)$. \underline{We call $Y$ the conformal Gauss map} because it is a conformally invariant version of the Gauss map, in the sense that for any $\Xi \in \Conf(\s^3)$, the conformal Gauss map $Y_{\Xi\circ\Psi}$ is given by the following relation :
	\begin{align*}
		Y_{\Xi\circ \Psi} = M_{\Xi} Y.
	\end{align*}
	
	\textbf{Relation with the Willmore functional :}\\
	Let $g=\Psi^*\xi$. The derivatives of $Y$ are given by
	\begin{align*}
		\g Y &= \g H \begin{pmatrix}
			\Psi\\ 1
		\end{pmatrix} + \begin{pmatrix}
			H\g \Psi + \g N\\ 0
		\end{pmatrix}.
	\end{align*}
	So that
	\begin{align}\label{derivatives_Y}
		\g Y &= \g H \begin{pmatrix}
			\Psi \\ 1
		\end{pmatrix} - \begin{pmatrix}
			\Arond \g \Psi\\ 0
		\end{pmatrix} = (\g H)\nu - \Arond(\g \nu).
	\end{align}
	where $\nu = \begin{pmatrix}
		\Psi \\ 1
	\end{pmatrix}$ and $\Arond \g \Psi$ stand for $\Arond_i^{\ j} \dr_j \Psi$, where the index is raised by the metric $g$. The map $Y:(\Sigma,g) \to \s^{3,1}$ is conformal : $\scal{\dr_i Y_k}{\dr_j Y_k}_\eta = \frac{1}{2}\left|\Arond_\Psi\right|^2_g g_{ij}$, so that
	\begin{align*}
		\Ar(Y) := \int_\Sigma d\vol_{Y^*\eta} = \frac{1}{2}\Er(\Psi).
	\end{align*}
	Since $Y$ is conformal, the Hopf differential $\scal{\dr Y}{\dr Y}_\eta (dz)^2$ vanishes. In particular, if $Y$ is defined on an annuli $[0,T]\times \s^1$, then for any $t\in[0,T]$,
	\begin{align}\label{residue_hopf_diff}
		\int_{\{t\}\times\s^1} |\dr_t Y|^2_\eta &= \int_{\{t\}\times\s^1} |\dr_\theta Y|^2_\eta .
	\end{align}
	Another consequence is the relation
	\begin{align*}
		\Dr(Y) := \frac{1}{2} \int_\Sigma |\g Y|^2_\eta d\vol_g = \Ar(Y).
	\end{align*}
	Moreover, $\Psi$ is Willmore if and only if $Y$ is harmonic : $\lap_g Y = -|\g Y|_\eta^2 Y$. In this case, the following quartic is holomorphic :
	\begin{align*}
		\scal{\dr^2 Y}{\dr^2 Y}_\eta (dz)^4 = -\scal{\dr Y}{\dr^3 Y} (dz)^4.
	\end{align*}
	In particular, if $Y$ is not constant, the zeros of $\g Y$ are isolated. We will consider the euclidean and the Minkowski norms on $Y$. To distinguish them in $L^p$ norms, we will denote $L^p_\xi$ the norms
	\begin{align*}
		\| Y\|_{L^p_\xi} = \big\| |Y|_\xi \big\|_{L^p}, & & \|\g Y \|_{L^p_\xi} = \big\| |\g Y|_\xi \big\|_{L^p}.
	\end{align*}
	We denote by $L^p_\eta$ the norms
	\begin{align*}
		\| Y\|_{L^p_\eta} = \big\| |Y|_\eta \big\|_{L^p}, & & \|\g Y \|_{L^p_\eta} = \big\| |\g Y|_\eta \big\|_{L^p}.
	\end{align*}

	\textbf{Variations of a conformal Gauss map :} \\
	For any spacelike map $Z:\Sigma\to \s^{3,1}$, we define the area and the dirichlet energy in $\s^{3,1}$ by
	\begin{align*}
		\Ar(Z) := \int_\Sigma d\vol_{Z^*\eta} = \int_\Sigma \sqrt{\det\scal{\dr_i Z}{\dr_j Z}_\eta } dx, &  & \Dr(Z) := \frac{1}{2}\int_\Sigma |\g Z|^2_\eta.
	\end{align*}
	We will use restictions of these functionals. Given a measurable set $\Omega\subset \Sigma$, we let 
	\begin{align*}
		\Ar(Z;\Omega) := \int_\Omega d\vol_{Z^*\eta}, &  & \Dr(Z;\Omega) := \frac{1}{2}\int_\Omega |\g Z|^2_\eta.
	\end{align*}
	We denote $\delta^2 \Ar_Z$ and $\delta^2 \Dr_Z$ the second derivatives of the area and the Dirichlet functional : if $(Z_t)_{t\in(-1,1)}$ is a variation of $Z$, then
	\begin{align*}
		\delta^2 \Ar_Z \left[ (\dr_t Z_t)_{|t=0} \right] := \frac{d^2}{dt^2} \Ar(Z_t)_{|t=0}, &  & \delta^2 \Dr_Z \left[ (\dr_t Z_t)_{|t=0} \right] := \frac{d^2}{dt^2} \Dr(Z_t)_{|t=0}.
	\end{align*}
	
	Since the formulas are easier in $\s^3$, we will define the index for Willmore immersion in $\s^3$. To compute the index, we will need to understand the normal space of $Y(\Sigma)\subset \s^{3,1}$. According to \cite{schatzle2017}, the umbilic points of a Willmore immersion $\Psi : \Sigma \to \s^3$ which is not totally umbilic, is composed of isolated points and closed real analytic curves. We focus on points where $|\g Y|^2_\eta = \left| \Arond_\Psi \right|^2_{g_\Psi} \neq 0$. Since $Y$ is conformal, the family of orthogonal vectors $(Y,\dr_1 Y,\dr_2 Y)$ span a 3-dimensional space in $\R^{4,1}$. Since they are all spacelike, we can add to this family two vectors $\nu,\nu^*$ in the normal space of $Y(\Sigma)\subset \s^{3,1}$, satisfying $|\nu|^2_\eta =|\nu^*|^2_\eta = 0$, $\scal{\nu}{\nu^*}_\eta = -1$. Up to exchange $\nu$ and $\nu^*$, we obtain a direct basis $(Y,\dr_1 Y,\dr_2 Y,\nu,\nu^*)$ of $\R^{4,1}$. Up to the normalization $\nu_5 = 1$, the vector $\nu$ can be choosen of the form $\begin{pmatrix}
		\Psi\\ 1
	\end{pmatrix}$. Therefore, under a non degeneracy assumption, we can recover $\Psi$ from the knowledge of $Y$. This is a key point that we will develop to understand the index of a Willmore immersion. 
	We define the area index of $Y$ by
	\begin{align}\label{definition_indexY}
		\Ind_\Ar(Y) &:= \dim  \Span\left\{ Z : \Sigma\to T_Y \s^{3,1} \Big| \begin{array}{l}
			\scal{\nu}{\lap_{g_\Psi} Z - |\g Y|^2_\eta Z }_\eta = 0\\
			\scal{\nu }{\g Z}_\eta = 0\\
			\delta^2 \Ar_Y(Z) <0
		\end{array} \right\} .
	\end{align}
	
	\subsection{Degenerating surfaces}
	
	We recall some aspects of the Deligne-Mumford's description of the loss of compactness of the conformal class for a sequence of Riemann surfaces with fixed topology, see for instance \cite{hummel1997}. This presentation is taken from the section 1.4 of \cite{laurain2018a}.\\
	
	Let $(\Sigma,c_k)$ be a sequence of closed Riemann surfaces of fixed genus $g$ and varying conformal class $c_k$. If $g=0$, there is only one conformal class on the sphere. If $g=1$, then $(\Sigma,c_k)$ is conformally equivalent to $\R^2 / \left( \frac{1}{\sqrt{\Im(v_k)}} \Z \times \frac{v_k}{\sqrt{\Im(v_k)}} \Z \right)$ where $v_k$ lies in the fundamental domain $\{ z\in \C : |\Re (z)|\leq \frac{1}{2}, |z|\geq 1 \}$ of $\Hd^2/PSL_2(\Z)$. We say that the conformal classes $c_k$ degenerate if $|v_k|\xrightarrow[k\to \infty]{}{+\infty}$. \\
	If $g\geq 2$,  let $h_k$ be the hyperbolic metric associated with $c_k$. We say that $(\Sigma,c_k)$ degenerates if there exists closed geodesics whose length goes to zero. In that case, up to a subsequence, there exists
	\begin{itemize}
		\item an integer $N\in \inter{1}{3g-3}$,
		\item a sequence $\Lr_k = (\Gamma^i_k)_{1\leq i\leq N}$ of finitely many pairwise disjoint closed geodesics of $(\Sigma,h_k)$ with length converging to zero,
		\item a closed Riemann surfaces $(\bar{\Sigma},\bar{c})$,
		\item a complete hyperbolic surfaces $(\tilde{\Sigma},\tilde{h})$ with $2N$ cusps $\{(q^i_1,q^i_2) \}_{1\leq i\leq N}$ such that $\tilde{\Sigma}$ is obtained topologically by removing the geodesics of $\Lr_k$ to $\Sigma$ and after closing each component of $\dr(\Sigma\setminus \Lr_k)$ by adding a puncture $q^i_l$ at each of these component. Moreover, $\bar{\Sigma}$ is topologically equal to $\tilde{\Sigma}$ and the complex structure defined by $\tilde{h}$ on $\tilde{\Sigma}\setminus \{ q^i_l \}$ extends uniquely to $\bar{c}$. We can also equipped $\bar{\Sigma}$ with a metric $\bar{h}$ with constant curvature, but not necessarily hyperbolic since the genus of $\bar{\Sigma}$ can be lower than to one of $\Sigma$.
	\end{itemize}
	$(\tilde{\Sigma},\tilde{h})$ is called the nodal surface of the converging sequence and $(\bar{\Sigma},\bar{c})$ is its renormalization. These objects are related, in the sense that, there exists a diffeomorphism $\psi_k : \tilde{\Sigma}\setminus \{q^i_l\} \to \Sigma \setminus \Lr_k$ such that $\tilde{h}_k := \psi_k ^* h_k$ converges in $C^\infty_{loc}$ to $\tilde{h}$.

	\section{Sequence of Willmore immersions with bounded energy}\label{section_bounded_sequence}
	
	\subsection{Bubble compactness}\label{bubble_compactness}
	
	We give a precise definition of bubble convergence.
	\begin{defi}\label{def_bubble_conv}
		Let $(h_k)_{k\in\N}$ be a sequence of metrics on a given closed surface $\Sigma$, with constant curvature and fixed volume if needed. We assume that $(\Sigma,h_k)$ converges to a nodal surface $(\tilde{\Sigma},\tilde{h})$. Let $(\tilde{\Sigma}^l)_{l\in\inter{1}{q}}$ be the connected components of $\tilde{\Sigma}$. We say that a sequence $(\Psi_k)_{k\in\N}$ of conformal Willmore immersions $(\Sigma,h_k) \to\s^3$ satisfy a weak bubble convergence if up to a subsequence, $(\Psi_k)_{k\in\N}$ satisfy the following : there exists
		\begin{itemize}
			\item a number $I\in\N$ and converging sequences $(x^i_k)_{k\in\N,i\in\inter{1}{I}}\subset\Sigma$, let $S = \left\{ \displaystyle{\lim_{k\to\infty}} x^i_k : i\in\inter{1}{I} \right\}$;
			\item for each $i\in\inter{1}{I}$, there exists a sequence $(r^i_k)_{k\in\N}\subset \R^*_+$ converging to $0$;
			\item conformal maps $\Xi^l_k,\xi^i_k\in\Conf(\s^3)$, for any $l\in\inter{1}{q}$, $i\in\inter{1}{I}$, $k\in\N$;
			\item diffeomorphisms $f_k^l:\Sigma \to \Sigma$, for any $l\in\inter{1}{q}$, $k\in\N$;
			\item a finite number of bubbles : smooth Willmore immersions $\omega^i : \R^2 \to \s^3$, for any $i\in\inter{1}{I}$, with possible ends and branch points;
		\end{itemize}
		such that
		\begin{itemize}
			\item for each $l\in\inter{1}{q}$, $(\Xi^l_k\circ\Psi_k\circ f_k^l)_{k\in\N}$ converges in the $C^m_{loc}(\tilde{\Sigma}^l \setminus S)$ topology, for any $m \in \N$, to a smooth possibly branched immersion $\Psi_\infty^l :\tilde{\Sigma}^l\to \s^3$, which is Willmore away from $S$;
			\item For each $i\in\inter{1}{I}$, in some charts around $\displaystyle{\lim_{k\to\infty}} x^i_k$, the sequence $\Big(\xi^i_k\circ \Psi_k ( x^i_k + r^i_k \cdot) \Big)_{k\in\N}$ converges to $\omega^i$ in the $C^m_{loc}$-topology, for any $m\in\N$, on $\R^2$ minus a finite set.
		\end{itemize}
		We say that the convergence is strong if there is no loss of energy :
		\begin{align*}
			\lim_{k\to \infty} \Er(\Psi_k) &= \sum_{l=1}^q \Er(\Psi_\infty^l) + \sum_{i=1}^{I} \Er(\omega^i).
		\end{align*}
	\end{defi}
	
	Consider $(\Psi_k)_{k\in\N}$ a sequence of Willmore conformal immersion $\Sigma\to \s^3$ such that
	\begin{align}\label{hypothesis_bound_E}
		\sup_{k\in\N} \Er(\Psi_k) <\infty.
	\end{align}
	Let $g_k := \Psi_k^*(\xi_{|\s^3})$. \\
	
	If $\Sigma = \s^2$, there is only one conformal class. Thanks to \cite{bernard2014}, there is a strong bubble convergence.\\
	If the genus of $\Sigma$ is greater or equal to $1$, then we can decompose $(\Sigma,g_k)$ into thin and thick parts. On the thick parts, the metric converges smoothly and we can again apply \cite{bernard2014} to obtain a bubble tree. The thin parts are conformally equivalent to long cylinders and will be constructed in the next section. \\
	
	With the notations of the definition \ref{def_bubble_conv}, consider some bubbles $(x^1_k)_{k\in\N},...,(x^I_k)_{k\in\N}$ forming at the same point in $\Sigma$. At the scale of each $x^i_k$, a blow up of $(\Psi_k)_{k\in\N}$, with speed $r^i_k$, locally converges. We want to understand what happen far from $x^i_k$. For $i\neq j$, the speeds $r^i_k$ and $r^j_k$ satisfy either $r^i_k \ust{k\to\infty}{=} o (r^j_k)$ or $r^j_k \ust{k\to\infty}{=} o (r^i_k)$. We can assume that $r^1_k< \cdots < r^I_k$. We are interested in the behaviour of $(\Psi_k)_{k\in\N}$ in the regions $B(x^{i+1}_k,r^{i+1}_k) \setminus B(x^i_k,r^i_k)$ and in the region $B(0,1)\setminus B(x^I_k,r^I_k)$. These domains are \textbf{neck regions}. They are all described as a degenerating annuli $B(0,R_k)\setminus B(0,r_k)$ with $r_k \ust{k\to \infty}{=} o(R_k)$, where no bubble is forming inside and $(\Psi_k)_{k\in\N}$ converges, in the sense of the definition \ref{def_bubble_conv}, on the boundary. The hypothesis that there is no bubble can be translated in the estimate :
	\begin{align*}
		\lim_{k\to \infty} \sup_{\rho\in[r_k,R_k/2]} \int_{B(0,2\rho)\setminus B(0,\rho)} |\g N_k|^2 dx = 0.
	\end{align*}
	
	\subsection{Choice of the thin parts}

	\textbf{In the torus case :}\\
	
	A torus is isometric to a cylinder $C_\ell := \frac{1}{\sqrt{2\pi \ell}} \left( \s^1 \times [0,\ell] \right)$, for some $\ell>0$, with the identification $(\theta,0) \sim (\theta,\ell)$. \\
	Then, $C_\ell$ admits the following chart :
	\begin{align*}
		\psi_\ell : \left| \begin{array}{c c c}
			B(0,1) \setminus B(0,e^{-\ell}) & \to & C_\ell \\
			(\theta,r) & \mapsto & \left( \frac{\cos\theta}{\sqrt{2\pi \ell}} , \frac{\sin\theta}{\sqrt{2\pi \ell}} , \frac{-\log r}{\sqrt{2\pi \ell}} \right)
		\end{array} \right. .
	\end{align*}
	By the theorem 0.2 of \cite{laurain2018} assures that the conformal factor $\lambda_\Psi$ of an immersion $\Psi : C_\ell \to \s^3$ satisfies :
	\begin{align*}
		\| \g \lambda_\Psi \|_{L^{2,\infty}} &\leq C\| \g N_\Psi \|_{L^2}^2.
	\end{align*}
	If $(\Psi_k)_k$ is a sequence of immersions from $C_{\ell_k} \to \R^3$ satisfying (\ref{hypothesis_bound_E}), there is a finite number $t^1,...,t^{I_k} \in [-\frac{\ell_k}{2} , \frac{\ell_k}{2} ]$ such that
	\begin{align*}
		\liminf_{k\to \infty} \int_{\s^1 \times [t^i,t^i+1]} |\g N_k|^2 dt d\theta \geq \frac{\ve_0}{2},
	\end{align*}
	where $\ve_0$ is the threshold of energy needed to have $\ve$-regularity of \cite{riviere2008}, theorem 1.5. If $(g_k)_k$ degenerates in the moduli space, then $\ell_k \to +\infty$, and we can find $t_k$ such that for any $i\in\inter{1}{I_k}$, $|t_k - t^i|\xrightarrow[k\to \infty]{}{+\infty}$. Thanks to the $\ve$-regularity, the sequence $\Psi_k(\cdot + t_k)$ will converge to a possibly trivial bubble. Hence, if we cut the torus at $t_k$, we obtain a cylinder with no concentration point near the boundary.\\
	
	\textbf{In the hyperbolic case :} \\
	
	Thanks to the collar lemma, we know that choosing $\delta<\sinh(1)$, the thin part $\{ x\in (\Sigma,h_k) : \text{inj}(x)<\delta \}$ consists in a finite number of collars. Up to extraction, this number is fixed for $k$ large enough. Here we restrict ourselves to one collar. This collar contains a smallest closed geodesic and is conformal to an hyperbolic cylinder of the form
	\begin{align*}
		A_\ell := \left\{ z= re^{i\theta} \in \Hd : r\in [1,e^\ell], \arctan\left[ \sinh\frac{\ell}{2} \right] < \theta < \pi - \arctan\left[ \sinh \frac{\ell}{2} \right] \right\},
	\end{align*}
	with the identification $e^{i\theta} \sim e^{\ell + i\theta}$. The geodesic correspond the to line $\{\theta = \frac{\pi}{2} \}$. \textbf{This is a collar region}. We will rather use the parametrization :
	\begin{align*}
		P_\ell := \left\{ (t,\theta) : \theta \in [0,2\pi], \frac{2\pi}{\ell} \arctan\left[ \sinh \frac{\ell}{2} \right] < t < \frac{2\pi}{\ell} \left( \pi - \arctan\left[ \sinh\frac{\ell}{2} \right] \right) \right\},
	\end{align*}
	with the identification $(t,0)\sim (t,2\pi)$. In this parametrization, the constant curvature metric is
	\begin{align*}
		ds^2 = \left( \frac{\ell}{2\pi \sinh\left(\frac{\ell t}{2\pi} \right)} \right)^2 (dt^2 + d\theta^2 ).
	\end{align*}
	The geodesic corresponds to the line $\{t=\frac{\pi^2}{\ell} \}$. As the length of the geodesic goes to $0$, the cylinder becomes $[0,T_k]\times \s^1$ with $T_k \xrightarrow[k\to \infty]{}{+\infty}$.

	\subsection{A priori estimates on necks and collars}
	
	In both cases, necks and collars are conformally equivalent to long cylinders or degenerating annuli. Since we work with conformally invariant problems, we can work on both representation, depending on the context.\\
	
	In the section \ref{gauge_neck_region}, we will work with the description by degenerating annuli $\Cr_k = B_{R_k}(0) \setminus B_{r_k}(0)$ with $R_k =o(1)$ and $r_k = o(R_k)$, necks and collars enjoy the following two estimates :
	\begin{align}\label{apriori_necks}
		\sup_{k\in\N} \int_{\Cr_k} |\g N_k|^2 dx \leq \ve_0, &  & \limsup_{k\to \infty} \sup_{\rho \in [r_k, \frac{R_k}{2}]} \int_{B_{2\rho} \setminus B_\rho} |\g N_k|^2 dx = 0,
	\end{align}
	where $\ve_0$ is the amount of energy needed to obtain the $\ve$-regularity of Rivière \cite{riviere2008}. In the sections \ref{section_convergence_geodesic} and \ref{proof_quantization}, we will work with the description by long cylinder $\Cr_k = [-T_k,T_k]\times \s^1$ with $T_k \xrightarrow[k\to\infty]{}{+\infty}$. The estimate \eqref{apriori_necks} translate into
	\begin{align}\label{apriori_necks_cylinder}
		\sup_{k\in\N} \int_{\Cr_k} |\g N_k|^2 dx \leq \ve_0, &  & \limsup_{k\to \infty} \sup_{t\in[-T_k+1,T_k-1]} \int_{[t-1,t+1]\times \s^1} |\g N_k|^2 dx = 0.
	\end{align}

	\section{Main results}\label{section_main_results}
	
	\subsection{Statements}

	We prove the following bubble convergence result for the conformal Gauss map :
	
	\begin{theorem}\label{bubbling_CGM}
		Let $\Sigma$ be a closed Riemann surface and consider $\Phi_k :\Sigma\to \R^3$ a sequence of smooth immersions. Assume that on the whole surface $\Sigma$, it holds
		\begin{align*}
			\sup_{k\in \N} \int_\Sigma |\g \vec{n}_k|^2_{\Phi_k^*\xi} d\vol_{\Phi_k^*\xi} <\infty.
		\end{align*}
		There exists $\ve_0>0$ satisfying the following. Consider a cylinder $\Cr_k := [-T_k,T_k]\times \s^1$ with $T_k \xrightarrow[k\to \infty]{}{+\infty}$, associated to the flat metric. Assume that there exists an embedding $\Cr_k \hookrightarrow \Sigma$ and that each $\Phi_k : \Cr_k \to \R^3$ is a smooth conformal Willmore immersion. Let $\vec{n}_k$ be the Gauss map of $\Phi_k$. Assume that in $\Cr_k$ :
		\begin{align*}
			\sup_{k\in\N} \int_{\Cr_k} |\g \vec{n}_k|^2 <\ve_0, &  & \lim_{k\to +\infty} \sup_{t\in[-T_k+1,T_k-1]} \int_{[t-1,t+1]\times \s^1} |\g \vec{n}_k|^2 dx = 0.
		\end{align*}
		Let $t_k \in [-T_k,T_k]$ with $|t_k-T_k| \xrightarrow[k\to \infty]{}{+\infty}$ and $|t_k+T_k| \xrightarrow[k\to \infty]{}{+\infty}$. Let $Y_k$ be the conformal Gauss map of $\Phi_k$. There exists $\xi_k\in \Conf(\R^3)$ such that $\xi_k\circ \Phi_k$ has the same topology as $\Phi_k$ and $M_k:= M_{\xi_k} \in SO(4,1)$ satisfies
		\begin{align*}
			\sup_{k\in\N,\theta\in\s^1} |M_k Y_k(t_k,\theta)|_\xi <\infty.
		\end{align*}
		Furthermore for any $\delta>0$, if there exists $\kappa_k>0$ satsfying
		\begin{align*}
			\forall k\in\N,\ \ \ \ \sup_{x,y \in [t_k-1,t_k+\kappa_k]\times\s^1} |M_k Y_k(x) - M_k Y_k(y)|_\xi = \delta,
		\end{align*}
		then there exists a reparametrization $\alpha_k :(-s_k,s_k)\times \s^1 \to (t_k-1,t_k+\kappa_k)\times \s^1$ such that $(M_k Y_k\circ \alpha_k)_{k\in\N}$ converges in the $C^2_{loc}$-topology to a time-like straight line of euclidean oscillation $\delta$.
	\end{theorem}
	
	The following shows that when computing the index of $\Phi$, we can compute only the second variations of $Y$, with variations being conformal Gauss maps.
	
	\begin{proposition}\label{equality_index}
		Let $\Phi : \Sigma \to \R^3$ be a smooth immersion and $Y$ be its conformal Gauss map. Assume that the umbilic points of $\Phi$ are nowhere dense. Consider a vector field $Z : \Sigma \to T_Y \s^{3,1}$. There exists a smooth variation $(\Phi_t)_{t\in(-1,1)}$ of $\Phi$ such that $(\dr_t \Phi_t)_{|t=0} \perp \g \Phi$ and its conformal Gauss map $(Y_t)_{t\in(-1,1)}$ satisfies $(\dr_t Y_t)_{|t=0} = Z$ if and only if $Z$ satisfies 
		\begin{align}\label{system_condition_LCGM}
			\left\{ \begin{array}{l}
				\scal{\nu}{\lap_{g_\Phi}Z - |\g Y|^2_\eta Z }_\eta = 0, \\
				\scal{\nu}{\g Z}_\eta = 0,
			\end{array}
			\right.
		\end{align}
		where $\nu = \begin{pmatrix}
			\Phi\\ \frac{|\Phi|^2-1}{2} \\ \frac{|\Phi|^2+1}{2}
		\end{pmatrix}$. In particular, if $\Phi$ is a smooth Willmore surfaces not totally umbilic, then $\Ind_\Er(\Phi) = \Ind_\Ar(Y)$.
	\end{proposition}
	
	If $\pi$ is the stereographic projection of $\s^3\setminus \{\mathrm{north pole}\}$ onto $\R^3$, and $\Psi := \pi^{-1}\circ\Phi$, we recall that 
	\begin{align*}
		\begin{pmatrix}
			\Phi\\ \frac{|\Phi|^2-1}{2} \\ \frac{|\Phi|^2+1}{2}
		\end{pmatrix} &= \frac{|\Phi|^2+1}{2} \begin{pmatrix} \Psi\\ 1 \end{pmatrix}.
	\end{align*}
	Hence, the conditions \eqref{system_condition_LCGM} are conformally invariant. We obtain a quantization result for a sequence of bounded energy and index :
	
	\begin{theorem}\label{quantization_bounded_index}
		Let $(\Sigma,h_k)$ be a sequence of closed surfaces with fixed genus, constant curvature and normalized volume if needed. We assume that this sequence converges to a nodal surfaces $(\tilde{\Sigma},\tilde{h})$. Let $(\tilde{\Sigma}^l)_{l\in\inter{1}{q}}$ be the connected components of $\tilde{\Sigma}$. Let $(\Phi_k)_{k\in\N}$ a sequence of conformal Willmore immersion $(\Sigma,h_k) \to \R^3$ satisfying
		\begin{align*}
			\sup_{k\in\N}\Big( \Er(\Phi_k) + \Ind_\Er (\Phi_k) \Big) <\infty.
		\end{align*}
		Then, there exists $q$ branched immersions $\Phi^l_\infty : \tilde{\Sigma}^l \to \R^3$ and a finite number of possibly branched immersions $\omega_j :\s^2\to \R^3$ which are Willmore away from possibly finitely many points, and such that, up to a subsequence,
		\begin{align*}
			\lim_{k\to \infty} \Er(\Phi_k) &= \sum_{l=1}^q \Er(\tilde{\Phi}^l_\infty) + \sum_{j=1}^p \Er(\omega_j).
		\end{align*}
	\end{theorem}

	\subsection{Sketch of the proofs}
	
	\textbf{Proof of the theorem \ref{bubbling_CGM} :}\\
	In order to show the convergence to a curve, we prove that the angular derivative always vanishes very fast, see the lemma \ref{step1_conv_geodesic}, then we prove that the distance between $Y$ and its average $Y^*(r) = \fint_{\dr B_r(0)} Y$ has the same size as the angular derivative, see the lemma \ref{estimate_xistar}. With these properties in hand, the limit object must be a curve. Since we can pass to the limit in the equation of harmonic maps, the limit curve must be a geodesic. Using \eqref{residue_hopf_diff}, the limit curve must be light-like.\\
	The first issue, is that the system satisfied by $Y_k$ is highly different from the one of harmonic maps with values into a compact Riemannian manifold, since it is of the form $|\lap Y_k|_\xi \leq|\g Y_k|^2_\xi |Y_k|_\xi$, without $L^\infty$ bounds a priori since $\s^{3,1}$ is not compact.  To solve this issue, we find a good gauge, that is a matrix in $SO(4,1)$ such that at each scale we have an $\ve$-regularity estimate for the conformal Gauss map, see the section \ref{gauge_neck_region}. Then we  work only on a region where the oscillations of $Y_k$ are bounded. \\

	\textbf{Proof of the proposition \ref{equality_index} :}\\
	If we consider an arbitrary space-like variation $(Y_t)_t$ of $Y$, there is a priori no reason to know whether each $Y_t$ is a conformal Gauss map or not. However, each $Y_t$ is close to $Y$. Therefore, its normal space in $\s^{3,1}$ is close to the one of $Y$ and we can recover an immersion $\Psi_t : \Sigma \to \s^3$ by looking at isotropic vectors of its normal space. An other important property of conformal Gauss maps, is that its mean curvature must always vanish in some isotropic direction. Therefore, any vector field $Z$ along $Y$ coming from a variation through conformal Gauss maps must satisfy the linearized equation corresponding to this restriction. Palmer \cite{palmer1991} showed that away from the umbilic points of $\Psi$, satisfying this linearized equation is a necessary and sufficient condition. Here we have a less geometric point of view and we obtain two equations equivalent to the condition of coming from a variation through conformal Gauss maps. \\

	\textbf{Proof of the theorem \ref{quantization_bounded_index} :}\\
	In a Riemannian manifold with positive curvature, if a geodesic is too long, then there exists a Jacobi field on this long section of the geodesic. So in the case of harmonic maps with values in a positively curved Riemannian manifold, once we have a Jacobi field on the limit geodesic, we can project it before the limit to obtain a Jacobi field along the harmonic maps. Therefore, the number of the Jacobi field along the limit is controlled by the indexes of the harmonic maps. In our case, the limit has no lorentz length so we cannot construct Jacobi fields directly on the limit (the limit geodesic is minimizing in a certain sense). So we build them before the limit. Since the limit is a straight line, we can consider a unit space-like constant vector which is always in the normal space of the limit. We add a perturbation to obtain a sequence of vector fields along our sequence of conformal Gauss maps which come from a variation through conformal Gauss maps. We show that this perturbation is small thanks to the theorem \ref{bubbling_CGM} and that before the limit, these vector fields are Jacobi fields.

	\section{Proof of the theorem \ref{bubbling_CGM}}\label{proof_theoremA}
	
	\subsection{Gauge in necks and collars}\label{gauge_neck_region}
	
	\textbf{Only in this section}, we work with Willmore immersions $\Phi_k : \Cr_k \to \R^3$, where $\Cr_k = B_{1/R}\setminus B_{R r_k}$ for some $R>1$, with $r_k \xrightarrow[k\to \infty]{}{0}$. We assume that $\Cr_k$ is a part of a closed surface $\Sigma$ and that $\Phi_k$ is defined on the whole $\Sigma$, but not necessarily Willmore on $\Sigma\setminus \Cr_k$. We denote $g_k := \Phi_k^*\xi$, $\vec{n}_k$ their Gauss maps, $H_k$ their mean curvature and $\Arond_k$ the traceless part of their second fundamental form. The conformal Gauss maps $Y_k$ are given by
	\begin{align}\label{def_Y_R3}
		Y_k := H_k \begin{pmatrix}
			\Phi_k\\
			\frac{|\Phi_k|^2-1}{2} \\
			\frac{|\Phi_k|^2+1}{2}
		\end{pmatrix} + \begin{pmatrix}
			\vec{n}_k\\ \scal{\vec{n}_k}{\Phi_k} \\ \scal{\vec{n}_k}{\Phi_k}
		\end{pmatrix}, &  & \g Y_k := (\g H_k) \begin{pmatrix}
			\Phi_k\\
			\frac{|\Phi_k|^2-1}{2} \\
			\frac{|\Phi_k|^2+1}{2}
		\end{pmatrix} - e^{-2\lambda_k} \Arond_k \begin{pmatrix}
			\g \Phi_k\\ \scal{\g \Phi_k}{\Phi_k} \\ \scal{\g \Phi_k}{\Phi_k}
		\end{pmatrix},
	\end{align}
	where $\lambda_k$ is the conformal factor of $\Phi_k$. We can recover the mean curvature by the formula $H_k = (Y_k)_5 - (Y_k)_4$.\\
	The assumptions \eqref{apriori_necks} can be written :
	\begin{align}\label{apriori_necks_R3}
		\sup_{k\in\N} \int_{\Cr_k} |d\vec{n}|^2_{g_k} d\vol_{g_k} <\ve_0, & & \lim_{R\to \infty} \lim_{k\to \infty} \sup_{\rho\in[R r_k,1/(2R)]} \int_{B_{2\rho}\setminus B_\rho} |d\vec{n}_k|^2_{g_k} d\vol_{g_k}  =0.
	\end{align}
	Furthermore, we assume that on the whole surface $\Sigma$, it holds
	Assume that on the whole surface $\Sigma$, it holds
	\begin{align}\label{hyp:uniform_L2_sff}
		\sup_{k\in \N} \int_\Sigma |d\vec{n}_k|^2_{g_k} d\vol_{g_k} <\infty.
	\end{align}
	
\begin{remark}
The hypothesis that $\Phi_k$ is defined on a closed surface is purely technical, in order to use the lemma 5.3 in \cite{riviere2013a}. This lemma will be used only once to prove the estimate \eqref{eq:Linfty_phi}. We will work mainly on the annuli $\Cr_k$, so all the other estimates come from \eqref{apriori_necks_R3}.
\end{remark}

	For each $\rho \in[Rr_k, 1/(2R)]$, we consider the annuli
	\begin{align}\label{def_annuli}
		A_{\rho} := B_{2\rho} \setminus B_{\rho} \subset N_k,\ \ \ \ \ \ \hat{A}_{\rho} := B_{\frac{19}{10}\rho}\setminus B_{\frac{11}{10}\rho},\ \ \ \ \ \tilde{A}_\rho := B_{\frac{18}{10}\rho} \setminus B_{\frac{12}{10}\rho},\ \ \ \ \bar{A}_\rho := B_{\frac{17}{10}\rho}\setminus B_{\frac{13}{10}\rho}.
	\end{align}
	The goal of this section, see the \cref{existence_gauge}, is to prove that we can restrict ourselves to the following situation :
	\begin{itemize}
		\item  For every $\rho_k \in[Rr_k, \frac{1}{R}]$, there exists $M_{k,\rho_k} \in SO(4,1)$ such that
		\begin{align*}
			\rho_k\| M_{k,\rho_k} \g Y_k \|_{L^\infty_\xi\left( \bar{A}_{\rho_k} \right) } & \leq C\|\g Y_k \|_{L^2_\eta \left( A_{\rho_k} \right)},  \\
			\|M_{k,\rho_k} Y_k\|_{L^\infty_\xi (\bar{A}_{\rho_k})} & \leq C,
		\end{align*}
		\item We have the uniform estimate :
		\begin{align*}
			\lim_{R\to \infty} \lim_{k\to \infty} \sup_{s_k \in [\rho_k/2, 1/R]} s_k\| M_{k,\rho_k} \g Y_k \|_{L^\infty_\xi\left( \bar{A}_{s_k} \right) } =0.
		\end{align*}
	\end{itemize}

	\begin{lemma}\label{gauge_dyadic_annuli}
		There exists $C_0>1$ such that for any $k\in\N$ and any dyadic annuli $A_{\rho_k} \subset \Cr_k$, there exists a conformal transformation $\Theta_{k,\rho_k} \in \Conf(\R^3)$ such that $\Psi_{k,\rho_k} := \Theta_{k,\rho_k}\circ \Phi_k$ has the same topology as $\Phi_k$ with the following estimates :
		\begin{align}\label{defect_conformal_factor}
			\frac{1}{C_0} \leq \left\| \frac{ |\g \Psi_{k,\rho_k}| }{ |\g \Phi_k| } \right\|_{L^\infty(A_{\rho_k})} \leq C_0,
		\end{align}
		And
		\begin{align}
			\fint_{A_{\rho_k}} H_{\Psi_{k,\rho_k}} &= 0.
		\end{align}
	\end{lemma}
	
	\begin{proof}
		We proceed as in \cite{bernard2020} for the existence of conformal transformations. We will use some results from \cite{bernard2020} that we recall in the \cref{technical_lemmas_gauge}. First, we rescale : if we consider the domain to be a subset of $\C$, then we let
		\begin{align*}
			& \tilde{\Cr}_k := \frac{1}{\rho_k} \Cr_k = B_{1/(\rho_k R)} \setminus B_{Rr_k/\rho_k},\\
			\forall z\in \tilde{\Cr}_k,\ \ \ \ &\phi_{k,\rho_k}(z) := \frac{1}{\rho_k} e^{-\overline{\lambda_k}} \left( \Phi_k(\rho_k z) - \Phi_k(\rho_k) \right),\\
			& \overline{\lambda_k} := \fint_{A_{\rho_k}} \lambda_k.
		\end{align*}
		So that, by the harnack estimate on the conformal factor, see lemma II.1 of \cite{bernard2019} :
		\begin{align}\label{normalization_phi}
			\phi_{k,\rho_k}(1)=0,\ \ \ \ \ \ C^{-1} < |\g \phi_{k,\rho_k}| < C \ \text{on }A_1.
		\end{align}
		If $\fint_{A_1} H_{\phi_{k,\rho_k}} =0$, then we take $\Theta_{k,\rho_k} = \mathrm{id}$. If not, we consider the inversion $\iota_a(x) := \frac{x-a}{|x-a|^2}$ around a point $a\in\R^3$. In $SO(4,1)$, we have
		\begin{align*}
			M:= M_{\iota_a} &= \begin{pmatrix}
				-id & 0 & 0 \\
				0 & 1 & 0 \\
				0 & 0 & -1
			\end{pmatrix} \begin{pmatrix}
				id & a & -a \\
				-a^T & 1- \frac{|a|^2}{2} & \frac{|a|^2}{2} \\
				-a^t & - \frac{|a|^2}{2} & 1+ \frac{|a|^2}{2}
			\end{pmatrix} \\
			&= \begin{pmatrix}
				-id & -a & a \\
				-a^T & 1- \frac{|a|^2}{2} & \frac{|a|^2}{2} \\
				a^T & \frac{|a|^2}{2} & -1-\frac{|a|^2}{2}
			\end{pmatrix}.
		\end{align*}
		Let $Y_{\phi_{k,\rho_k}} := (Y_{123}, Y_4,Y_5)$ be the conformal Gauss map of $\phi_{k,\rho_k}$. The conformal Gauss map of $\psi_{a,k,\rho_k} := \iota_a\circ \phi_{k,\rho_k}$ is given by
		\begin{align*}
			Y_{\psi_{a,k,\rho_k}} &= MY_{\phi_{k,\rho_k}} = \begin{pmatrix}
				-Y_{123} +(Y_5 - Y_4) a\\
				-\scal{a}{Y_{123}} + Y_4 + \frac{ |a|^2}{2}(Y_5 - Y_4) \\
				\scal{a}{Y_{123}} - Y_5 - \frac{|a|^2}{2} (Y_5- Y_4)
			\end{pmatrix}.
		\end{align*}
		Let $\bar{\cdot}$ denote the average on the annuli $A_1$. Thanks to \eqref{def_Y_R3}, the average of the mean curvature is given by
		\begin{align*}
			\overline{H_{\psi_{a,k,\rho_k}}} &= \overline{(Y_{\psi_{a,k,\rho_k}})_5} - \overline{(Y_{\psi_{a,k,\rho_k}})_4} \\
			&= 2\scal{a}{ \overline{ Y_{123} }} - \overline{Y_4} - \overline{Y_5} -|a|^2 \overline{H_{\phi_{k,\rho_k}}}.
		\end{align*}
		We consider an orthonormal basis of $\R^3$ given by $\left\{ \frac{ \overline{Y_{123}} }{| \overline{Y_{123}}|} , v_1, v_2 \right\}$, and decompose $a= x\frac{ \overline{Y_{123}} }{| \overline{Y_{123}}|} + yv_1 + zv_2$. Then,
		\begin{align*}
			\overline{H_{\psi_{a,k,\rho_k}}} &= 2x \left| \overline{ Y_{123} }\right| - \overline{Y_5} - \overline{Y_4}  - (x^2+ y^2 +z^2) \overline{ H_{\phi_{k,\rho_k}} } \\
			&= - \overline{H_{\phi_{k,\rho_k}}} \left[ \left( x - \frac{ \left| \overline{Y_{123}} \right| }{ \overline{H_{\phi_{k,\rho_k}}} } \right)^2 + y^2 + z^2 - \frac{ \left| \overline{Y_{123}} \right|^2 }{ \overline{H_{\phi_{k,\rho_k}}}^2 } + \frac{ \overline{Y_4} + \overline{Y_5} }{\overline{H_{\phi_{k,\rho_k}}} } \right] \\
			&= -\overline{H_{\phi_{k,\rho_k}}} \left[ \left( x - \frac{ \left| \overline{Y_{123}} \right| }{ \overline{H_{\phi_{k,\rho_k}}} } \right)^2 + y^2 + z^2 - \frac{ |\overline{Y_{\phi_{k,\rho_k}}} |^2_\eta }{ \overline{H_{\phi_{k,\rho_k}}}^2 } \right].
		\end{align*}
		Using the \cref{average_spacelike} and \eqref{apriori_necks_R3}, we have $\left|\overline{Y_{\phi_{k,\rho_k}}} \right|^2_\eta>\frac{1}{2}$ for $k$ large enough. So the solutions $a\in\R^3$ such that $\overline{H_{\psi_{a,k,\rho_k}}} = 0$ are exactly the sphere $\s_{k,\rho_k}$ of center $c_{k,\rho_k}:=\frac{ \overline{Y_{123}} }{ \overline{H_{\phi_{k,\rho_k}}} }$ and radius $r_{k,\rho_k} :=\frac{ |\overline{Y_{\phi_{k,\rho_k}}} |_\eta }{ \overline{H_{\phi_{k,\rho_k}}} }$. If $a\in \s_{k,\rho_k}$, then the conformal factor $\lambda_{\psi_{a,k,\rho_k}}$ of $\psi_{a,k,\rho_k}$ is given by
		\begin{align*}
			\lambda_{\psi_{a,k,\rho_k}} &= \lambda_{\phi_{k,\rho_k}} + \log|\phi_{k,\rho_k} - a|^2.
		\end{align*}
		Therefore, we need to find $a\in \s_{k,\rho_k}$ such that $0 < |\phi_{k,\rho_k} - a|$ on $\tilde{\Cr}_k$ and $C^{-1}<|\phi_{k,\rho_k}-a|<C$ on $A_1$. We prove that there exists $t>0$ independant of $R,k,\rho$ such that $B(0,t) \cap \s_{k,\rho_k} \neq \emptyset$ and any choice $a\in B(0,t) \cap \s_{k,\rho_k}$ gives $C^{-1} <|\phi_{k,\rho_k}-a|<C$ : thanks to \eqref{apriori_necks_R3} and the $\ve$-regularity, we have the estimate
		\begin{align*}
			\left\| Y_{\phi_{k,\rho_k}} \right\|_{L^2_\xi(A_1)} &\leq C.
		\end{align*}
		By Cauchy-Schwarz, $\left| \overline{Y_{\phi_{k,\rho_k}}} \right|_\xi < C$, and
		\begin{align*}
			|c_{k,\rho_k}| -r_{k,\rho_k} &= \frac{ |\overline{Y_{123}} | }{ |\overline{H_{\phi_{k,\rho_k}}}| } - \frac{ |\overline{Y_{\phi_{k,\rho_k}}} |_\eta }{ |\overline{H_{\phi_{k,\rho_k}}}| } \\
			&= \pm \frac{ |\overline{Y_4} + \overline{Y_5} |}{ |\overline{Y_{123}}| + |\overline{Y_{\phi_{k,\rho_k}}}|_\eta }.
		\end{align*}
		So that we can choose $t= \max(2\|\phi_{k,\rho_k}\|_{L^\infty(A_1)}, |\overline{Y_4} + \overline{Y_5} | )$. Note that the smallness of $\|\g \vec{n}_{\phi_{k,\rho_k}} \|_{L^2(B_2\setminus B_{1/2})}$ warrants the possibility the construct a sequence $(\phi_k)_{k\in\N}$ satisfying (\ref{normalization_phi}) and $\phi_k(A_1)$ being arbitrarely close to $\s_{\phi_k}$, since its radius is bounded from below.\\
		
		Therefore, it remains to show that we can choose $a_k\in B(0,t) \cap \s_{k,\rho_k}$ such that $\dist[a_k,\phi_{k,\rho_k}(\Sigma)]>0$. The umbilic points of $\phi_{k,\rho_k}$ are isolated points or closed curve. In particular, the set $\phi_{k,\rho_k}(\Sigma)\cap [B(0,t) \cap \s_{k,\rho_k}]$ cannot be an open set of $\s_{k,\rho_k}$ if $\phi_{k,\rho_k}$ is not a round sphere. So there exists $a_{k,\rho_k} \in [B(0,t) \cap \s_{k,\rho_k}]\setminus \phi_{k,\rho_k}(\Sigma)$. \\
		We define $\Psi_{k,\rho_k}$ and $\Theta_{k,\rho_k}$ as :
		\begin{align}\label{def_psi}
			\Psi_{k,\rho_k}(z) = \Theta_{k,\rho_k} \circ  \Phi_k(z) &:= \rho_k e^{\fint_{A_{\rho_k}} \lambda_k} \iota_{a_{k,\rho_k}}\circ \phi_{k,\rho_k}(z/\rho_k).
		\end{align}
	\end{proof}

	\begin{remark}\label{remark_possible_sphere}
		Let $s_k\in [Rr_k, \frac{1}{R}]$, and $\tilde{g}_k := \Psi_{k,s_k}^* \xi$. Using the proposition 1.2.1 in \cite{kuwert2012}, we can compute the new second fundamental form of $\Psi_{k,s_k}$ :
		\begin{align}\label{conformal_change_secondff}
			A^{\Psi_{k,s_k}}_{ij} &= e^{u\circ\Phi_k} \left( A^{\Phi_k}_{ij} - \scal{\g u\circ \Phi_k}{\vec{n}_{\Phi_k}} \delta_{ij}\right),
		\end{align}
		where $u$ is the conformal factor of $\Theta_{k,s_k}$. Dilations and rotations have constant conformal factor. Only the inversions have a non trivial conformal factor. If $a\in\R^3$ and  $\forall x\in\R^3,\ \iota_a(x):= \frac{x-a}{|x-a|^2}$. Then
		\begin{align*}
			\g \iota_a(x) &= \frac{id}{|x-a|^2} -2 \frac{(x-a)\otimes (x-a)}{|x-a|^4}.
		\end{align*}
		Hence, $|\g \iota_a(x)|^2 =C |x-a|^{-4}$, so that
		\begin{align*}
			\g \left( \log |\g \iota_a(x)|^2 \right) &= C\frac{x-a}{|x-a|^2}.
		\end{align*}
		Hence, on a dyadic annuli $A_{\rho_k} \subset B_{1/R}\setminus B_{Rr_k}$, there exists a dilation factor $\alpha_k \in \R\setminus \{0\}$, a rotation $R_k \in SO(3)$ and a point $a_k \in \R^3$ such that
		\begin{align*}
			\int_{A_{\rho_k}} |d\vec{n}_{\Psi_{k,s_k}}|^2_{\tilde{g}_k} d\vol_{\tilde{g}_k} &= \frac{C\rho_k^2}{\dist( \alpha_k R_k\Phi_k(A_{\rho_k}), a_k)^2 } + o(1).
		\end{align*}
		Therefore, one bubble might appear exactly where the center of the inversion is. However, the quantity $\left| \Arond_{\phi_k} \right|^2_{\tilde{g}_k} d\vol_{\tilde{g}_k} = \left| \Arond_k \right|^2_{g_k} d\vol_{g_k}$ is a pointwise conformal invariant. So the bubble is totally umbilic : it is a round sphere. This bubble can't appear at the scale $s_k$ by \eqref{defect_conformal_factor}. Therefore, up to restrict the region $\Cr_k$, \textbf{we assume that there is no new bubble}. In particular, we preserve \eqref{apriori_necks}. We will pursue this discussion with the remark \ref{remark_creation_oscillation} below.
	\end{remark}

	\begin{corollary}\label{control_normal_by_arond}
		With the notations of the lemma \ref{gauge_dyadic_annuli}, we obtain for any $\rho \in [2Rr_k, 1/(2R)]$,
		\begin{align*}
			\rho \| \g \vec{n}_{\Psi_{k,\rho}} \|_{L^\infty(\hat{A}_\rho)} &\leq C\left\| \Arond_{\Psi_{k,\rho}} e ^{-\lambda_{\Psi_{k,\rho}} } \right\|_{L^2(A_\rho)}.
		\end{align*}
	\end{corollary}
	
	\begin{proof}
		We proceded as in \cite{bernard2020}. By definition of $ H_{\Psi_{k,\rho}}$ and $\Arond_{\Psi_{k,\rho}}$, we have
		\begin{align*}
			\| \g \vec{n}_{\Psi_{k,\rho}} \|_{L^2(A_\rho)} &\leq \left\| H_{\Psi_{k,\rho}} e^{\lambda_{\Psi_{k,\rho}} } \right\|_{L^2(A_\rho)} + \left\| \Arond_{\Psi_{k,\rho}} e^{-\lambda_{\Psi_{k,\rho}} } \right\|_{L^2(A_\rho)}.
		\end{align*}
		Since $\fint_{A_\rho} H_{\Psi_{k,\rho}}=0$,
		\begin{align*}
			\| \g \vec{n}_{\Psi_{k,\rho}} \|_{L^2(A_\rho)}  &\leq \left\| \left(H_{\Psi_{k,\rho}} - \fint_{A_\rho} H_{\Psi_{k,\rho}} \right) e^{\lambda_{\Psi_{k,\rho}} } \right\|_{L^2(A_\rho)} + \left\| \Arond_{\Psi_{k,\rho}} e^{-\lambda_{\Psi_{k,\rho}} } \right\|_{L^2(A_\rho)}.
		\end{align*}
		By the \cref{control_h_a},
		\begin{align*}
			\| \g \vec{n}_{\Psi_{k,\rho}} \|_{L^2(A_\rho)} &\leq C \left\| \Arond_{\Psi_{k,\rho}} e^{-\lambda_{\Psi_{k,\rho}} } \right\|_{L^2(A_\rho)}.
		\end{align*}
		By conformal invariance,
		\begin{align*}
			\| \g \vec{n}_{\Psi_{k,\rho}} \|_{L^2(A_\rho)} &\leq C \left\| \Arond_{\Phi_k} e^{-\lambda_{\Phi_k} } \right\|_{L^2(A_\rho)}.
		\end{align*}
		We conclude with the $\ve$-regularity.
	\end{proof}

	\begin{lemma}\label{existence_gauge}
		There exists $C>1$ such that for every $k\in\N$, there exists a path $(\Xi_{k,\rho})_{\rho\in[Rr_k, 1/(2R)]}$ of conformal transfomations in $\R^3$ such that $(M_{k,\rho})_{\rho\in[Rr_k, 1/(2R)]} := (M_{\Xi_{k,\rho}})_{\rho\in[Rr_k, 1/(2R)]} \subset SO(4,1)$ satisfies : \\
		For every $\rho_k \in[Rr_k, \frac{1}{R}]$,
		\begin{align}
			\rho_k\| M_{k,\rho_k} \g Y_k \|_{L^\infty_\xi\left( \tilde{A}_{\rho_k} \right) } & \leq C\|\g Y_k \|_{L^2_\eta \left( A_{\rho_k} \right)},  \label{gradientY_bound_good_scale}\\
			\|M_{k,\rho_k} Y_k\|_{L^\infty_\xi (\tilde{A}_{\rho_k})} & \leq C. \label{Linfty_bound_Y}
		\end{align}
		For any $\alpha\in(0,1)$ there exists $C_\alpha>0$ such that for any $k\in\N$ and $s_k\in[ \alpha \rho_k, \frac{1}{R}]$, if $\phi_{k,\rho_k} := \Xi_{k,\rho_k} \circ \Phi_k$ :
		\begin{align}\label{gradientY_bound_wrong_scale}
			s_k\| M_{k,\rho_k} \g Y_k \|_{L^\infty_\xi\left( \tilde{A}_{s_k} \right) } & \leq C_\alpha \|\g \vec{n}_{\phi_{k,\rho_k}} \|_{L^2 \left( A_{s_k} \right)}.
		\end{align}
		Furthermore, the topology of $\phi_{k,\rho_k}(\Sigma)$ is the same as $\Phi_k(\Sigma)$ and we have the bound
		\begin{align}\label{eq:Linfty_phi}
			\sup_{k\in\N,\rho_k\in[Rr_k,1/R]} \| \phi_{k,\rho_k} \|_{L^\infty(\Sigma)} <\infty.
		\end{align}
	\end{lemma}
	
	\begin{remark}\label{global_L2xi_bound_Y}
		As a consequence, we have the $L^2_\xi$ bound : $\sup_{k\in\N} \int_{[ \alpha \rho_k, \frac{1}{R}]\times \s^1} |M_{k,\rho_k} \g Y_k|^2_\xi <\infty $. Indeed, since there is no new bubble, the formula \eqref{conformal_change_secondff} and the estimate \eqref{gradientY_bound_wrong_scale} shows that 
		\begin{align*}
			\int_{[ \rho_k, \frac{1}{2R}]\times \s^1} |M_{k,\rho_k} \g Y_k|^2_\xi &\leq C \int_{[\frac{1}{2} \rho_k, \frac{1}{R}]\times \s^1} |\g \vec{n}_{\phi_{k,\rho_k}}|^2_{\tilde{g}_k} d\vol_{\tilde{g}_k} \\
			&\leq C\int_{[\frac{1}{2} \rho_k, \frac{1}{R}]\times \s^1} |\g \vec{n}_{\Phi_k}|^2_{g_k} d\vol_{g_k} + o(1),
		\end{align*}
		where $\tilde{g}_k := \phi_{k,\rho_k}^*\xi$ and $g_k = \Phi_k^* \xi$.
	\end{remark}
	
	\begin{proof}
		Let $\Psi_{k,\rho_k}$ be as in the \cref{gauge_dyadic_annuli} and
		\begin{align*}
			\phi_{k,\rho_k}(z) &:=  \frac{1}{\rho_k} e^{-\overline{\lambda_k}} \left( \Psi_{k,\rho_k}(\rho_k z) - \Psi_{k,\rho_k}(\rho_k) \right), \\
			\overline{\lambda_k} &:= \fint_{A_{\rho_k}} \lambda_k .
		\end{align*}
		Let $\mu_{k,\rho_k}$ be the conformal factor of $\phi_{k,\rho_k}$. By the definition of $\Psi_{k,\rho_k}$ by (\ref{def_psi}), and the controls on the conformal factor, \cite[Theorem 0.2]{laurain2018} and \cite[Lemma II.1]{bernard2019} :
		\begin{align}\label{forget_conf_factor}
			\frac{1}{C} \leq e^{\mu_{k,\rho_k}} \leq C\ \ \ \ \text{on }A_1.
		\end{align}
		In particular, $\left\| \phi_{k,\rho_k} \right\|_{L^\infty(\tilde{A}_1)} \leq C$. We also have the bound
		\begin{align*}
			\|\g \vec{n}_{\phi_{k,\rho_k}} \|_{L^\infty \left( \hat{A}_1 \right)} &\leq C \left\| \Arond_{\phi_{k,\rho_k}} e^{-\lambda_{k,\rho_k}} \right\|_{L^2 \left( A_1 \right) }.
		\end{align*}
		Note that $\phi_{k,\rho_k}(z) = \Xi_{k,\rho_k} \circ \Phi_k(\rho_k z)$ for some $\Xi_{k,\rho_k} \in \Conf(\R^3)$. We denote $M_{k,\rho_k} := M_{\Xi_{k,\rho_k}} $.\\
		
		\underline{Proof of the estimate \eqref{Linfty_bound_Y} :}\\
		We can estimate the mean curvature of $\phi_{k,\rho_k}$ in the following way : since $C^{-1}< \fint_{\tilde{A}_1} \mu_{k,\rho_k}<C$,
		\begin{align*}
			\left\| H_{\phi_{k,\rho_k}} \right\|_{L^\infty(\tilde{A}_1)} &= \left\| e^{-2\mu_{k,\rho_k}} \scal{\lap \phi_{k,\rho_k}}{\vec{n}_{\phi_{k,\rho_k}}} \right\|_{L^\infty(\tilde{A}_1)}\\
			&\leq C\left\| e^{-2\mu_{k,\rho_k}} |\g^2 \phi_{k,\rho_k} | \right\|_{L^\infty(\tilde{A}_1)}.
		\end{align*}
		By (\ref{forget_conf_factor}) :
		\begin{align*}
			\left\| H_{\phi_{k,\rho_k}} \right\|_{L^\infty(\tilde{A}_1)} &\leq C \left\| e^{-\mu_{k,\rho_k}} |\g^2 \phi_{k,\rho_k} | \right\|_{L^\infty(\tilde{A}_1)}.
		\end{align*}
		By the $\ve$-regularity :
		\begin{align*}
			\left\| H_{\phi_{k,\rho_k}} \right\|_{L^\infty(\tilde{A}_1)} &\leq C \left( 1+ \left\|\g \vec{n}_{\phi_{k,\rho_k}} \right\|_{L^2(\hat{A}_1)} \right).
		\end{align*}
		By the lemma \ref{control_normal_by_arond} :
		\begin{align*}
			\left\| H_{\phi_{k,\rho_k}} \right\|_{L^\infty(\tilde{A}_1)} &\leq C\left( 1+ \|\g Y_k\|_{L^2_\eta(A_\rho)} \right)\\
			&\leq C.
		\end{align*}
		For the conformal Gauss map, we have $M_{k,\rho_k} Y_k(\rho_k z) = Y_{\phi_{k,\rho_k}}(z)$, so that
		\begin{align*}
			\| M_{k,\rho_k}Y_k\|_{L^\infty_\xi (\tilde{A}_{\rho_k})} &= \left\| Y_{\phi_{k,\rho_k}} \right\|_{L^\infty_\xi(\tilde{A}_1)} \\
			&\leq \|H_{\phi_{k,\rho_k}}\|_{L^\infty(\tilde{A}_1)} \left( 1+ \|\phi_{k,\rho_k} \|_{L^\infty(\tilde{A}_1)} + \|\phi_{k,\rho_k} \|_{L^\infty(\tilde{A}_1)}^2 \right) \\
			& + 1+ \|\phi_{k,\rho_k} \|_{L^\infty(\tilde{A}_1)} \\
			&\leq C.
		\end{align*}
		
		\underline{Proof of the estimate \eqref{gradientY_bound_good_scale} :}\\
		Considering the gradient, we have
		\begin{align*}
			\rho_k \left\| M_{k,\rho_k} \g Y_k  \right\|_{L^\infty_\xi(\tilde{A}_{\rho_k})} &= \| \g Y_{\phi_{k,\rho_k}} \|_{L^\infty_\xi (\tilde{A}_1)}\\
			&\leq \left\|\g H_{\phi_{k,\rho_k}} \right\|_{L^\infty(\tilde{A}_1)} +C \left\| \Arond_{\phi_{k,\rho_k}} e^{-\mu_{k,\rho_k}} \right\|_{L^\infty(\tilde{A}_1)} \\
			&\leq C\left\| \g \left( e^{-2\mu_{k,\rho_k}} \scal{\g \phi_{k,\rho_k}}{\g \vec{n}_{\phi_{k,\rho_k}}} \right) \right\|_{L^\infty(\tilde{A}_1)} + C\|\g \vec{n}_{\phi_{k,\rho_k}} \|_{L^\infty(\tilde{A}_1)} \\
			&\leq C\left\| e^{-\mu_{k,\rho_k}} |\g^2\vec{n}_{\phi_{k,\rho_k}}| + e^{-2\mu_{k,\rho_k}} |\g^2 \phi_{k,\rho_k}| |\g \vec{n}_{\phi_{k,\rho_k}} | \right\|_{L^\infty(\tilde{A}_1)} +  C\|\g \vec{n}_{\phi_{k,\rho_k}} \|_{L^\infty(\tilde{A}_1)} \\
			&\leq C \|\g^2 \vec{n}_{\phi_{k,\rho_k}} \|_{L^\infty(\tilde{A}_1)} +  \left\| e^{-\mu_{k,\rho_k}} \g^2 \phi_{k,\rho_k} \right\|_{L^\infty(\tilde{A}_1)} \|\g \vec{n}_{\phi_{k,\rho_k}} \|_{L^2(\tilde{A}_1)} + C\|\g \vec{n}_{\phi_{k,\rho_k}} \|_{L^\infty(\tilde{A}_1)} \\
			&\leq C\|\g Y_k\|_{L^2_\eta(A_\rho)}.
		\end{align*}
		
		\underline{Proof of the estimate \eqref{gradientY_bound_wrong_scale} :}\\
		We proceed as for (\ref{gradientY_bound_good_scale}).
		\begin{align*}
			\rho_k \left\| M_{k,\rho_k} \g Y_k  \right\|_{L^\infty_\xi(\tilde{A}_{s_k})} &= \| \g Y_{\phi_{k,\rho_k}} \|_{L^\infty_\xi (\tilde{A}_{s_k/\rho_k})}\\
			&\leq \left\|\g H_{\phi_{k,\rho_k}} \right\|_{L^\infty(\tilde{A}_{s_k/\rho_k})} +C \left\| \Arond_{\phi_{k,\rho_k}} e^{-\mu_{k,\rho_k}} \right\|_{L^\infty(\tilde{A}_{s_k/\rho_k})} \\
			&\leq C\left\| \g \left( e^{-2\mu_{k,\rho_k}} \scal{\g \phi_{k,\rho_k}}{\g \vec{n}_{\phi_{k,\rho_k}}} \right) \right\|_{L^\infty(\tilde{A}_{s_k/\rho_k})} + C\| \g \vec{n}_{\phi_{k,\rho_k}} \|_{L^\infty(\tilde{A}_{s_k/\rho_k})} \\
			&\leq C\left\| e^{-\mu_{k,\rho_k}} |\g^2\vec{n}_{\phi_{k,\rho_k}}| + e^{-2\mu_{k,\rho_k}} |\g^2 \phi_{k,\rho_k}| |\g \vec{n}_{\phi_{k,\rho_k}} | \right\|_{L^\infty(\tilde{A}_{s_k/\rho_k})} +  C\|\g \vec{n}_{\phi_{k,\rho_k}} \|_{L^\infty(\tilde{A}_{s_k/\rho_k})}.
		\end{align*}
		Using the harnack estimate on the conformal factor, \cite[Lemma II.1]{bernard2019}, and letting $\overline{\mu_{k,\rho_k}} := \fint_{\tilde{A}_{s_k/\rho_k}} \mu_{k,\rho_k}$, we obtain
		\begin{align*}
			\rho_k \left\| M_{k,\rho_k} \g Y_k  \right\|_{L^\infty_\xi(\tilde{A}_{s_k})} &\leq C e^{-\overline{\mu_{k,\rho_k}}} \Big( \|\g^2 \vec{n}_{\phi_{k,\rho_k}} \|_{L^\infty(\tilde{A}_{s_k/\rho_k})} \\
			&+  \left\| e^{-\mu_{k,\rho_k}} \g^2 \phi_{k,\rho_k} \right\|_{L^\infty(\tilde{A}_{s_k/\rho_k})} \|\g \vec{n}_{\phi_{k,\rho_k}} \|_{L^\infty(\tilde{A}_{s_k/\rho_k})} \Big) + C\|\g \vec{n}_{\phi_{k,\rho_k}} \|_{L^\infty(\tilde{A}_{s_k/\rho_k})}.
		\end{align*}
		So
		\begin{align*}
			\rho_k \left\| M_{k,\rho_k} \g Y_k  \right\|_{L^\infty_\xi(\tilde{A}_{s_k})} &\leq C e^{-\overline{\mu_{k,\rho_k}}} \left( \frac{\rho_k}{s_k} \right)^2 \|\g \vec{n}_{\phi_{k,\rho_k}} \|_{L^2(A_{s_k/\rho_k})} + C\frac{\rho_k}{s_k}  \|\g \vec{n}_{\phi_{k,\rho_k}} \|_{L^2(A_{s_k/\rho_k})}.
		\end{align*}
		Hence,
		\begin{align}\label{almost_estimate_wrong_scale}
			s_k \left\| M_{k,\rho_k} \g Y_k  \right\|_{L^\infty_\xi(\tilde{A}_{s_k})} &\leq C\left( 1+ \frac{\rho_k}{s_k} e^{-\overline{\mu_{k,\rho_k}}} \right)  \|\g \vec{n}_{\phi_{k,\rho_k}} \|_{L^2(A_{s_k/\rho_k})}.
		\end{align}
		Consider $j\in\N$ such that $2^j \rho_k\leq s_k < 2^{j+1} \rho_k$. We have
		\begin{align*}
			-\overline{\mu_{k,\rho_k}} &= -\fint_{\tilde{A}_{s_k/\rho_k}} \mu_{k,\rho_k} \\
			&= \left( \fint_{\tilde{A}_{2^j}} \mu_{k,\rho_k} -\fint_{\tilde{A}_{s_k/\rho_k}} \mu_{k,\rho_k} \right) - \fint_{\tilde{A}_{2^j}} \mu_{k,\rho_k}\\
			&= \sum_{i=0}^{j-1} \left( \fint_{\tilde{A}_{2^i}} \mu_{k,\rho_k} -\fint_{\tilde{A}_{2^{i+1}}} \mu_{k,\rho_k} \right) + \left( \fint_{\tilde{A}_{2^j}} \mu_{k,\rho_k} -\fint_{\tilde{A}_{s_k/\rho_k}} \mu_{k,\rho_k} \right) - \fint_{\tilde{A}_1} \mu_{k,\rho_k} .
		\end{align*}
		So
		\begin{align*}
			e^{-\overline{\mu_{k,\rho_k}}} &= \exp\left[ \left( \fint_{\tilde{A}_{2^j}} \mu_{k,\rho_k} -\fint_{\tilde{A}_{s_k/\rho_k}} \mu_{k,\rho_k} \right) - \fint_{\tilde{A}_1} \mu_{k,\rho_k} \right] \prod_{i=0}^{j-1} \exp\left[ \fint_{\tilde{A}_{2^i}} \mu_{k,\rho_k} -\fint_{\tilde{A}_{2^{i+1}}} \mu_{k,\rho_k} \right].
		\end{align*}
		If $j\geq 0$, using the controls on the conformal factor, see \cite[Theorem 0.2]{laurain2018} and \cite[Lemma II.1]{bernard2019}, together with \eqref{forget_conf_factor}, we obtain
		\begin{align*}
			e^{-\overline{\mu_{k,\rho_k}}} &\leq C \exp\left[ Cj \sup_{\rho\in[ Rr_k, \frac{1}{2R}]} \|\g \vec{n}_{\phi_{k,\rho_k}} \|_{L^2(B_{2\rho}\setminus B_\rho)}^2 \right].
		\end{align*}
		Using \eqref{apriori_necks_R3}, up to reduce $\ve_0$ we obtain
		\begin{align*}
			e^{-\overline{\mu_{k,\rho_k}}} &\leq Ce^{(\log 2)j} \leq C\frac{s_k}{\rho_k}.
		\end{align*}
		Therefore, \eqref{almost_estimate_wrong_scale} reduces to
		\begin{align*}
			s_k \left\| M_{k,\rho_k} \g Y_k  \right\|_{L^\infty_\xi(\tilde{A}_{s_k})} &\leq C  \|\g \vec{n}_{\phi_{k,\rho_k}} \|_{L^2(A_{s_k/\rho_k})}.
		\end{align*}

		\underline{Topology :}\\
		The topology is preserved thanks to the \cref{gauge_dyadic_annuli}. \\
		
		\underline{$L^\infty$-bound :}\\
		Here we use the fact that $\Phi_k$ is defined on a whole closed surface $\Sigma$ together with the bound \eqref{hyp:uniform_L2_sff}. The $L^\infty$ bound is given by the Lemma 5.13 in \cite{riviere2013a} : There exists $r>0$ and a bounded sequence $(p_k)_{k\in\N}\subset \R^3$ such that for any $k\in\N$, $\phi_{k,\rho_k}(\Sigma)\cap B_r(p_k) = \emptyset$. We consider the inversion $\iota_k$ with respect to $B_r(p_k)$.
		For the conformal Gauss map, the inversion $\iota_k$ corresponds to a multiplication by a bounded matrix. Hence, we preserve the estimates $L^\infty_\xi\cap W^{1,2}_\xi$. The resulting conformal transformation is the one we are looking for.
	\end{proof}

	\begin{remark}\label{remark_creation_oscillation}
		We conclude the \cref{remark_possible_sphere}. We now prove that the sphere that might appear by our choice of conformal transformation generates oscillations of the conformal Gauss map. The radius of this new sphere must go to $0$. Indeed, the bubble has to be degenerated thanks to the proof of \cite[Theorem 0.2]{laurain2018a}. We sketch it briefly for completeness. If the sphere wasn't degenerated, then we could blow up the neck region and obtain a Willmore sphere with at least 2 ends. With an appropriate inversion, we obtain a Willmore sphere. By Bryant's classification, this surface has at least 4 ends. Thanks to \eqref{apriori_necks_R3}, this surface must be totally umbilic, contradiction. Hence, the sphere has to be degenerated. For a sphere $\s(p,r)$ of center $p\in\R^3$ and radius $r>0$, the conformal Gauss map is constant and given by 
		\begin{align*}
			Y_{\s(p,r)} &= \frac{1}{r} \begin{pmatrix}
				p\\ \frac{|p|^2-r^2 - 1}{2} \\ \frac{|p|^2 - r^2 +1}{2}
			\end{pmatrix}.
		\end{align*}
		If $r\to 0$, then $Y_{\s(p,r)} \to \infty$. By \eqref{Linfty_bound_Y}, there must be oscillations.
	\end{remark}

	\subsection{Convergence to a geodesic in the $C^1$-topology }\label{section_convergence_geodesic}

	The goal of the section, see the proposition \ref{convergence_geodesic}, is to prove that on a region of fixed oscillations, with the hypothesis of the lemma \ref{existence_gauge}, then $(Y_k)_k$ converges to a light-like geodesic. Following the analysis of \cite{zhu2008}, we consider a domain as a long cylinder. We consider the cylinders
	\begin{align*}
		\Cr_k &:= [-T_k, T_k]\times \s^1, \\
		\bar{\Cr}_k &:=[-T_k + 2, T_k-2]\times \s^1, \\
		\forall s,t,\ \ \ Q(t,s) &:= [t-s,t+s]\times \s^1 .
	\end{align*}
	where $T_k \xrightarrow[k\to \infty]{}{+\infty}$.\\
	
	\begin{claim}
		We can restrict ourselves to the following setting : there exists some fixed $\delta>0$, such that
		\begin{align}
			& \sup_{k\in\N} \Big( \|\g Y_k\|_{L^2_\xi(\Cr_k)} + \| Y_k \|_{L^\infty_\xi(\Cr_k)} \Big) <\infty, \label{uniform_Linfty_bound} \\
			& \| \g Y_k\|_{L^\infty_\xi(\bar{\Cr}_k)} \xrightarrow[k\to \infty]{}{0}, \label{unif_conv_to_zero}  \\
			& \osc_{\Cr_k} Y_k = \delta. \label{hypothesis_oscillation}
		\end{align}
	\end{claim}
	
	\begin{proof}
		By the lemma \ref{existence_gauge} and the remark \ref{global_L2xi_bound_Y}, we are in the following setting. The sequence $(Y_k)_{k\in\N}$ is a defined from the flat cylinder $\Cr_k$ into $\s^{3,1}$, with the following estimates (for simplicity, we consider that the matrix in $SO(4,1)$ is $I_5$) :
		\begin{align*}
			& \sup_{k\in\N} \Big( \|\g Y_k\|_{L^2_\xi(\Cr_k)} + \| Y_k \|_{L^\infty_\xi(Q(-T_k+1,1))} \Big) <\infty, \\
			& \| \g Y_k\|_{L^\infty_\xi(\bar{\Cr}_k)} \xrightarrow[k\to \infty]{}{0}.
		\end{align*}
		
		Let $t_1,t_2\in [-T_k,T_k]$ with $t_1<t_2$. We define the euclidean average length of $Y_k$ between $t_1$ and $t_2$ :
		\begin{align*}
			L_k(t_1,t_2) := \int_{t_1}^{t_2} \left( \int_{\s^1} |\dr_t Y_k(t,\theta)|^2_\xi d\theta \right)^\frac{1}{2} dt .
		\end{align*}
		
		We recall the proposition 3.5 in \cite{zhu2008} : for any $-T_k < t_1<t_2 < T_k$,
		\begin{align}\label{oscillations}
			\osc_{[t_1,t_2]\times \s^1} Y_k &\leq 4\pi \|\g Y_k\|_{L^\infty_\xi(\bar{\Cr}_k)} + \frac{1}{\sqrt{2\pi}} L(t_1,t_2) .
		\end{align}
		Hence, we have two possibilities :
		\begin{itemize}
			\item $\osc_{[t_1,t_2]\times \s^1} Y_k \xrightarrow[k\to \infty]{}{0}$ : there is no oscillation in the neck, geometrically, this means that the two sides $\displaystyle{\lim_{k\to\infty}} Y_k(\cdot + T_k)$ and $\displaystyle{\lim_{k\to\infty}} Y_k(\cdot - T_k)$ are connected.
			\item There exists $\delta>0$ fixed and $t_k > -T_k+1$ such that
			\begin{align}
				\osc_{[-T_k,t_k]\times \s^1} Y_k &= \delta .
			\end{align}
			In this case, the euclidean average length $L(-T_k,T_k)$ is bounded away from $0$.
		\end{itemize}
		We focus on the second case : we assume that \eqref{hypothesis_oscillation} holds until the end of the section. Therefore, up to change the cylinder, we have the following estimates :
		\begin{align*}
			& \sup_{k\in\N} \Big( \|\g Y_k\|_{L^2_\xi(\Cr_k)} + \| Y_k \|_{L^\infty_\xi(\Cr_k)} \Big) <\infty, \\
			& \| \g Y_k\|_{L^\infty_\xi(\bar{\Cr}_k)} \xrightarrow[k\to \infty]{}{0}.
		\end{align*}
	\end{proof}
	
	For a function $g : \R^2 \to \R$, we let $g^*(r) := \fint_{\s^1} g(r,\theta) d\theta = \fint_{\dr B_r(0)} g$. If $t\in[-T_k,T_k]$, let
	\begin{align*}
		\alpha_k(t) &:= \int_{\s^1} |\dr_\theta Y_k(t,\theta)|^2_\xi d\theta, \\
		\beta_k(t) &:= \int_{\s^1} |\g Y_k(t,\theta)|^2_\xi d\theta, \\
		\gamma_k(t) &:= \int_{\s^1} |\g(Y_k - Y_k^*)(t,\theta)|^2_\xi d\theta .
	\end{align*}
	By differentiating $\alpha_k$, we can obtain a first estimate independant of the hypothesis \eqref{hypothesis_oscillation} :
	
	\begin{lemma}\label{estimate_xitheta}
		Under the assumptions \eqref{uniform_Linfty_bound} and \eqref{unif_conv_to_zero}, for $k$ large enough, we have 
		\begin{align*}
			\forall t\in [-T_k,T_k], \ \ \ \alpha_k''(t)& \geq \alpha_k(t).
		\end{align*}
		For any $t_1<t<t_2$, we can estimate
		\begin{align}\label{pointwise_alpha}
			\alpha_k(t) &\leq \max(\alpha_k(t_1),\alpha_k(t_2)).
		\end{align}
		Moreover, for any $\nu\in(0,1]$, if $|t_1-t_2|>1$ :
		\begin{align}\label{integral_xitheta}
			\int_{t_1}^{t_2} \alpha_k(t)^\nu dt &\leq C(\nu) \left[ \alpha_k(t_1)^\nu + \alpha_k(t_2)^\nu \right].
		\end{align}
	\end{lemma}
	
	\begin{proof}
		We compute
		\begin{align*}
			\alpha_k''(t) &= \frac{d}{dt} \left( 2\int_{\s^1} \scal{\dr_\theta Y_k}{\dr^2_{\theta t} Y_k}_\xi d\theta \right) \\
			&= 2\int_{\s^1} \left( |\dr^2_{\theta t} Y_k|^2_\xi + \scal{\dr_\theta Y_k}{\dr^3_{\theta tt} Y_k}_\xi \right) \\
			&= 2\int_{\s^1} \left( |\dr^2_{\theta t} Y_k|^2_\xi - \scal{\dr^2_{\theta\theta} Y_k}{\dr^2_{tt} Y_k}_\xi \right) \\
			&= 2\int_{\s^1} \left( |\dr^2_{\theta t} Y_k|^2_\xi - \scal{\dr^2_{\theta\theta} Y_k}{\lap Y_k - \dr^2_{\theta \theta} Y_k}_\xi \right) \\
			&= 2\int_{\s^1} \left( |\dr^2_{\theta t} Y_k|^2_\xi + |\dr^2_{\theta \theta} Y_k|^2_\xi + \scal{\dr^2_{\theta \theta} Y_k}{ |\g Y_k|^2_\eta Y_k}_\xi \right) \\
			&= 2\int_{\s^1} \left( |\dr^2_{\theta t} Y_k|^2_\xi + |\dr^2_{\theta \theta} Y_k|^2_\xi - \scal{\dr_{\theta} Y_k}{ \dr_\theta \left( |\g Y_k|^2_\eta Y_k\right) }_\xi \right).
		\end{align*}
		We can bound the third term using \eqref{uniform_Linfty_bound}
		\begin{align}
			\left| \scal{\dr_{\theta} Y_k}{ \dr_\theta \left( |\g Y_k|^2_\eta Y_k\right) }_\xi \right| &\leq |\dr_\theta Y_k|_\xi \left| \dr_\theta \left( |\g Y_k|^2_\eta Y_k\right) \right|_\xi \nonumber \\
			&\leq C|\g \dr_\theta Y_k |_\xi |\g Y_k|_\xi |\dr_\theta Y_k|_\xi + |\g Y_k|^2_\eta |\dr_\theta Y_k|_\xi^2 .  \label{third_derivative}
		\end{align}
		Hence,
		\begin{align*}
			\alpha_k''(t) &\geq  \left( 2 - C\|\g Y_k\|_{L^\infty_\xi(\bar{\Cr}_k)}^2 \right) \int_{\s^1}\left(  |\dr^2_{\theta t} Y_k|^2_\xi +  |\dr^2_{\theta \theta} Y_k|^2_\xi\right) -C\|\g Y_k\|_{L^\infty_\eta(\bar{\Cr}_k)}^2 \int_{\s^1} |\dr_\theta Y_k|^2_\xi .
		\end{align*}
		By \eqref{unif_conv_to_zero} and Poincaré inequality on $\s^1$, we obtain for $k$ large enough :
		\begin{align*}
			\alpha_k''(t) &\geq \frac{5}{4} \int_{\s^1} |\dr^2_{\theta\theta} Y_k|^2_\xi - \frac{1}{1000} \int_{\s^1} |\dr_\theta Y_k|^2_\xi \\
			&\geq \int_{\s^1} |\dr_\theta Y_k|^2_\xi = \alpha_k(t).
		\end{align*}
		Consider $-T_k<t_1<t_2<T_k$ and let $\tau:[t_1,t_2]\to \R$ be the solution to
		\begin{align*}
			\tau'' &= \tau ,\\
			\tau(t_1) &= \alpha_k(t_1) ,\\
			\tau(t_1) &= \alpha_k(t_2).
		\end{align*}
		Then $\tau(t) = \lambda e^t + \mu e^{-t}$ with constants
		\begin{align*}
			\lambda &:= \frac{ e^{-2t_2 - t_1} \alpha_k(t_1) - e^{-2t_1 - t_2} \alpha_k(t_2) }{ e^{-2t_2} - e^{-2t_1} }, \\
			\mu &:= \frac{ e^{-t_2} \alpha_k(t_2) - e^{-t_1} \alpha_k(t_1) }{e^{-2t_2} - e^{-2t_1} }.
		\end{align*}
		By the maximum principle : for any $t\in[t_1,t_2]$,
		\begin{align*}
			0 \leq \alpha_k(t) &\leq \tau(t) = \frac{ e^{t-2t_2 - t_1} - e^{-t-t_1} }{ e^{-2t_2} - e^{-2t_1} } \alpha_k(t_1) + \frac{ - e^{t-2t_1 - t_2} + e^{-t-t_2} }{ e^{-2t_2} - e^{-2t_1} } \alpha_k(t_2) \\
			&\leq \max(\alpha_k(t_1),\alpha_k(t_2)).
		\end{align*}
		Hence, if $\nu\in(0,1]$ and $|t_1-t_2|>1$ :
		\begin{align*}
			\int_{t_1}^{t_2} \alpha_k(t)^\nu dt &\leq \frac{1}{\nu} \left( |\lambda|^\nu (e^{\nu t_2} - e^{\nu t_1} ) + |\mu|^\nu (e^{-\nu t_1} - e^{-\nu t_2}) \right)\\
			&\leq \frac{C(\nu)}{(e^{-2t_2} - e^{-2t_1})^\nu} \left[ |\alpha_k(t_1)|^\nu \left( e^{-\nu(2t_2+t_1) + \nu t_2} - e^{-\nu(2t_2 + t_1) + \nu t_1} + e^{-2\nu t_1} - e^{-\nu(t_1 + t_2)} \right) \right.\\
			& \left. + |\alpha_k(t_2)|^\nu \left( e^{-\nu(2t_1+t_2)} (e^{\nu t_2} - e^{\nu t_1}) + e^{-\nu t_2} (e^{-\nu t_1} - e^{-\nu t_2} ) \right) \right] \\
			&\leq \frac{C(\nu)}{(e^{-2t_2} - e^{-2t_1})^\nu}(e^{-2\nu t_1} + e^{-2\nu t_2}) \left( \alpha_k(t_1)^\nu + \alpha_k(t_2)^\nu \right) \\
			&\leq C(\nu) \left( \alpha_k(t_1)^\nu + \alpha_k(t_2)^\nu \right).
		\end{align*}
	\end{proof}
	
	Thanks to \eqref{integral_xitheta} with the choice $\nu = \frac{1}{2}$, we obtain
	\begin{align*}
		\int_{-T_k+2}^{T_k-2} \alpha_k(t)^\frac{1}{2} dt &\leq C\|\g Y_k\|_{L^\infty(\bar{\Cr}_k)}^{1/2} \xrightarrow[k\to \infty]{}{0}.
	\end{align*}
	We have assumed that there exists oscillations of size $\delta$ in $\bar{\Cr}_k$, by \eqref{hypothesis_oscillation}. By \eqref{oscillations}, we obtain
	\begin{align*}
		\int_{-T_k+2}^{T_k-2} \alpha_k(t)^\frac{1}{2} dt \ust{k\to \infty}{=} o\left( \int_{-T_k+2}^{T_k-2} \beta_k(t)^\frac{1}{2} dt \right).
	\end{align*}
	In order to show the convergence to a geodesic, we first need to show that $|\dr_\theta Y_k|_\xi = o(|\dr_t Y_k|_\xi) $ uniformly along the cylinder. Since the oscillations only come from $\dr_t Y_k$ by \eqref{oscillations}, we show that $\alpha_k = o(\beta_k)$.\\
	
	\begin{claim}\label{claim_boundary_conditions_C1}
		There exists $-T_k < t_{1,k}<t_{2,k}<T_k$ such that
		\begin{enumerate}[label=\alph*)]
			\item\label{restriction_Cr1} $\alpha_k(t_{1,k}) = o(\beta_k(t_{1,k}))$ and $\alpha_k(t_{2,k}) = o(\beta_k(t_{2,k}))$,
			\item\label{restriction_Cr3} $\osc_{[t_{1,k}, t_{2,k}]\times \s^1} Y_k \geq \delta/2$.
		\end{enumerate}
	\end{claim} 
	
	\begin{proof}
		Choose $-T_k<\tau_{1,k}<\tau_{2,k} <T_k$ such that
		\begin{align*}
			\osc_{[-T_k,\tau_{1,k} ]\times \s^1} Y_k = \osc_{[\tau_{2,k},T_k ]\times \s^1} Y_k = \frac{\delta}{10}.
		\end{align*}
		We remark that $\beta_k$ never vanishes since the critical points of $Y_k$ are isolated. Consider $t_{1,k} \in [-T_k,\tau_{1,k}]$ and $t_{2,k} \in [\tau_{2,k},T_k]$ such that 
		\begin{align*}
			\frac{\alpha_k(t_{1,k}) }{\beta_k(t_{1,k})} = \min_{ [-T_k,\tau_{1,k} ] } \frac{\alpha_k}{\beta_k}, &  & \frac{\alpha_k(t_{2,k}) }{\beta_k(t_{2,k})} = \min_{ [\tau_{2,k},T_k ] } \frac{\alpha_k}{\beta_k}.
		\end{align*}
		Let $I_1 := [-T_k,\tau_{1,k} ]$, $I_2 := [\tau_{2,k},T_k ]$ and consider $i\in\{1,2\}$. If $\frac{\alpha_k(t_{i,k}) }{\beta_k(t_{i,k})}$ is bounded below by some constant $c>0$, then 
		\begin{align*}
			\forall t\in I_i, \ \ \ \ \beta_k(t) \leq \frac{\alpha_k(t)}{c}.
		\end{align*}
		Thanks to \eqref{oscillations}, 
		\begin{align*}
			\frac{\delta}{10} = \osc_{I_i} Y_k &\leq  4\pi \|\g Y_k\|_{L^\infty_\xi(I_i \times \s^1)} + \frac{1}{\sqrt{2\pi}} \int_{I_i} \beta_k(t)^\frac{1}{2} dt \\
			&\leq 4\pi \|\g Y_k\|_{L^\infty_\xi(I_i \times \s^1)} + \frac{1}{\sqrt{2\pi c}} \int_{I_i} \alpha_k(t)^\frac{1}{2} dt.
		\end{align*}
		Using \eqref{integral_xitheta}, we conclude that 
		\begin{align*}
			\frac{\delta}{10} &\leq C\|\g Y_k\|_{L^\infty_\xi(I_i \times \s^1)} \xrightarrow[k\to\infty]{}{0}.
		\end{align*}
		Contradiction. Therefore, we obtain $\alpha_k(t_{i,k}) \ust{k\to\infty}{=} o(\beta_k(t_{i,k}))$. Furthermore,
		\begin{align*}
			\osc_{[t_{1,k},t_{2,k}]\times \s^1} Y_k \geq \delta - \osc_{[-T_k,t_{1,k}]\times \s^1} Y_k - \osc_{[t_{2,k},T_k]\times \s^1} Y_k \geq \frac{4}{5}\delta.
		\end{align*}
	\end{proof}
	
	Note that in the condition \ref{restriction_Cr3}, one can replace $\delta/2$ by $a\delta$ for any fixed $a\in(0,1)$. We show in the lemma \ref{step1_conv_geodesic} that the condition \ref{restriction_Cr1} becomes uniform on the cylinder $[t_{1,k},t_{2,k}]\times \s^1$. We will see that these estimates implies that $Y_k$ converges to a light-like geodesic in $\s^{3,1}$. To do so, we will compare $Y_k$ to its average $Y_k^*$. \\
	
	\begin{lemma}\label{step1_conv_geodesic}
		Consider $-T_k<t_{1,k}<t_{2,k}<T_k$ such that $\alpha_k(t_{1,k}) \ust{k\to \infty}{=} o(\beta_k(t_{1,k}))$ and $\alpha_k(t_{2,k}) \ust{k\to \infty}{=} o(\beta_k(t_{2,k}))$. Then
		\begin{align}\label{difference_t_theta}
			\max_{[t_{1,k},t_{2,k}]} \alpha_k \ust{k\to \infty}{=} o\left(\min_{[t_{1,k},t_{2,k}]} \beta_k \right),
		\end{align}
		And 
		\begin{align}\label{uniform_arclength}
			\max_{[t_{1,k},t_{2,k}]} \beta_k &\ust{k\to\infty}{\sim}  \min_{[t_{1,k},t_{2,k}]} \beta_k .
		\end{align}
	\end{lemma}
	
	\begin{proof}
		We first estimate the difference between $\beta_k$ and $\alpha_k$. We compute :
		\begin{align}
			\beta_k'(t) &= 2\int_{\s^1} \left( \scal{\dr_t Y_k}{\dr^2_{tt} Y_k}_\xi + \scal{\dr_\theta Y_k}{\dr^2_{t\theta} Y_k}_\xi \right) \nonumber\\
			&= 2\int_{\s^1} \scal{\dr_t Y_k}{\lap Y_k - \dr^2_{\theta\theta } Y_k}_\xi + \alpha_k'(t) \nonumber \\
			&= 2\int_{\s^1} \left( \scal{\dr_t Y_k}{-|\g Y_k|^2_\eta Y_k}_\xi + \scal{\dr^2_{t\theta} Y_k}{\dr_\theta Y_k}_\xi \right) + \alpha_k'(t) \nonumber \\
			&= -2\int_{\s^1} \scal{\dr_t Y_k}{|\g Y_k|^2_\eta Y_k}_\xi + 2\alpha_k'(t). \label{derivative_difference_beta_alpha}
		\end{align}
		Thanks to \eqref{residue_hopf_diff} :
		\begin{align}
			|(\beta_k-2\alpha_k)'(t) | &\leq C\|\g Y_k\|_{L^\infty_\xi(\bar{\Cr}_k)} \int_{\s^1} |\g Y_k|^2_\eta \nonumber \\
			&\leq C\|\g Y_k\|_{L^\infty_\xi(\bar{\Cr}_k)} \int_{\s^1} |\dr_\theta Y_k|^2_\eta \nonumber\\
			&\leq C\|\g Y_k\|_{L^\infty_\xi(\bar{\Cr}_k)} \alpha_k(t). \label{comparison_xit_xitheta}
		\end{align}
		We can integrate on $[t_{1,k},t]$ :
		\begin{align}
			|\beta_k(t) - 2\alpha_k(t)| &\geq |\beta_k(t_{1,k}) - 2\alpha_k(t_{1,k})| - \int_{t_{1,k}}^t |(\beta_k-2\alpha_k)'| \nonumber \\
			&\geq |\beta_k(t_{1,k}) - 2\alpha_k(t_{1,k})| - C\|\g Y_k\|_{L^\infty_\xi(\bar{\Cr}_k)} \int_{t_{1,k}}^{t_{2,k}} \alpha_k  \nonumber \\
			&\geq |\beta_k(t_{1,k}) - 2\alpha_k(t_{1,k})| - C\|\g Y_k\|_{L^\infty_\xi(\bar{\Cr}_k)} \left[ \alpha_k(t_{1,k}) + \alpha_k(t_{2,k}) \right]. \label{pointwise_diff_alpha_beta1}
		\end{align}
		If we integrate on $[t,t_{2,k}]$ :
		\begin{align}
			|\beta_k(t) - 2\alpha_k(t)| &\geq |\beta_k(t_{2,k}) -2\alpha_k(t_{2,k})| - C\|\g Y_k\|_{L^\infty_\xi(\bar{\Cr}_k)} \left[ \alpha_k(t_{1,k}) + \alpha_k(t_{2,k}) \right]. \label{pointwise_diff_alpha_beta2}
		\end{align}
		We sum \eqref{pointwise_diff_alpha_beta1} and \eqref{pointwise_diff_alpha_beta2} :
		\begin{align}\label{gap_alpha_beta}
			|\beta_k(t) - 2\alpha_k(t)| &\geq \frac{1}{2} \Big( |\beta_k(t_{1,k}) - 2\alpha_k(t_{1,k})| + |\beta_k(t_{2,k}) - 2\alpha_k(t_{2,k})| \Big) - C\|\g Y_k\|_{L^\infty_\xi(\bar{\Cr}_k)} \left[ \alpha_k(t_{1,k}) + \alpha_k(t_{2,k}) \right].
		\end{align}
		By the hypothesis \ref{restriction_Cr1} on $t_{1,k}$ and $t_{2,k}$, we have $\alpha_k(t_{1,k}) = o(|\beta_k(t_{1,k}) -2\alpha_k(t_{1,k})|)$ and $\alpha_k(t_{2,k}) = o(|\beta_k(t_{2,k}) -2\alpha_k(t_{2,k})|)$. Using this in \eqref{gap_alpha_beta}, we conclude that for any $t\in[t_{1,k},t_{2,k}]$,
		\begin{align*}
			\max(\alpha_k(t_{1,k}),\alpha_k(t_{2,k})) &\ust{k\to\infty}{=}o\Big( |\beta_k(t) - 2\alpha_k(t)| \Big).
		\end{align*}
		Thanks to \eqref{pointwise_alpha}, we obtain that for any $t\in[t_{1,k},t_{2,k}]$,
		\begin{align*}
			\max(\alpha_k(t_{1,k}),\alpha_k(t_{2,k})) &\ust{k\to\infty}{=} o\Big( \beta_k(t) + 2\max(\alpha_k(t_{1,k}),\alpha_k(t_{2,k})) \Big).
		\end{align*}
		So $\max(\alpha_k(t_{1,k}),\alpha_k(t_{2,k})) = o(\beta_k(t))$. Choosing $t$ such that $\beta_k(t) = \min_{[t_{1,k},t_{2,k}]} \beta_k$, we obtain \eqref{difference_t_theta}.\\
		To prove \eqref{uniform_arclength}, consider $t_{min},t_{max} \in [t_{1,k},t_{2,k}]$ such that 
		\begin{align*}
			\max_{[t_{1,k},t_{2,k}]} \beta_k = \beta_k(t_{max}), & & \min_{[t_{1,k},t_{2,k}]} \beta_k = \beta_k(t_{min}).
		\end{align*}
		If we integrate \eqref{derivative_difference_beta_alpha} on $[t_{min},t_{max}]$, we obtain thanks to \eqref{difference_t_theta} and \eqref{comparison_xit_xitheta} :
		\begin{align*}
			\left| \max_{[t_{1,k},t_{2,k}]} \beta_k - \min_{[t_{1,k},t_{2,k}]} \beta_k \right| &\ust{k\to \infty}{=} o\left( \min_{[t_{1,k},t_{2,k}]} \beta_k \right).
		\end{align*}
	\end{proof}
	
	Up to replace $\Cr_k$ by $[t_{1,k},t_{2,k}]\times \s^1$, we assume that $\alpha_k(-T_k) = o(\beta_k(-T_k))$ and $\alpha_k(T_k) = o(\beta_k(T_k))$, with $T_k \to +\infty$. We can now estimate the whole difference $\gamma_k$.
	
	\begin{lemma}\label{estimate_xistar}
		It holds
		\begin{align*}
			\max_{[-T_k,T_k]} \gamma_k \ust{k\to \infty}{=} o\left(\min_{[-T_k,T_k]} \beta_k \right).
		\end{align*}
	\end{lemma}
	
	\begin{proof}
		We compare $\gamma_k$ and $\alpha_k$ :
		\begin{align*}
			\gamma_k'(t) &= 2\int_{\s^1} \scal{\g(Y_k - Y_k^*)}{\dr_t \g(Y_k - Y_k^*)}_\xi \\
			&= 2\int_{\s^1} \scal{\dr_\theta Y_k}{ \dr^2_{t\theta} Y_k}_\xi + 2\int_{\s^1} \scal{\dr_t(Y_k - Y_k^*)}{\dr^2_{tt}(Y_k - Y_k^*)}_\xi \\
			&= \alpha_k'(t) + 2\int_{\s^1} \scal{\dr_t(Y_k - Y_k^*)}{\lap(Y_k - Y_k^*) - \dr^2_{\theta\theta} Y_k}_\xi .
		\end{align*}
		Using an integration by parts on the last term and $\lap Y_k^* = \fint_{\s^1} \dr^2_{tt} Y_k = \fint_{\s^1} \lap Y_k$, we obtain
		\begin{align*}
			\gamma_k'(t) &= \alpha_k'(t) - 2\int_{\s^1} \scal{\dr_t(Y_k - Y_k^*)}{ |\g Y_k|^2_\eta Y_k - \fint_{\s^1} |\g Y_k|^2_\eta Y_k }_\xi + 2\int_{\s^1} \scal{\dr^2_{t\theta} Y_k}{\dr_\theta Y_k}_\xi \\
			&= 2\alpha_k'(t) - 2\int_{\s^1} \scal{\dr_t(Y_k - Y_k^*)}{ |\g Y_k|^2_\eta Y_k - \fint_{\s^1} |\g Y_k|^2_\eta Y_k }_\xi .
		\end{align*}
		Thanks to \eqref{residue_hopf_diff},
		\begin{align}
			\left| \frac{d}{dt}\left( \gamma_k - 2\alpha_k \right) \right| &\leq C\|\g (Y_k-Y_k^*)\|_{L^\infty_\xi} \int_{\s^1} |\g Y_k|^2_\eta \nonumber \\
			&\leq C\|\g Y_k\|_{L^\infty_\xi} \int_{\s^1} |\dr_\theta Y_k|^2_\eta \nonumber \\
			&\leq C\|\g Y_k\|_{L^\infty_\xi} \alpha_k(t). \label{comparison_xistar_xitheta}
		\end{align}
		For any $t\in[-T_k,T_k]$, using \eqref{integral_xitheta}, we obtain by integration on $[-T_k,t]$ and $[t,T_k]$ :
		\begin{align}\label{uniform_estimate_gamma}
			|\gamma_k(t) - 2\alpha_k(t)| &\leq \min\Big( \gamma_k(-T_k) - 2\alpha_k(-T_k), \gamma_k(T_k) - 2\alpha_k(T_k) \Big) + C\left( \alpha_k(T_k) + \alpha_k(-T_k) \right). 
		\end{align}
		Thanks to \eqref{uniform_arclength}, we get $\gamma_k=o(\beta_k)$ uniformly in $[-T_k,T_k]$ if and only if $\gamma_k(T_k) = o(\beta_k(T_k))$ or $\gamma_k(-T_k) = o(\beta_k(-T_k))$.
		Considering $t = \pm T_k$ in \eqref{uniform_estimate_gamma}, we conclude that $\gamma_k(T_k) = o(\beta_k(T_k))$ if and only if $\gamma_k(-T_k) = o(\beta_k(-T_k))$.\\
		
		In clear, we have two possibilities : either $\gamma_k = o(\beta_k)$ uniformly on $[-T_k,T_k]$, or $\beta_k = O(\gamma_k)$ uniformly on $[-T_k,T_k]$.\\
		
		We show that the second case cannot occur considering the hypothesis \eqref{hypothesis_oscillation}. By contradiction, assume that $\beta_k = O(\gamma_k)$ on $[-T_k,T_k]$, that is to say, there exists $C>0$ such that for $k$ large enough : $C\gamma_k(t) \geq \beta_k(t)$ for any $t\in[-T_k,T_k]$. We proceed as in the lemma \ref{estimate_xitheta} : first, we find an equation for $\gamma_k$.
		\begin{align}
			\gamma_k''(t) &= \frac{d}{dt} \left( 2\int_{\s^1} \scal{\g(Y_k-Y_k^*)}{\dr_t \g(Y_k-Y_k^*)}_\xi \right) \nonumber \\
			&= 2\int_{\s^1} |\dr_t \g (Y_k - Y_k^*)|^2_\xi + 2\int_{\s^1} \scal{\g(Y_k - Y_k^*)}{\dr^2_{tt} \g(Y_k -Y_k^*)}_\xi \nonumber \\
			&= 2\int_{\s^1} |\dr_t \g (Y_k - Y_k^*)|^2_\xi + 2\int_{\s^1} \scal{\g(Y_k - Y_k^*)}{\g\left( \lap(Y_k-Y_k^*) - \dr^2_{\theta\theta} Y_k \right)}_\xi . \label{equation_xistar_computation}
		\end{align}
		We can estimate the last term using
		\begin{align*}
			\int_{\s^1} \scal{\g(Y_k - Y_k^*)}{\g\left( \lap(Y_k-Y_k^*) - \dr^2_{\theta\theta} Y_k \right)}_\xi &= \int_{\s^1} \scal{\g(Y_k - Y_k^*)}{\g \left( |\g Y_k|^2_\eta Y_k - \fint_{\s^1} |\g Y_k|^2_\eta Y_k \right)}_\xi\\
			& + \int_{\s^1} |\dr_\theta \g(Y_k-Y_k^*)|^2_\xi .
		\end{align*}
		As in \eqref{third_derivative}, we can bound :
		\begin{align*}
			\int_{\s^1} \scal{\g(Y_k - Y_k^*)}{\g\left( \lap(Y_k-Y_k^*) - \dr^2_{\theta\theta} Y_k \right)}_\xi &\geq \int_{\s^1} |\dr_\theta \g(Y_k-Y_k^*)|^2_\xi - C\|\g Y_k\|_{L^\infty_\xi} \int_{\s^1} |\g Y_k|^2_\xi .
		\end{align*}
		Therefore, in \eqref{equation_xistar_computation}, using Poincaré inequality in $\s^1$ :
		\begin{align}
			\gamma_k''(t) &\geq 2\int_{\s^1} |\g^2(Y_k-Y_k^*)|^2_\xi - C\|\g Y_k\|_{L^\infty_\xi} \int_{\s^1} |\g Y_k|^2_\xi \nonumber \\
			&\geq 2\gamma_k(t)  - C\|\g Y_k\|_{L^\infty_\xi} \beta_k(t). \label{equation_xistar}
		\end{align}
		Now, our assumption $\beta_k \leq C\gamma_k $ gives $\gamma_k''(t) \geq \gamma_k(t)$ for $k$ large enough. Hence, we are back with the same equation as in the lemma \ref{estimate_xitheta}, and we also have
		\begin{align*}
			\int_{-T_k}^{T_k} \beta_k(t)^{1/2} dt \leq C\int_{-T_k}^{T_k} \gamma_k(t)^{1/2} dt &\leq C\left[ \gamma_k(-T_k)^{1/2} + \gamma_k(T_k)^{1/2} \right].
		\end{align*}
		Using \eqref{oscillations}, the oscillations must vanish. Absurd. Hence, $\gamma_k = o(\beta_k)$. We conclude thanks to \eqref{uniform_estimate_gamma}.
	\end{proof}
	
	We show the convergence of $Y_k^*$ to a straight line.
	
	\begin{proposition}\label{convergence_geodesic}
		Under the assumptions \eqref{uniform_Linfty_bound}-\eqref{unif_conv_to_zero}-\eqref{hypothesis_oscillation}, there exists a geodesic $\sigma$ of $\s^{3,1}$, and a parametrization $\tau_k : (-T_k,T_k) \to (-L_k,L_k)$ such that
		\begin{align*}
			Y_k^*\circ \tau_k^{-1} \xrightarrow[k\to \infty]{}{\sigma},
		\end{align*}
		with $|\sigma'|_\xi = 1$. The curve $\sigma$ is a timelike straight line. The convergence holds in the $C^2_{loc}$-topology.
	\end{proposition}
	
	\begin{proof}
		We show that if we reparametrize $(Y_k^*)_{k\in\N}$ by euclidean arclength, it converges to a geodesic. Note that $\dr_t Y_k^*$ doesn't vanish. Indeed, $\beta_k$ doesn't vanish since the critical points of a conformal Gauss map coming from a Willmore surface are isolated. Hence, we can use the lemma \ref{estimate_xistar} to obtain uniformly on $\bar{\Cr}_k$ :
		\begin{align}
			\left| \frac{ |\dr_t Y_k^*|^2_\xi }{\fint_{\s^1} |\dr_t Y_k|^2_\xi } - 1 \right| &= \left| \frac{ \int_{\s^1} |\dr_t Y_k^*|^2_\xi - |\dr_t Y_k|^2_\xi }{ \fint_{\s^1} |\dr_t Y_k|^2_\xi } \right| \nonumber \\
			&= \left| \frac{ \int_{\s^1} \scal{\dr_t(Y_k^* - Y_k)}{\dr_t Y_k}_\xi + |\dr_t(Y_k^* - Y_k)|^2_\xi }{ \fint_{\s^1} |\dr_t Y_k|^2_\xi } \right| \nonumber \\
			&\leq C\left(  \frac{ \gamma_k }{ \beta_k } + \sqrt{ \frac{\gamma_k}{\beta_k} } \right) \xrightarrow[k\to \infty]{}{0}. \label{uniform_close_average}
		\end{align}
		Now, we can consider the equation followed by $Y_k^*$ :
		\begin{align*}
			\frac{d^2 Y_k^*}{ds^2} &= \frac{1}{|\dr_t Y_k^*|_\xi} \frac{d}{dt}\left( \frac{1}{|\dr_t Y_k^*|_\xi} \frac{d Y_k^*}{dt} \right) \\
			&= \frac{1}{|\dr_t Y_k^*|_\xi^2} \frac{d^2 Y_k^*}{dt^2}  - \frac{1}{|\dr_t Y_k^*|_\xi^4} \scal{\dr_t Y_k^*}{\dr^2_{tt} Y_k^*}_\xi \frac{d Y_k^*}{dt}.
		\end{align*}
		We compute the second derivative
		\begin{align*}
			\frac{d^2 Y_k^*}{dt^2} &= \fint_{\s^1} \dr^2_{tt} Y_k (t,\theta) d\theta\\
			&= \fint_{\s^1} \lap Y_k(t,\theta) d\theta \\
			&= -\fint_{\s^1} |\g Y_k(t,\theta)|^2_\eta Y_k(t,\theta) d\theta.
		\end{align*}
		Since $(Y_k)_k$ is bounded in $L^\infty$ on $\Cr_k$, it follows by \eqref{residue_hopf_diff} :
		\begin{align*}
			\left| \dr^2_{ss} Y_k^* \right|_\xi &\leq C \frac{\int_{\s^1} |\g Y_k|^2_\eta }{|\dr_t Y_k^*|^2_\xi} = C\frac{\int_{\s^1} |\dr_\theta Y_k|^2_\eta }{|\dr_t Y_k^*|^2_\xi }.
		\end{align*}
		Thanks to \eqref{difference_t_theta} :
		\begin{align}\label{laplacian_average}
			\| \dr^2_{ss} Y_k^* \|_{L^\infty_\xi(\bar{\Cr}_k)} \leq C \sup_{[-T_k,T_k]} \frac{\alpha_k }{\beta_k} \xrightarrow[k\to \infty]{}{0}.
		\end{align}
		Therefore, the limit curve $\sigma := \displaystyle{\lim_{k\to \infty}} Y_k^*$, parametrized by euclidean arclength, exists in $C^{1,a}$ for every $a\in(0,1)$.\\
		It satisfies $\sigma''=0$ in the weak sense. Hence, $\sigma$ is a straight line, and \eqref{laplacian_average} shows that the convergence is actually in the $C^2$-topology.
	\end{proof}
	
	\begin{remark}[About the lines in $\s^{3,1}$]
		The only vectors $a,b\in\R^5$ such that the line $(t\mapsto at+b)$ is contained in $\s^{3,1}$, are exactly those satisfying $|a|^2_\eta = 0$, $|b|^2_\eta = 1$ and $\scal{a}{b}_\eta = 0$. Indeed if for any $t\in\R$ we have $|at+b|^2_\eta = 1 = |a|^2_\eta t^2 + 2t\scal{a}{b}_\eta + |b|^2_\eta$, then we can identify two polynomials.
	\end{remark} 
	
	We prove the convergence in the $C^1$-topology of $Y_k$ to a straight line thanks to the lemma \ref{estimate_xistar}.
	
	\begin{proposition}\label{proposition_C1_convergence}
		Under the assumptions \eqref{uniform_Linfty_bound}-\eqref{unif_conv_to_zero}-\eqref{hypothesis_oscillation}, there exists a geodesic $\sigma$ of $\s^{3,1}$, and a parametrization $\tau_k : (-T_k,T_k)\times \s^1 \to (-L_k,L_k)\times \s^1$ such that
		\begin{align*}
			Y_k\circ \tau_k^{-1} \xrightarrow[k\to \infty]{}{\sigma},
		\end{align*}
		with $|\sigma'|_\xi = 1$. The curve $\sigma$ is a timelike straight line. The convergence holds in the $C^1_{loc}$-topology. 
	\end{proposition}
	
	\begin{proof}
		We first prove the convergence in the $C^1$-topology. The system satisfied by $Y_k -Y_k^*$ is the following : at a point $(t,\theta)\in \Cr_k$,
		\begin{align}\label{system_Y_Ystar}
			\lap(Y_k-Y_k^*)(t,\theta) &= -|\g Y_k(t,\theta)|^2_\eta Y_k(t,\theta) + \fint_{\s^1} |\g Y_k(t,\bar{\theta})|^2_\eta Y_k(t,\bar{\theta}) d\bar{\theta}.
		\end{align}
		By elliptic regularity, we obtain for any $Q(t_k,4)\subset \bar{\Cr}_k$ :
		\begin{align}
			\|\g(Y_k - Y_k^*)\|_{L^\infty_\xi(Q(t_k,2))} &\leq \left\|-|\g Y_k|^2_\eta Y_k + \fint_{\s^1} |\g Y_k(\cdot,\bar{\theta})|^2_\eta Y_k(\cdot,\bar{\theta}) d\bar{\theta} \right\|_{L^4_\xi(Q(t_k,3))} \nonumber\\
			&+ \| Y_k - Y_k^*\|_{L^4_\xi( Q(t_k,3))} \nonumber \\
			&\leq C\left\||\g Y_k|^2_\eta (Y_k-Y_k^*) \right\|_{L^4_\xi(Q(t_k,3))} + C\left\| |\g Y_k|^2_\eta - |\g Y_k^*|^2_\eta \right\|_{L^4(Q(t_k,3))} \label{conv_C1_lign1}\\
			&+ \| Y_k - Y_k^*\|_{L^4_\xi(Q(t_k,3))}. \label{conv_C1_lign2}
		\end{align}
		Thanks to \eqref{residue_hopf_diff}, the lemma \ref{step1_conv_geodesic} and the $\ve$-regularity, we can bound the first term of \eqref{conv_C1_lign1} by
		\begin{align*}
			\left\||\g Y_k|^2_\eta (Y_k-Y_k^*) \right\|_{L^4_\xi(Q(t_k,3))} &= \left( \int_{Q(t_k,3)} |\g Y_k|^8_\eta |Y_k-Y_k^*|^4_\xi \right)^\frac{1}{4}\\
			&\leq C\left( \int_{Q(t_k,4)} |\dr_\theta Y_k|^2_\xi \right)^3.
		\end{align*}
		So
		\begin{align}\label{conv_C1_lign11}
			\left\||\g Y_k|^2_\eta (Y_k-Y_k^*) \right\|_{L^4_\xi(Q(t_k,3))} &\ust{k\to\infty}{=} o\left( \min_{[-T_k,T_k]} \beta_k \right).
		\end{align}
		For the second term of \eqref{conv_C1_lign1}, it holds 
		\begin{align*}
			\Big| |\g Y_k|^2_\eta - |\g Y_k^*|^2_\eta \Big| &= \Big| |\g Y_k - \g Y_k^*|^2_\eta + 2\scal{\g(Y_k-Y_k^*)}{\g Y_k^*}_\eta \Big|\\
			&\leq |\g Y_k - \g Y_k^*|^2_\xi + |\g Y_k-\g Y_k^*|_\xi |\g Y_k^*|_\xi.
		\end{align*}
		So
		\begin{align}\label{conv_C1_lign12}
			\left\| |\g Y_k|^2_\eta - |\g Y_k^*|^2_\eta \right\|_{L^4(Q(t_k,3))} &\ust{k\to\infty}{=}o\left( \min_{[-T_k,T_k]} \beta_k \right).
		\end{align}
		For the term \eqref{conv_C1_lign2}, it holds : 
		\begin{align*}
			\| Y_k - Y_k^*\|_{L^4_\xi( Q(t_k,3))}^4 &\leq C \int_{Q(t_k,3)} \left( \int_{\s^1} |\dr_\theta Y_k(t,\bar{\theta})|_\xi d\bar{\theta} \right)^4 dtd\theta \\
			&\leq C\int_{t_k-3}^{t_k+3} \left( \int_{\s^1} |\dr_\theta Y_k(t,\bar{\theta})|_\xi^2 d\bar{\theta} \right)^2 dt \\
			&\leq C\int_{t_k-3}^{t_k+3} \alpha_k(t)^2 dt.
		\end{align*}
		So
		\begin{align}\label{conv_C1_lign21}
			\| Y_k - Y_k^*\|_{L^4_\xi(Q(t_k,3))}&\ust{k\to\infty}{=} o\left( \min_{[-T_k,T_k]} \beta_k^\frac{1}{2} \right).
		\end{align}
		Using the estimates \eqref{conv_C1_lign11}, \eqref{conv_C1_lign12} and \eqref{conv_C1_lign21} in the estimates \eqref{conv_C1_lign1} and \eqref{conv_C1_lign2}, we obtain 
		\begin{align*}
			\|\g(Y_k - Y_k^*)\|_{L^\infty_\xi(Q(t_k,2))} &\ust{k\to\infty}{=} o\left( \min_{[-T_k,T_k]} \beta_k^\frac{1}{2} \right).
		\end{align*}
		Thanks to \eqref{uniform_close_average}, it holds
		\begin{align}\label{speed_conv_geod}
			\sup_{(t_k,\theta_k)\in\bar{\Cr}_k} \left( \frac{|\dr_t Y_k(t_k,\theta_k) - \dr_t Y_k^*(t_k)|_\xi}{|\dr_t Y_k^*(t_k)|_\xi} + \frac{|\dr_\theta Y_k(t_k,\theta_k)|_\xi}{|\dr_t Y_k^*(t_k)|_\xi} \right) \ust{k\to\infty}{=} o(1).
		\end{align}
	\end{proof}
	
\begin{remark}\label{rk:unif_gradient_Y}
If we combine the convergence \eqref{speed_conv_geod}, \eqref{uniform_close_average} and \eqref{uniform_arclength}, we conclude that 
\begin{align*}
	\left| \frac{\inf_{\tilde{\Cr_k}} |\g Y_k|_\xi }{ \sup_{\tilde{\Cr_k}} |\g Y_k|_\xi } - 1 \right| \xrightarrow[k\to\infty]{}{0}.
\end{align*}
Furthermore, since the limit geodesic is light-like, it holds :
\begin{align}\label{eq:conv_lightlike}
	\left\| \frac{ |\g Y_k|_\eta }{ |\g Y_k|_\xi } \right\|_{L^\infty(\tilde{\Cr_k})} \xrightarrow[k\to\infty]{}{0}.
\end{align}
\end{remark}
	
	\subsection{Convergence to a geodesic in the $C^2$-topology}
	
	Since the change of parametrization degenerates by definition, the convergence in the $C^2$-topology doesn't follow from a direct bootstrap argument. As in the lemma \ref{estimate_xitheta}, we derive an equation for the quantity $\int_{\s^1} |\dr_\theta \g Y_k|^2_\xi d\theta$ :
	\begin{lemma}\label{estimate_delta}
		Given $t\in [-T_k,T_k]$ and $k\in\N$, let $\delta_k(t) := \int_{\s^1} |\dr_\theta \g Y_k|^2_\xi d\theta$. For $k$ large enough,
		\begin{align*}
			\forall t\in [-T_k,T_k], \ \ \ \delta_k''(t)& \geq \delta_k(t).
		\end{align*}
		For any $t_1<t<t_2$, we can estimate
		\begin{align}\label{pointwise_delta}
			\delta_k(t) &\leq \max(\delta_k(t_1),\delta_k(t_2)).
		\end{align}
		Moreover, for any $\nu\in(0,1]$, if $|t_1-t_2|>1$ :
		\begin{align}\label{integral_delta}
			\int_{t_1}^{t_2} \delta_k(t)^\nu dt &\leq C(\nu) \left[ \delta_k(t_1)^\nu + \delta_k(t_2)^\nu \right].
		\end{align}
	\end{lemma}
	
	\begin{proof}
		We differentiate twice :
		\begin{align*}
			\frac{d^2 \delta_k}{dt^2} &= 2 \frac{d}{dt} \int_{\s^1} \scal{\dr_\theta \g Y_k}{\dr^2_{\theta t} \g Y_k}_\xi d\theta \\
			&= 2\int_{\s^1} |\dr^2_{\theta t} \g Y_k|^2_\xi + \scal{\dr_\theta \g Y_k}{\dr^3_{\theta t t} \g Y_k}_\xi \ d\theta \\
			&= 2\int_{\s^1} |\dr^2_{\theta t} \g Y_k|^2_\xi - \scal{\dr^2_{\theta\theta} \g Y_k}{\dr^2_{t t}\g Y_k}_\xi \ d\theta \\
			&= 2\int_{\s^1} |\dr^2_{\theta t} \g Y_k|^2_\xi - \scal{\dr^2_{\theta\theta} \g Y_k}{\g \left[\lap Y_k - \dr^2_{\theta\theta} Y_k \right]}_\xi \ d\theta .
		\end{align*}
		Therefore,
		\begin{align*}
			\frac{d^2 \delta_k}{dt^2} &=2\int_{\s^1} |\dr_\theta \g^2 Y_k|^2_\xi + \scal{\dr^2_{\theta\theta} \g Y_k}{ \g \left[|\g Y_k|^2_\eta Y_k \right] }_\xi \ d\theta.
		\end{align*}
		The last term can be bounded by Cauchy-Schwarz :
		\begin{align*}
			\frac{d^2 \delta_k}{dt^2} &\geq 2\int_{\s^1} |\dr_\theta \g^2 Y_k|^2_\xi d\theta - \int_{\s^1} |\dr^2_{\theta\theta}\g Y_k|_\xi \left( |\g Y_k|^2_\eta |\g Y_k|_\xi + C|\g Y_k|_\xi |\g^2 Y_k|_\xi \right) d\theta \\
			&\geq (2 - C\|\g Y_k\|_{L^\infty_\xi}^2 ) \int_{\s^1} |\dr_\theta \g^2 Y_k|^2_\xi d\theta - C\|\g Y_k\|_{L^\infty_\xi}^2 \int_{\s^1} |\g Y_k|^2_\xi + |\g^2 Y_k|^2_\xi \ d\theta.
		\end{align*}
		Using Poincaré inequality on the last term, we obtain
		\begin{align*}
			\frac{d^2 \delta_k}{dt^2} &\geq(2 - C\|\g Y_k\|_{L^\infty_\xi}^2 ) \int_{\s^1} |\dr_\theta \g^2 Y_k|^2_\xi d\theta - C\|\g Y_k\|_{L^\infty_\xi}^2 \int_{\s^1} |\dr_\theta \g Y_k|^2_\xi + |\dr_\theta \g^2 Y_k|^2_\xi \ d\theta \\
			&\geq (2 - C\|\g Y_k\|_{L^\infty_\xi}^2 ) \int_{\s^1} |\dr_\theta \g^2 Y_k|^2_\xi d\theta - C\|\g Y_k\|_{L^\infty_\xi}^2 \int_{\s^1} |\dr_\theta \g Y_k|^2_\xi\ d\theta.
		\end{align*}
		Using Poincaré inequality on the first term, we obtain
		\begin{align*}
			\frac{d^2 \delta_k}{dt^2} &\geq (2 - C\|\g Y_k\|_{L^\infty_\xi}^2 ) \int_{\s^1} |\dr_\theta \g Y_k|^2_\xi d\theta
		\end{align*}
		For $k$ large enough, we obtain :
		\begin{align*}
			\frac{d^2 \delta_k(t)}{dt^2} &\geq \int_{\s^1} |\dr_\theta \g Y_k|^2_\xi d\theta = \delta_k(t).
		\end{align*}
		The rest of the proof is similar to the lemma \ref{estimate_xitheta}.
	\end{proof}
	
	We can now compare $\delta_k$ to the arclength of $Y_k^*$, which is equivalent to $\beta_k$. Thanks to \eqref{integral_delta}, we obtain
	\begin{align*}
		\int_{-T_k}^{T_k} \delta_k(t)^\frac{1}{4} dt &\ust{k\to\infty}{=} o\left( \int_{-T_k}^{T_k} \beta_k(t)^\frac{1}{2} dt \right).
	\end{align*}
	Hence, as for the proof of the claim \ref{claim_boundary_conditions_C1}, we can consider $-T_k<t_{1,k}<t_{2,k}<T_k$ satisfying \ref{restriction_Cr3} and 
	\begin{align*}
		\delta_k(t_{1,k}) \ust{k\to\infty}{=} o(\beta_k(t_{1,k})^2), & & 	\delta_k(t_{2,k}) \ust{k\to\infty}{=} o(\beta_k(t_{2,k})^2).
	\end{align*}
	Thanks to \eqref{pointwise_delta} and \eqref{uniform_arclength}, we conclude that
	\begin{align}\label{uniform_delta_beta2}
		\max_{[ t_{1,k},t_{2,k} ]} \delta_k &\ust{k\to\infty}{=} o\left( \min_{[ t_{1,k},t_{2,k} ]} \beta_k^2 \right). 
	\end{align}
	As for the $C^1$-convergence, we replace $[-T_k,T_k]$ by $[t_{1,k},t_{2,k}]$ if needed. We show that the second derivatives $\g^2 (Y_k-Y_k^*)$ are small compared to the arclength $|\g Y_k^*|^2_\xi$. 
	
	\begin{lemma}\label{uniform_bound_hessian}
		It holds 
		\begin{align*}
			\sup_{[-T_k+4,T_k-4]\times\s^1} \frac{|\g^2 Y_k |_\xi}{|\g Y_k^*|^2_\xi } \xrightarrow[k\to\infty]{}{0}.
		\end{align*}
	\end{lemma}
	
	\begin{proof}
		By differentiating \eqref{system_Y_Ystar}, we obtain 
		\begin{align*}
			\lap \g\left[ Y_k - Y_k^*\right](t,\theta) &= \g\left[ -|\g Y_k(t,\theta)|^2_\eta (Y_k(t,\theta)-Y_k^*(t)) + \fint_{\s^1} \left( |\g Y_k(t,\theta)|^2_\eta - |\g Y_k(t,\bar{\theta})|^2_\eta \right) Y_k(t,\bar{\theta}) d\bar{\theta} \right].
		\end{align*}
		Therefore, we have the following estimate :
		\begin{align*}
			\left| \lap \g\left[ Y_k - Y_k^*\right] \right|_\xi &\leq |\g Y_k|^2_\eta |\g(Y_k - Y_k^*)|_\xi + C|\scal{\g Y_k}{\g^2 Y_k}_\eta|\\
			&+ C\fint_{\s^1} \left( |\scal{\g Y_k}{\g^2 Y_k}_\eta| + |\g Y_k|^2_\eta |\g Y_k|_\xi \right) .
		\end{align*}
		By elliptic regularity, it holds for any $Q(t,4)\subset [-T_k,T_k]\times \s^1$ :
		\begin{align}
			\left\| \g^2 (Y_k - Y_k^*) \right\|_{L^\infty_\xi(Q(t,2))} &\leq C \left\| |\g Y_k|^2_\eta |\g(Y_k - Y_k^*)|_\xi \right\|_{L^4(Q(t,3))} \label{elliptic_reg_11}\\
			&+ C\left\| \scal{\g Y_k}{\g^2 Y_k}_\eta \right\|_{L^4(Q(t,3))} \label{elliptic_reg_12}\\
			&+ C\|\g(Y_k - Y_k^*) \|_{L^4_\xi(Q(t,3))}. \label{elliptic_reg_13}
		\end{align}
		The boundary term \eqref{elliptic_reg_13} can be estimated in the same manner as \eqref{conv_C1_lign2} :
		\begin{align}\label{elliptic_reg_23}
			\|\g(Y_k - Y_k^*) \|_{L^4_\xi(Q(t,3))} &\leq C\left( \int_{t-3}^{t+3} \delta_k^2 \right)^\frac{1}{4} \ust{k\to\infty}{=} o\left(\min \beta_k \right).
		\end{align}
		The term \eqref{elliptic_reg_11} can be estimated thanks to the $C^1$-convergence, see the proposition \ref{proposition_C1_convergence} :
		\begin{align}\label{elliptic_reg_21}
			\left\| |\g Y_k|^2_\xi |\g(Y_k - Y_k^*)|_\xi \right\|_{L^4(Q(t,2))} &\leq \left\| \g Y_k \right\|_{L^\infty_\xi(Q(t,2))}^2 \left\| \g Y_k - \g Y_k^* \right\|_{L^\infty_\xi(Q(t,2))} \ust{k\to\infty}{=} o(\min\beta_k ).
		\end{align}
		The term \eqref{elliptic_reg_12} is invariant by isometries of $\s^{3,1}$. Therefore, it can be estimated thanks to the $\ve$-regularity, see the corollary \ref{control_normal_by_arond} :
		\begin{align*}
			\left\| \scal{\g Y_k}{\g^2 Y_k}_\eta \right\|_{L^4(Q(t,3))} &\leq C\|\g Y_k\|_{L^2_\eta(Q(t,4))}^2.
		\end{align*}
		Since $Y_k$ is conformal,
		\begin{align}\label{elliptic_reg_22}
			\left\| \scal{\g Y_k}{\g^2 Y_k}_\eta \right\|_{L^4_\xi(Q(t,3))} &\leq C\|\dr_\theta Y_k\|_{L^2_\xi(Q(t,4))}^2 \ust{k\to\infty}{=} o(\min\beta_k).
		\end{align}
		Using \eqref{elliptic_reg_21}-\eqref{elliptic_reg_22}-\eqref{elliptic_reg_23} in \eqref{elliptic_reg_11}-\eqref{elliptic_reg_12}-\eqref{elliptic_reg_13}, we obtain
		\begin{align*}
			\left\| \g^2 (Y_k - Y_k^*) \right\|_{L^\infty_\xi(Q(t,2))} &\ust{k\to\infty}{=} o(\min(\beta_k)).
		\end{align*}
		Thanks to \eqref{uniform_close_average}, we conclude that 
		\begin{align*}
			\sup_{p_k\in[-T_k+4,T_k-4]\times\s^1} \frac{|\g^2 (Y_k-Y_k^*)(p_k)|_\xi}{|\g Y_k^*(p_k)|^2_\xi} \xrightarrow[k\to\infty]{}{0}.
		\end{align*}
	By \cref{convergence_geodesic}, it holds 
	\begin{align*}
		\sup_{p_k\in[-T_k+4,T_k-4]\times\s^1} \frac{|\g^2 Y_k^*(p_k)|_\xi}{|\g Y_k^*(p_k)|^2_\xi} \xrightarrow[k\to\infty]{}{0}.
	\end{align*}
So we conclude.
	\end{proof}
	
	We can now conclude the proof of the convergence in the $C^2$-topology.
	\begin{proposition}
		Under the assumptions \eqref{uniform_Linfty_bound}-\eqref{unif_conv_to_zero}-\eqref{hypothesis_oscillation}, there exists a geodesic $\sigma$ of $\s^{3,1}$, and a parametrization $\tau_k : (-T_k,T_k)\times \s^1 \to (-L_k,L_k)\times \s^1$ such that
		\begin{align*}
			Y_k\circ \tau_k^{-1} \xrightarrow[k\to \infty]{}{\sigma},
		\end{align*}
		with $|\sigma'|_\xi = 1$. The curve $\sigma$ is a timelike straight line. The convergence holds in the $C^2_{loc}$-topology. 
	\end{proposition}
	
	\begin{proof}
		Thanks to the proposition \ref{proposition_C1_convergence}, we have the convergence in the $C^1$-topology. It remains to show that $ \g^2 (Y_k\circ \tau_k^{-1}) $ converges to $0$ uniformly. The new parametrization doesn't act on the angular parametrization. Hence, $	\dr^2_{\theta\theta} (Y_k\circ \tau_k^{-1}) = (\dr^2_{\theta\theta} Y_k)\circ \tau_k^{-1} $ uniformly converges to $0$ thanks to \eqref{unif_conv_to_zero}. For the mixed derivatives, it holds
		\begin{align*}
			\dr^2_{\theta s} Y_k &= \frac{1}{|\g Y_k^*|_\xi} \dr^2_{\theta t} Y_k = |\g Y_k^*|_\xi \frac{\dr^2_{t \theta} (Y_k - Y_k^*)}{|\g Y_k^*|^2_\xi}.
		\end{align*}
		Thanks to the lemma \ref{uniform_bound_hessian} and \eqref{unif_conv_to_zero}, we obtain that $\| \dr^2_{\theta s} Y_k \|_{L^\infty_\xi }$ converges to $0$. For the purely radial part of the hessian, it holds 
		\begin{align*}
			\dr^2_{ss} Y_k &= \dr^2_{ss} (Y_k - Y_k^*) + \dr^2_{ss} Y_k^*\\
			&= \frac{1}{|\dr_t Y_k^*|_\xi} \dr_t \left( \frac{1}{|\dr_t Y_k^*|_\xi} \dr_t (Y_k-Y_k^*) \right) + \dr^2_{ss} Y_k^*\\
			&= \frac{\dr^2_{tt} (Y_k - Y_k^*)}{|\dr_t Y_k^*|^2_\xi} -\frac{1}{|\dr_t Y_k^*|^3_\xi} \scal{\dr_t Y_k^*}{\dr^2_{tt} Y_k^*}_\xi \dr_t (Y_k-Y_k^*) + \dr^2_{ss} Y_k^*.
		\end{align*}
		Thanks to the lemma \ref{uniform_bound_hessian} and the proposition \ref{convergence_geodesic}, we obtain that $\dr^2_{ss} Y_k$ converges uniformly to $0$.
	\end{proof}

	\section{Proof of the proposition \ref{equality_index}}\label{proof_equality_index}

	Let $\Psi : \Sigma\to\s^3$ be a smooth immersion, not totally umbilic. Consider $Y$ its conformal Gauss map. Given a vector field $Z\in T_Y\s^{3,1}$, the goal of this section, see the proposition \ref{caracterization_CGM_variation}, is to show that $Z$ satisfies a system of two equations if and only if there exists a variation $(Y_t)_{t\in(-1,1)}$ satisfying two conditions : each $Y_t$ is a conformal Gauss map of some immersion $\Psi_t$ with $(\dr_t Y_t)_{|t=0} = Z$ and $(\dr_t \Psi_t)_{|t=0} \perp \g \Psi$.\\
	
	We denote $\Ur := \{x\in\Sigma : |\g Y|_\eta = 0\}$ the set of umbilic points of $\Psi$. Given a function $f_t(x)$, we denote $\dot{f} := (\dr_t f_t)_{|t=0}$.\\
	
	According to \cite[Proposition 3.3]{palmer1991}, a vector field $Z\in T_Y \s^{3,1}$ has to satisfy the following condition in order to come from a variation through conformal Gauss maps. Let $\xi$ be the normal part of $Z$ along $Y(\Sigma\setminus \Ur)$. Then we can choose $(Y_t)_{t\in(-1,1)}$ a variation through conformal Gauss maps such that $\dot{Y}=Z$ if and only if 
	\begin{align}\label{variation_through_CGM_equation}
		\scal{\lap^{NY} \xi +2\xi}{\nu}_\eta + \scal{\g^{TY} \xi}{\g^{TY} \nu}_\eta = 0,
	\end{align}
	where $\lap^{NY}$ is the rough laplacian on $NY := (T Y(\Sigma\setminus \Ur))^\perp$ computed with the metric $g_Y := Y^*\eta$, $\g^{TY}$ is the tangent part of $\g^{g_Y}$ to $Y(\Sigma\setminus \Ur)$, and $\nu := \begin{pmatrix}
		\Psi\\ 1
	\end{pmatrix}$. However, \eqref{variation_through_CGM_equation} is heavily degenerated around umbilic points of $\Psi$ since every part of the equation depends on the geometry of $Y$. Instead of considering the minimal surface's point of view, we consider the harmonic map's point of view. The equation \eqref{variation_through_CGM_equation} is the linearisation of the condition that the mean curvature $h_\gamma \in \R^{4,1}$ of the conformal Gauss map $\gamma$ of a smooth immersion $\psi :\Sigma\to \s^3$ satisfies, away from the umbilic points of $\psi$ :
	\begin{align*}
		h_{\gamma} &= (\lap_{g_\gamma} H_\psi + 2H_\psi) \begin{pmatrix}
			\psi\\ 1
		\end{pmatrix}, 
	\end{align*}
	where $g_\gamma := \gamma^*\eta$. Let $g_\psi := \psi^*\xi$. Using that $\gamma : (\Sigma,g_\psi) \to \s^{3,1} $ is conformal, the above equation can be written as
	\begin{align}\label{mean_curvature_CGM}
		\lap_{g_\psi} \gamma + |\g \gamma|^2_\eta \gamma &= (\lap_{g_\psi} H_\psi + |\g \gamma|^2_\eta H_\psi) \begin{pmatrix}
			\psi\\ 1
		\end{pmatrix}.
	\end{align}
	Since $\psi$ is smooth, this equation makes sense at umbilic points of $\psi$. This is the constraint that we linearize. 
	
	\begin{lemma}
		Let $\Psi : \Sigma \to \s^3$ be a smooth immersion and $Y$ be its conformal Gauss map. Let $(\Psi_t)_{t\in(-1,1)}$ be a smooth deformation of $\Psi$ such that $\dot{\Psi} \perp \g \Psi$. Consider $(Y_t)_{t\in(-1,1)}$ their conformal Gauss map. It holds
		\begin{align}\label{LCGM_perp_variation}
			\scal{\nu}{\lap_{g_\Psi} \dot{Y} - |\g Y|^2_\eta \dot{Y} }_\eta &= 0.
		\end{align}
	\end{lemma}
	
	\begin{proof}
		Let $\nu_t := \begin{pmatrix}
			\Psi_t \\ 1
		\end{pmatrix}$. Thanks to \eqref{mean_curvature_CGM}, it holds :
		\begin{align*}
			0 &= \frac{d}{dt} \left( \scal{ \nu_t }{ \lap_{g_{\Psi_t}} Y_t }_\eta \right)_{|t=0} \\
			&= \scal{\dot{\nu}}{\lap_{g_\Psi} Y}_\eta + \scal{\nu}{ \frac{\dr}{\dr t}\left( \lap_{g_{\Psi_t}} Y_t \right)_{|t=0} }_\eta\\
			&= \scal{\dot{\nu}}{\lap_{g_\Psi} Y}_\eta + \scal{ \nu }{  \lap_{g_{\Psi}} \dot{Y} - \frac{1}{2\sqrt{\det g_{\Psi} }} \tr_{g_\Psi}(\dot{g}_\Psi) \dr_i \left[ g_\Psi^{ij} \sqrt{\det g_\Psi} \dr_j Y \right] }_\eta \\
			&+ \scal{ \nu }{ -\frac{1}{\sqrt{\det g_\Psi}} \dr_i\left[ g_\Psi^{i\alpha} (\dot{g}_\Psi)_{\alpha\beta} g_\Psi^{\beta j} \sqrt{\det g_\Psi} \dr_j Y \right] + \frac{1}{2\sqrt{\det g_{\Psi} }} \dr_i \left[ g_\Psi^{ij} \tr_{g_\Psi}(\dot{g}_\Psi) \sqrt{\det g_\Psi} \dr_j Y \right] }_\eta .
		\end{align*}
		Since $\lap_{g_\Psi} Y \perp \nu$ and $\g Y\perp \nu$, the above equation reduces to :
		\begin{align}\label{LCGM_step1}
			0 &= \scal{\dot{\nu}}{\lap_{g_\Psi} Y}_\eta + \scal{ \nu }{ \lap_{g_{\Psi}} \dot{Y} - \scal{\dot{g}_\Psi}{\g^2 Y}_{g_\Psi} }_\eta.
		\end{align}
		Using \eqref{derivatives_Y}, we compute the hessian of $Y$ :
		\begin{align*}
			\g^2 Y &= (\g^2 H) \nu + (\g H - \g \Arond)(\g \nu) - \Arond \g^2 \nu.
		\end{align*}
		Since $\nu,\g \nu\perp \nu$, we obtain 
		\begin{align*}
			\scal{\nu}{\g^2 Y}_\eta &= -\scal{\nu}{\Arond \g^2 \nu}_\eta \\
			&= \scal{\g \nu}{\Arond \g \nu}_\eta \\
			&= \Arond.
		\end{align*}
		Therefore, \eqref{LCGM_step1} reduces to
		\begin{align}\label{LCGM_step2}
			0 &= \scal{\dot{\nu}}{\lap_{g_\Psi} Y}_\eta + \scal{ \nu }{ \lap_{g_{\Psi}} \dot{Y}}_\eta - \scal{\dot{g}_\Psi}{\Arond}_{g_\Psi} .
		\end{align}
		Using $\dot{\Psi}\perp \g \Psi$, we obtain
		\begin{align*}
			(\dot{g}_\Psi)_{ij} &= \scal{\dr_i \dot{\Psi} }{\dr_j \Psi}_\xi + \scal{\dr_i \Psi}{\dr_j \dot{\Psi}}_\xi \\
			&= -2 \scal{\dot{\Psi} }{\dr^2_{ij} \Psi}_\xi \\
			&= - 2 \scal{\dot{\Psi}}{N}_\xi A_{ij}.
		\end{align*}
		We now write $\scal{\dot{\Psi}}{N}_\xi$  in terms of $\dot{Y}$. By definition of a conformal Gauss map and using $\nu\perp \dot{\nu}$, it holds 
		\begin{align*}
			\scal{\dot{\nu}}{Y}_\eta &= \scal{ \begin{pmatrix}
					\dot{\Psi} \\ 0
			\end{pmatrix} }{ \begin{pmatrix}
					N\\ 0
			\end{pmatrix} }\\
			&= \scal{\dot{\Psi}}{N}_\xi.
		\end{align*}
		Hence, \eqref{LCGM_step2} can be written as 
		\begin{align}\label{LCGM_step3}
			0 &= \scal{\dot{\nu}}{\lap_{g_\Psi} Y}_\eta + \scal{\nu}{\lap_{g_\Psi} \dot{Y}}_\eta +2\scal{\dot{\nu}}{Y}_\eta |\Arond|^2_{g_\Psi}.
		\end{align}
		Furthermore, we can write the first term as
		\begin{align*}
			\scal{\dot{\nu}}{\lap_{g_\Psi} Y}_\eta &= \scal{\dot{\nu} }{(\lap_{g_\Psi} H) \nu +(\g H - \di \Arond)\g \nu - \Arond \cdot (\g^2 \nu) }_\eta.
		\end{align*}
		Using $\dot{\Psi}\perp \g \Psi$, we conclude $\dot{\nu}\perp \g \nu$ and we obtain
		\begin{align*}
			\scal{\dot{\nu}}{\lap_{g_\Psi} Y}_\eta &= - \scal{\dot{\nu}}{\Arond \cdot \g^2 \nu}_\eta \\
			&= -\Arond^{\alpha\beta} \scal{\dot{\Psi}}{\g_{\alpha\beta} \Psi}_\xi\\
			&= -\scal{\dot{\Psi}}{N}_\xi |\Arond|^2_{g_\Psi}.
		\end{align*}
		Therefore, \eqref{LCGM_step3} can be written as
		\begin{align*}
			0 &= \scal{\nu}{\lap_{g_{\Psi}}\dot{Y}}_\eta + \scal{\dot{\nu}}{Y}_\eta |\Arond|^2_{g_\Psi}.
		\end{align*}
		Using $|\Arond|^2_{g_\Psi} = |\g Y|^2_{g_\Psi,\eta}$ and $\scal{\dot{\nu}}{Y}_\eta = -\scal{\dot{Y}}{\nu}_\eta$, we conclude.
	\end{proof}
	
	Since we consider variations of $\Psi$ satisfying $\dot{\Psi}\perp \g \Psi$, the relation \eqref{LCGM_perp_variation} is not enough to caracterize the variations coming from conformal Gauss maps. A second condition to fully complete the description is the following.
	
	\begin{lemma}\label{orthogonality_lemma}
		Let $\Psi : \Sigma \to \s^3$ be a smooth immersion. Let $(\Psi_t)_{t\in(-1,1)}$ be a smooth deformation of $\Psi$ such that $\dot{\Psi} \perp \g \Psi$. Consider $(Y_t)_{t\in(-1,1)}$ their conformal Gauss map. It holds $\scal{\g \dot{Y} }{\nu}_\eta = 0$.
	\end{lemma}
	
	\begin{proof}
		By direct computation, using $\nu\perp \dot{\nu}$ and \eqref{derivatives_Y}, it holds
		\begin{align*}
			\scal{\g Y}{\dot{\nu}}_\eta &= - \Arond \scal{\g \nu}{\dot{\nu}}_\eta.
		\end{align*}
		Using $\dot{\Psi}\perp \g \Psi$, we obtain $\scal{\g \nu}{\dot{\nu}}_\eta = 0$. We conclude thanks to $\scal{\g Y}{\dot{\nu}}_\eta = - \scal{\g \dot{Y} }{\nu}_\eta$.
	\end{proof}
	
	To show that the converse is true, we will need to test the equation \eqref{LCGM_perp_variation} against linear combinations of $\nu$ and $\g \nu$.
	
	\begin{lemma}\label{LCGM_invariants}
		Let $\Psi : \Sigma \to \s^3$ be a smooth immersion. Let $\Omega\subset\Sigma$ be an open set included in the domain of a chart where $g_\Psi$ is conformally flat. For any $\alpha^1,\alpha^2,\beta\in C^\infty(\Omega)$, it holds
		\begin{align*}
			\scal{\nu}{\lap_{g_\Psi} (\alpha^k \dr_k \nu) - |\g Y|^2_\eta \alpha^k\dr_k \nu }_\eta &= -2\di_{g_\Psi}(\alpha),\\
			\scal{\nu}{\lap_{g_\Psi} (\beta \nu) - |\g Y|^2_\eta \beta \nu }_\eta &= -2\beta.
		\end{align*}
	\end{lemma}
	
	\begin{proof}
		Let $\lambda$ be the conformal factor of $\Psi$ on $\Omega$ : $g_\Psi = \Psi^*\xi = e^{2\lambda} h$, where $h$ is the flat metric on $\R^2$.	We show the first relation. By direct computation using $\nu \perp \g \nu$,
		\begin{align*}
			\scal{\nu}{\lap_{g_\Psi} (\alpha^k \dr_k \nu) - |\g Y|^2_\eta \alpha^k\dr_k \nu }_\eta &= \scal{\nu}{\lap_{g_\Psi} (\alpha^k \dr_k \nu) }_\eta \\
			&= 2\scal{\nu }{\g \alpha^k \g \dr_k \nu}_\eta + \scal{\nu}{\alpha^k \lap_{g_\Psi}(\dr_k \nu)}_\eta \\
			&= -2\scal{\g \nu }{\g \alpha^k \dr_k \nu}_\eta + \scal{\nu}{\alpha^k e^{-2\lambda}\lap_h(\dr_k \nu)}_\eta \\
			&= -2(\g \alpha^k) \scal{\g \Psi}{\dr_k \Psi}_\xi + \scal{ \begin{pmatrix}
					\Psi\\ 1
			\end{pmatrix} }{\alpha^k e^{-2\lambda}\dr_k\left[ e^{2\lambda} \begin{pmatrix}
					H N -2\Psi \\ 0
				\end{pmatrix} \right] }_\eta \\
			&= -2 \dr_k \alpha^k - 4 \alpha^k (\dr_k \lambda) \\
			&= -2 e^{-2\lambda} \dr_k(e^{2\lambda} \alpha^k ) = -2 \di_{g_\Psi}(\alpha).
		\end{align*}
		We show the second relation : using $\nu,\g \nu \perp \nu$,
		\begin{align*}
			\scal{\nu}{\lap_{g_\Psi} (\beta \nu) - |\g Y|^2_\eta \beta \nu }_\eta &= \scal{\nu}{\lap_{g_\Psi} (\beta \nu)}_\eta \\
			&= \beta \scal{\nu}{\lap_{g_\Psi}\nu}_\eta \\
			&= \beta \scal{ \begin{pmatrix}
					\Psi\\ 1
			\end{pmatrix} }{ \begin{pmatrix}
					HN - 2\Psi\\ 0
			\end{pmatrix} }_\eta \\
			&= - 2\beta.
		\end{align*}
	\end{proof}
	
	We can now prove the caracterization of the variations of $Y$ coming from conformal Gauss maps, which concludes the proof of the proposition \ref{equality_index}. 
	\begin{proposition}\label{caracterization_CGM_variation}
		Let $\Psi : \Sigma \to \s^3$ be a smooth immersion and $Y$ be its conformal Gauss map. Consider a smooth map $Z : (\Sigma,g_\Psi)\to \R^{4,1}$ such that $Z\in T_Y \s^{3,1}$. Assume that the set $\Ur\subset \Sigma$ of umbilic points of $\Psi$ are nowhere dense, i.e. $\overline{\Sigma\setminus \Ur} = \Sigma$, and that $Z$ satisfy the two following equations :
		\begin{align}
			\scal{\nu}{\lap_{g_\Psi} Z - |\g Y|^2_\eta Z }_\eta &= 0, \label{LCGM_perp}\\
			\scal{\g Z }{\nu}_\eta = 0. \label{orthogonality}
		\end{align}
		Then, there exists a smooth variations $(Y_t)_{t\in(-1,1)}$ of $Y$ which are conformal Gauss maps of a variation $(\Psi_t)_{t\in(-1,1)}$ of $\Psi$ satisfying $\dot{\Psi} \perp \g \Psi$ and $\dot{Y} = Z$.
	\end{proposition}
	
	\begin{proof}
		Since $Z\in T_Y \s^{3,1}$, there exists a variation $(Y_t)_{t\in(-1,1)}$ of $Y$ in $\s^{3,1}$ such that $\dot{Y} = Z$. Consider a open set $\Omega\subset \Sigma\setminus \Ur$ relatively compact. Since $Y$ is strictly space-like on $\Omega$, there exists $t_0\in(0,1)$ depending on $\Omega$ such that for $|t|<t_0$, $Y_t$ is strictly space-like. So if $|t|<t_0$, the normal space of $Y_t$ contains two isotropic directions $\nu_t$ and $\nu_t^*$. Up to exchange $\nu_t$ and $\nu_t^*$, the family $(Y_t,\g Y_t,\nu_t,\nu_t^*)$ is a direct basis of $\R^{4,1}$. If we normalize $\nu_t$ in order to have $(\nu_t)_5 = 1$, then $\nu_t = \begin{pmatrix}
			\Psi_t \\ 1
		\end{pmatrix}$ and $(\Psi_t)_{t\in(-t_0,t_0)}$ is a smooth variation of $\Psi$ on $\Omega$. Let $\gamma_t$ be the conformal Gauss map of $\Psi_t$. Since $\Psi_t$ is also an immersion for $|t|$ small enough, the vector space $\Span(\nu_t,\g \nu_t)^\perp$ has dimension 2 and contains three vectors : $\gamma_t$, $Y_t$ and $\nu_t$. Hence, there exists $\alpha_t,\beta_t\in C^\infty(\Omega)$ such that $Y_t= \alpha_t \gamma_t + \beta_t \nu_t$. We remark that $\alpha_t = 1$ and $\beta_0 = 0$ since $Y_0 = Y = \gamma_0$ and $1 = |Y_t|^2_\eta = \alpha_t^2$. Therefore $Z =\dot{Y} = \dot{\gamma} + \dot{\beta} \nu$. Consider a family of diffeomorphisms $\vp_t \in \Diff(\Omega)$ such that $\Psi_t\circ\vp_t^{-1}$ satisfy $\dr_t[\Psi_t\circ\vp_t^{-1}]_{|t=0} \perp \g \Psi$. Let $\Gamma_t := \gamma_t\circ\vp_t^{-1}$ be the conformal Gauss map of $\Psi_t\circ\vp_t^{-1}$. Then
		\begin{align}\label{decomposition_Z_step1}
			Z &= \dot{\Gamma} + (\dot{\vp}\cdot \g )Y + \dot{\beta}\nu.
		\end{align}
		Using \eqref{derivatives_Y}, we obtain
		\begin{align*}
			Z &= \dot{\Gamma} + (\dot{\beta} + (\dot{\vp}\cdot \g )H )\nu - \dot{\vp} \Arond \g \nu.
		\end{align*}
		Since $Z$ and $\dot{\Gamma}$ satisfy \eqref{LCGM_perp}, by the lemma \ref{LCGM_invariants} we obtain
		\begin{align*}
			\dot{\beta} + (\dot{\vp}\cdot \g )H &= \di_{g_\Psi}(\dot{\vp} \Arond).
		\end{align*}
		Therefore, \eqref{decomposition_Z_step1} reduces to
		\begin{align*}
			Z &= \dot{\Gamma} +\di_{g_\Psi}(\dot{\vp} \Arond) \nu - \dot{\vp} \Arond \g \nu. 
		\end{align*}
		Using the relation \eqref{orthogonality}, the lemma \ref{orthogonality_lemma} and the orthogonality relations $\nu,\g\nu\perp \nu$, we obtain 
		\begin{align*}
			0 = \scal{\g Z}{\nu}_\eta &= \scal{\g \dot{\Gamma}}{\nu}_\eta  - \scal{\dot{\vp} \Arond \g^2 \nu}{\nu}_\eta = \dot{\vp}\Arond.
		\end{align*}	
		Since $\Arond$ is invertible on $\Omega$, we obtain $\dot{\vp} = 0$. So $Z = \dot{\Gamma}$ on $\Omega$. We remark that 
		\begin{align*}
			\scal{\dr_t[\Psi_t\circ \vp_t^{-1}]_{|t=0}}{N}_\xi = - \scal{\dot{\Gamma}}{\nu}_\eta = -\scal{Z}{\nu}_\eta.
		\end{align*}
		By differentiating the definition \eqref{definition_Y}, we obtain that $\dot{\Gamma}$ depends only on $\scal{\dr_t[\Psi_t\circ \vp_t^{-1}]_{|t=0}}{N}_\xi$. So $\dot{\Gamma}$ is actually the derivative at $t=0$ of the conformal Gauss map of the variation 
		\begin{align*}
			\tilde{\Psi}_t &:= \frac{\Psi - t \scal{Z}{\nu}_\eta N }{|\Psi - t \scal{Z}{\nu}_\eta N|_\xi}.
		\end{align*}
		The variation $(\tilde{\Psi}_t)_{t\in(-1,1)}$ is a smooth variation of $\Psi$ defined on the whole surface $\Sigma$. Hence, the equality $Z= \dot{\Gamma}$ holds on $\Sigma \setminus \Ur$. Since $\Sigma\setminus \Ur$ is dense in $\Sigma$, we obtain $Z = \dot{\Gamma}$ on the whole surface $\Sigma$ by continuity.
	\end{proof}

	\section{Proof of the theorem \ref{quantization_bounded_index}}\label{proof_quantization}
	
	\subsection{Setting}
	Consider a sequence $\Psi_k : \Sigma \to \s^3$ of Willmore immersions satisfying
	\begin{align*}
		\sup_{k\in\N} \Er(\Psi_k) <\infty.
	\end{align*}
	We assume that the metrics $g_k := \Psi_k^* (\xi_{|\s^3})$ degenerate in the moduli space of $\Sigma$. Consider its conformal Gauss map $(Y_k)_{k\in\N}$, and assume that there exists $\lambda>0$ and a collar region $\Cr_k:= [-T_k,T_k]\times \s^1$ satisfying
	\begin{align}\label{existence_energy}
		\liminf_{k\to \infty} \Ar(Y_k;\Cr_k) \geq \lambda.
	\end{align}
	
	\textbf{Goal :} We show that $\Ind_\Ar(Y_k) \xrightarrow[k\to \infty]{}{+\infty}$. Given the proposition \ref{equality_index}, we conclude that $\Ind_\Er(\Psi_k) \xrightarrow[k\to \infty]{}{+\infty}$.\\
	In the section \ref{section_construction_jacobi_field}, we show how to construct a Jacobi field on a region of given area. In the section \ref{section_explosion_index}, we conclude.\\
	
	\begin{remark}\label{umbilic_circle}
		According to \cite[Theorem 3.1]{schatzle2017}, there might exists some $t\in[-T_k,T_k]$ such that $\Arond_k = 0$ on the whole circle $\{t\}\times \s^1$. However, if we consider a diffeomorphism $f\in\Diff(\Cr_k)$, then $\{x\in\Cr_k : \Arond_{\Psi_k\circ f}(x) = 0 \} = f^{-1} (\{y\in \Cr_k : \Arond_{\Psi_k}(y) = 0\})$ since the traceless part of the second fundamental form is a geometric quantity. Hence, if we change the parametrization of $\Psi_k$, we can erase circles of umbilic points of the form $\{t\}\times \s^1$ on the domain. This is done in the \Cref{no_umbilic_circle}. Therefore, up to a reparametrization, it holds
		\begin{align}\label{hypothesis_no_umbilic_circle}
			\forall t\in[-T_k,T_k],\ \ \ \int_{\{ t\}\times \s^1} |\g Y_k|^2_\eta >0 .
		\end{align}
	\end{remark}
	
	\textbf{We will assume (\ref{hypothesis_no_umbilic_circle}) until the end of the section.}\\
	
	\begin{remark}\label{existence_oscillation_along_cylinder}
		We observe that $Y_k$ must converge to an infinite light-like straight line. Indeed, assume that there exists $t_k\in[-T_k,T_k]$ such that, letting $M_k \in SO(4,1)$ be the matrix given by the \cref{bubbling_CGM}, $\| M_k Y_k \|_{L^\infty_\xi([t_k,T_k]\times \s^1)} $ is bounded. Then the \cref{estimate_xitheta} holds on the whole domain $[t_k,T_k]\times \s^1$. In particular by \eqref{integral_xitheta}, $\| M_k\dr_\theta Y_k \|_{L^2_\xi([t_k,T_k]\times \s^1)} $ converges to $0$. Using \eqref{residue_hopf_diff}, we obtain that $\| \g Y_k \|_{L^2_\eta ([t_k,T_k]\times \s^1)}$ converges to $0$. So there is no area in this region. Therefore, on a region with given area, the euclidean oscillations must go to $+\infty$. So the \cref{bubbling_CGM} shows that $(Y_k)_k$ converges to an infinite light-like straight line, up to isometries and reparametrization.
	\end{remark}

	\subsection{Second variation of the Dirichlet energy of the conformal Gauss map of a Willmore immersion without umbilic points}\label{section_construction_jacobi_field}
	
	In this section, we show the following :
	\begin{lemma}\label{existence_jacobi_field}
		Let $\mu>0$. Assume that there exists a subset $[a_k,b_k]\subset [-T_k,T_k]$ satisfying 
		\begin{align*}
			2\mu \geq \Ar(Y_k;[a_k,b_k]\times \s^1) \geq \mu.
		\end{align*}
		Then there exists $k_0\in\N$ such that for $k\geq k_0$, there exists a vector field $V_k$ along $Y_k$ supported on the cylinder $[a_k,b_k]\times \s^1$ such that $\delta^2 \Dr_{Y_k}(V_k) <0$. Moreover, there exists a smooth variation $(\Psi_k^u)_{u\in(-1,1)}$ of $\Psi$ on $[a_k,b_k]\times \s^1$ such that $(\dr_u \Psi_k^u)_{|u=0} \perp \g \Psi_k$ and their conformal Gauss maps $(Y_k^u)_{u\in(-1,1)}$ satisfy $(\dr_u Y_k^u)_{|u=0} = V_k$.
	\end{lemma}
	
	In the rest of the section, we fix $\mu>0$ and a cylinder $\Cr_k :=[a_k,b_k]\times \s^1$ satisfying
	\begin{align}\label{assumption_small_energy}
		2\mu\geq \Ar(Y_k;\Cr_k) \geq \mu.
	\end{align}
	We denote the average lorentz length on $\Cr_k$ by 
	\begin{align*}
		\ell_k := \int_{a_k}^{b_k} \left( \int_{\s^1} |\g Y_k|^2_\eta d\theta \right)^\frac{1}{2} dt.
	\end{align*}
	We observe that $\ell_k$ cannot be bounded. Thanks to \eqref{apriori_necks_cylinder}, we have $\|\g Y_k\|_{L^\infty_\eta (\Cr_k)} \xrightarrow[k\to \infty]{}{0}$. Therefore, by the assumption \eqref{assumption_small_energy} we obtain
	\begin{align*}
		\mu &\leq \int_{\Cr_k} |\g Y_k|^2_\eta \leq \| \g Y_k\|_{L^\infty_\eta(\Cr_k)} \int_{\Cr_k} |\g Y_k|_\eta \leq \| \g Y_k\|_{L^\infty_\eta(\Cr_k)} \sqrt{2\pi} \int_{a_k}^{b_k} \left( \int_{\s^1} |\g Y_k|_\eta^2 d\theta \right)^\frac{1}{2} dt.
	\end{align*}
	So that :
	\begin{align}\label{infinite_length}
		\ell_k \xrightarrow[k\to \infty]{}{+\infty}.
	\end{align}
	A second observation is that the vector fields $(\nu_k)_{k\in\N}$ converge toward the limit geodesic $\sigma$. Indeed, since we have the uniform bound $|\nu_k|_\xi^2=2$, the normal vector fields $(\nu_k)_{k\in\N}$ also converge to a limit normal vector field $\nu_\infty$ along $\sigma$. Thanks to the theorem \ref{bubbling_CGM}, we know that the convergence of $(Y_k)_{k\in\N}$ to $\sigma$ holds in the $C^2$-topology, up to reparametrization and an isometry of $\s^{3,1}$. By \eqref{eq:conv_lightlike}, the fact that $|\g Y_k|_\eta \ust{k\to\infty}{=} o(|\g Y_k|_\xi)$ up to an isometry means that we have the following asymptotic expansion : thanks to \eqref{derivatives_Y},
	\begin{align*}
		\dr_t Y_k \ust{k\to \infty}{=} (\dr_t H_k) \nu_k + o(\dr_t H_k).
	\end{align*}
	Therefore, the arclength of $Y_k^*$ is given by the behaviour of $\dr_t H_k$, and we obtain :
	\begin{align*}
		\dr_s Y_k \ust{k\to \infty}{=} \frac{\nu_k}{\sqrt{2}} + o(1).
	\end{align*}
	The convergence of the second derivatives can be written as 
	\begin{align}\label{as:gnu}
		\dr_s \nu_k \ust{k\to \infty}{=} o(1).
	\end{align}
	Therefore, $\nu_\infty$ is actually constant proportional to $\dr_s \sigma$ and the scalar product with $\nu$ in \eqref{system_condition_LCGM} can be exchanged with $\dr_s \sigma$ up to a small error. We show that we can choose a variation through conformal Gauss maps that converges to a parallel space-like normal vector field along $\sigma$. \\
	
	Consider $a,b\in\R^{4,1}$ such that the limit geodesic $\sigma$ is given by $\forall s\in\R,\ \sigma(s) = as+b$. The vector space $\Span(a,b)^\perp$ has dimension $3$, so there exists a vector $E\in \R^{4,1}$ satisfying $|E|^2_\eta = 1$ and $\scal{E}{a}_\eta = \scal{E}{b}_\eta = 0$. It precisely means that $E\in T_\sigma \s^{3,1}$, $\g_{\dot{\sigma}} E = 0$ and $\scal{E}{\dot{\sigma}}_\eta = 0$, so $E$ is a parallel vector field along $\sigma$. Thanks to \eqref{hypothesis_no_umbilic_circle}, we can consider the following change of variable :
	\begin{align}\label{change_variable}
		\alpha_k(t) := \left( \int_{\s^1} |\g Y_k(t,\theta)|^2_\eta d\theta \right)^\frac{1}{2}, &  & ds = \alpha_k(t) dt.
	\end{align}
	The domain $[a_k,b_k]$ of the variable $t$ becomes $[0,\ell_k]$ for the variable $s$. Let $f_k(s) := \sin\left( \frac{\pi s}{\ell_k } \right)$ and $\rho_k(s)$ be a cut-off function : $\rho_k\geq 0$, $\rho_k(s) = 1$ if $s\in[1,\ell_k-1]$, $\rho_k(s) = 0$ if $s\in[0,\frac{1}{2}]\cup [\ell_k- \frac{1}{2}, \ell_k]$ and $|\dr_s^i \rho_k| \leq C$ for $i\in\inter{1}{4}$. Consider
	\begin{align}
		E_k &:= \rho_k f_k E - \rho_k f_k\scal{E}{Y_k}_\eta Y_k + d_k \nu_k + e^i_k \dr_i \nu_k, \label{definition_jacobi_field}\\
		e_k &:= \g^{g_{\Psi_k}} (\rho_k f_k) \scal{E}{\nu_k}_\eta, \label{component_gnu}\\
		d_k &:= -\frac{1}{2} \scal{\nu_k}{  \lap_{g_{\Psi_k}}(\rho_k f_k E) + |\g Y_k|^2_\eta \rho_k f_k E }_\eta - \scal{\g^{g_{\Psi_k}} (\rho_k f_k E) }{\g^{g_{\Psi_k}} \nu_k }_\eta. \label{component_nu}
	\end{align}
	Therefore $\scal{E_k}{Y_k}_\eta = 0$, so $E_k$ is a vector field along $Y_k$ and $E_k(s,\theta) = 0$ if $s\in[0,\frac{1}{2}]\cup[\ell_k-\frac{1}{2},\ell_k]$, so $E_k$ has compact support. From the \cref{equality_index}, we obtain :
	\begin{lemma}
		The vector field $E_k$ comes from a variation through conformal Gauss maps of immersions $(\Psi_k^u)_{u\in(-1,1)}$ such that $(\dr_u \Psi_k^u)_{|u=0} \perp \g \Psi_k$. 
	\end{lemma}
	
	\begin{proof}
		We check \eqref{system_condition_LCGM} by direct computations. Using that $Y_k,\nu_k,\g \nu_k,\g Y_k \perp \nu_k$ we obtain
		\begin{align*}
			\scal{\g E_k}{\nu_k}_\eta &= \scal{\g(\rho_k f_k E) }{\nu_k}_\eta + e^i_k \scal{\g \dr_i \nu_k}{\nu_k}_\eta \\
			&= \scal{\g(\rho_k f_k E) }{\nu_k}_\eta - e_k = 0.
		\end{align*}
		Thanks to the \cref{LCGM_invariants} and the orthogonality relations $Y_k,\g Y_k, \lap_{g_{\Psi_k}} Y_k \perp \nu_k$ :
		\begin{align*}
			\scal{\nu_k}{ \lap_{g_{\Psi_k}}E_k - |\g Y_k|^2_\eta E_k }_\eta &=\scal{\nu_k}{  \lap_{g_{\Psi_k}}(\rho_k f_k E) - |\g Y_k|^2_\eta \rho_k f_k E }_\eta - 2 d_k - 2 \di_{g_{\Psi_k}} (e_k) \\
			&= -\scal{\nu_k}{  \lap_{g_{\Psi_k}}(\rho_k f_k E) + |\g Y_k|^2_\eta \rho_k f_k E }_\eta - 2\scal{\g^{g_{\Psi_k}} (\rho_k f_k E) }{\g^{g_{\Psi_k}} \nu_k }_\eta - 2d_k\\
			&= 0.
		\end{align*}
		Thanks to the \cref{equality_index}, there is a variation $(\Psi_k^u)_{u\in(-1,1)}$ of $\Psi_k$ such that its conformal Gauss map satisfies $\dot{Y}_k = E_k$ and $\dot{\Psi}_k \perp \g \Psi_k$.
		\iffalse Furthermore, thanks to \eqref{normal_part_variation} and \eqref{definition_jacobi_field}, it holds 
		\begin{align*}
			\scal{\dot{\Psi}_k}{N_k}_\xi = -\scal{E_k}{\nu_k}_\eta = -f_k \scal{E}{\nu_k}_\eta.
		\end{align*}
		Hence, we have the following formula, thanks to \eqref{definition_Y} :
		\begin{align*}
			E_k = \dot{Y}_k &= \dot{H}_k \nu_k + \begin{pmatrix}
				H_k \dot{\Psi}_k + \dot{N}_k \\ 0
			\end{pmatrix}.
		\end{align*}
		By construction, it holds $\dot{\Psi}_k = -f_k \scal{E}{\nu_k}_\eta N_k$ and $\dot{N}_k = f_k \scal{E}{\nu_k} \Psi_k$. We also compute $\dot{H}_k$ :
		\begin{align*}
			\dot{H}_k &= \frac{1}{2}\left( \scal{\lap \dot{\Psi}_k}{N_k}_\xi + \scal{\lap \Psi_k}{\dot{N}_k}_\xi \right).
		\end{align*}
		Since $\Psi_k$ is conformal, the last term vanishes. It remains
		\begin{align*}
			\dot{H}_k &= -\frac{1}{2}\scal{\lap (f_k \scal{E}{\nu_k}_\eta N_k)}{N_k}_\xi.
		\end{align*}
		Since $N_k\perp \g N_k$, we obtain
		\begin{align*}
			\dot{H}_k &= -\frac{1}{2}\left( \lap (f_k \scal{E}{\nu_k}_\eta) + f_k \scal{E}{\nu_k}_\eta \scal{\lap N_k}{N_k}_\xi \right)\\
			&= -\frac{1}{2}\left( \lap (f_k \scal{E}{\nu_k}_\eta) - f_k \scal{E}{\nu_k}_\eta |\g N_k|^2_\xi \right).
		\end{align*}
		\fi
	\end{proof}
	
	We now prove that the main contribution in the computation of $\delta^2 \Dr_{Y_k}(E_k)$ comes from $f_k E$. By a classical computation, we have for any smooth variation $(Y^t_k)_{t\in(-1,1)}$ of $Y_k$ with compact support on $\Cr_k$ : 
	\begin{align}\label{general_formula_second_variation_dirichlet}
		\delta^2 \Dr_{Y_k}(\dot{Y}_k) &= \int_{\Cr_k} |\g \dot{Y}_k |^2_\eta - \Riem^{\s^{3,1}}(\dot{Y}_k , \g Y_k, \dot{Y}_k,\g Y_k) \ dt d\theta.
	\end{align}
	where $\dot{Y}_k := (\dr_t Y^t_k)_{|t=0}$ and $\Riem^{\s^{3,1}}(x,y,z,t) = \scal{x}{z}_\eta \scal{y}{t}_\eta - \scal{x}{t}_\eta \scal{y}{z}_\eta$. To obtain an asymptotic expansion, we use the fact that $\delta^2 \Dr_{Y_k}(E_k)$ is invariant by isometries of $\s^{3,1}$ and that the vector field $E_k$ is also a conformal invariant in the following sense. There exists a smooth variation $(\Psi_k^u)_{u\in(-1,1)}$ of $\Psi_k$ such that their conformal Gauss maps $(Y_k^u)_{u\in(-1,1)}$ satisfy $(\dr_u Y_k^u)_{|u=0}= E_k$. Then for any conformal transformation $\Theta \in \Conf(\s^3)$, we have $M_{\Theta} E_k = (\dr_u M_{\Theta} Y_k^u)_{|u=0}$, and $M_\Theta Y_k^u$ is the conformal Gauss map of the variation $\Theta \circ \Psi_k^u$ of $\Theta\circ\Psi_k$. Therefore in all the pointwise estimates of the proof of the following lemma, the isometry will depend on the point we are working on. We will precise which one at each step, but for simplicity, we will always assume that this matrix is the identity. 
	
	\begin{lemma}\label{asympotic_second_derivative_dirichlet_Ek}
		It holds
		\begin{align*}
			\delta^2 \Dr_{Y_k}(E_k) \ust{k\to\infty}{=} \int_{\Cr_k} |\g (\rho_k f_k)|^2 -\rho_k^2 f_k^2 |\g Y_k|^2_\eta\ dtd\theta + o\left( \frac{1}{\ell_k^2}+  \int_{\Cr_k} |\g f_k|^2 + f_k^2 |\g Y_k|^2_\eta\ dtd\theta \right).
		\end{align*}
	\end{lemma}
	
	\begin{proof}
		We proceed by brute force using the definition \eqref{definition_jacobi_field}. Thanks to the formula \eqref{general_formula_second_variation_dirichlet}  :
		\begin{align}\label{second_variation_dirichlet_Ek}
			\delta^2 \Dr_{Y_k}(E_k) &= \int_{\Cr_k} |\g E_k|^2_\eta - |E_k|^2_\eta |\g Y_k|^2_\eta + |\scal{E_k}{\g Y_k}_\eta|^2\ dtd\theta.
		\end{align}	
		To obtain an asympototic expansion, we will need to estimate the derivatives of $f_k$ and $\rho_k$. We first proceed to this preliminary work.
		
		\underline{Estimate of the first derivative of $f_k$ and $\rho_k$ :}\\
		We consider the change of variable \eqref{change_variable}. Using $\frac{ds}{dt} = \left( \int_{\s^1} |\g Y_k|^2_\eta d\theta \right)^\frac{1}{2} =: \alpha_k(t)$, we obtain $\dr_t = \alpha_k \dr_s$. We compute the first derivative of $f_k$. 
		\begin{align*}
			\dr_t f_k &= \alpha_k \dr_s f_k = \frac{\pi \alpha_k}{\ell_k} \cos\left( \frac{\pi s}{\ell_k} \right).
		\end{align*} 
		So 
		\begin{align}\label{derivative1_f}
			|\dr_t f_k| \leq C\ell_k^{-1} \alpha_k.
		\end{align}
		We also have $\dr_t \rho_k = \alpha_k \dr_s \rho_k$, so
		\begin{align}\label{derivative1_rho}
			|\dr_t \rho_k| &\leq C\alpha_k \mathbf{1}_{[0,1]\cup[\ell_k-1,\ell_k]}(s).
		\end{align}

		\underline{Estimate of the second derivative of $f_k$ and $\rho_k$ :}\\
		We will need the derivative of $\alpha_k$. It is given by
		\begin{align}\label{derivative_alpha}
			\dr_s \alpha_k &= \alpha_k^{-1} \int_{\s^1} \scal{\g Y_k}{\dr_s \g Y_k}_\eta d\theta. 
		\end{align}
		We differentiate twice $f_k$ :
		\begin{align*}
			\dr^2_{tt} f_k &= \alpha_k \dr_s\big( \alpha_k \dr_s f_k) \\
			&= \frac{\pi}{\ell_k} \left( \int_{\s^1} \scal{\g Y_k}{\dr_s \g Y_k}_\eta d\theta \right) \cos\left( \frac{\pi s}{\ell_k} \right) - \left( \frac{\pi \alpha_k}{\ell_k} \right)^2 \sin\left( \frac{\pi s}{\ell_k} \right).
		\end{align*}
		To bound the first term, we use \eqref{derivatives_Y} in the isometry of $\s^{3,1}$ given by the corollary \ref{control_normal_by_arond} :
		\begin{align*}
			\left| \int_{\s^1} \scal{\g Y_k}{\dr_t \g Y_k}_\eta d\theta \right| &= 	\left| \int_{\s^1} \scal{ (\g H_k) \begin{pmatrix}
					\Psi_k\\ 1
				\end{pmatrix} - \begin{pmatrix}
					\Arond_k \g \Psi_k \\ 0
			\end{pmatrix} }{\dr_t \left[(\g H_k) \begin{pmatrix}
					\Psi_k\\ 1
				\end{pmatrix} - \begin{pmatrix}
					\Arond_k \g \Psi_k \\ 0
				\end{pmatrix} \right]  }_\eta d\theta \right| \\
			&= \left| \int_{\s^1} \scal{ (\g H_k) \begin{pmatrix}
					\Psi_k\\ 1
				\end{pmatrix} - \begin{pmatrix}
					\Arond_k \g \Psi_k \\ 0
			\end{pmatrix} }{\dr_t\g H_k \begin{pmatrix}
					\Psi_k\\ 1
				\end{pmatrix} +(\g H_k) \begin{pmatrix}
					\dr_t \Psi_k\\ 0
			\end{pmatrix} }_\eta d\theta \right.\\
			&\left. - \int_{\s^1} \scal{ (\g H_k) \begin{pmatrix}
					\Psi_k\\ 1
				\end{pmatrix} - \begin{pmatrix}
					\Arond_k \g \Psi_k \\ 0
			\end{pmatrix} }{  \begin{pmatrix}
					\dr_t(\Arond_k) \g \Psi_k + \Arond \dr_t\g \Psi_k \\ 0
			\end{pmatrix} }_\eta d\theta \right| \\
			&= \left| \int_{\s^1}-\scal{\Arond_k \g\Psi_k}{(\g H_k)\dr_t \Psi_k}_\xi -\scal{(\g H_k)\Psi_k }{\Arond \dr_t \g \Psi_k}_\xi  + \scal{\Arond_k \g \Psi_k}{\dr_t[\Arond_k \g \Psi_k]}_\xi \ d\theta \right|\\
			&= \left|  \int_{\s^1} \scal{\Arond_k \g \Psi_k}{\dr_t[\Arond_k \g \Psi_k]}_\xi \ d\theta \right|.
		\end{align*}
		By the corollary \ref{control_normal_by_arond} and the $\ve$-regularity, it holds 
		\begin{align}\label{derivative2_alpha}
			\left| \int_{\s^1} \scal{\g Y_k}{\dr_t \g Y_k}_\eta d\theta \right| \leq \left( \int_{\s^1} |\Arond|^2_{g_\Psi} d\theta \right)^\frac{1}{2} \left(\int_{\s^1} |\g \Arond|^2_{g_\Psi} d\theta \right)^\frac{1}{2} &\leq C\alpha_k \|\g Y_k\|_{L^2_\eta(Q(t,2))}.
		\end{align}
		Hence, $|\int_{\s^1} \scal{\g Y_k}{\dr_s \g Y_k}_\eta d\theta| \leq C\|\g Y_k\|_{L^2_\eta(Q(t,2))}$. So 
		\begin{align}\label{derivative2_f}
			|\dr_{tt} f_k| \leq C\ell_k^{-1} \|\g Y_k\|_{L^2_\eta(Q(t,2))}.
		\end{align}
		We proceed as well for $\rho_k$ :
		\begin{align*}
			\dr^2_{tt} \rho_k &= \alpha_k(\dr_s \alpha_k) (\dr_s \rho_k) + \alpha_k^2 (\dr^2_{ss} \rho_k).
		\end{align*}
		Thanks to \eqref{derivative_alpha}-\eqref{derivative2_alpha}, we obtain
		\begin{align}\label{derivative2_rho}
			|\dr^2_{tt} \rho_k| &\leq C\|\g Y_k\|_{L^2_\eta(Q(t,2))} \mathbf{1}_{[0,1]\cup[\ell_k-1,\ell_k]}(s).
		\end{align}

		\underline{Estimate of the third derivative of $f_k$ and $\rho_k$ :}\\
		We first compute the third derivative of $f_k$ :
		\begin{align*}
			\dr^3_{ttt} f_k &= \alpha_k \dr_s\left[\frac{\pi}{\ell_k} \left( \int_{\s^1} \scal{\g Y_k}{\dr_s\g  Y_k}_\eta d\theta \right) \cos\left( \frac{\pi s}{\ell_k} \right) - \left( \frac{\pi \alpha_k}{\ell_k} \right)^2 \sin\left( \frac{\pi s}{\ell_k} \right) \right] \\
			&= \frac{\pi \alpha_k}{\ell_k} \left( \int_{\s^1} |\dr_s \g Y_k|^2_\eta + \scal{\g Y_k}{\dr^2_{ss}\g Y_k}_\eta \ d\theta \right) \cos\left( \frac{\pi s}{\ell_k} \right) \\
			&-\frac{\pi^2 \alpha_k}{\ell_k^2}\left( \int_{\s^1}\scal{\g Y_k}{\dr_s \g Y_k}_\eta d\theta \right) \sin\left( \frac{\pi s}{\ell_k} \right) \\
			&-\frac{2\pi^2 \alpha_k}{\ell_k^2} \left( \int_{\s^1}\scal{\g Y_k}{\dr_s\g Y_k}_\eta d\theta \right)\sin\left( \frac{\pi s}{\ell_k} \right) \\
			&- \left( \frac{\pi \alpha_k}{\ell_k} \right)^3 \sin\left( \frac{\pi s}{\ell_k} \right).
		\end{align*}
		The three last terms can be bounded by $\alpha_k/\ell_k^2$. Using $\dr_s = \alpha_k^{-1} \dr_t$ on the first term, we obtain
		\begin{align*}
			\dr^3_{ttt} f_k	&\ust{k\to\infty}{=} \frac{\pi\alpha_k}{\ell_k} \left( \int_{\s^1} |\dr_s \g Y_k|^2_\eta + \scal{\g Y_k}{\dr^2_{ss} \g Y_k}_\eta\ d\theta\right) \cos\left( \frac{\pi s}{\ell_k} \right) + O \left( \frac{ \alpha_k}{\ell_k^2 } \right) \\
			&\ust{k\to\infty}{=} \frac{\pi}{\ell_k \alpha_k} \left( \int_{\s^1} |\dr_t \g Y_k|^2_\eta + \scal{\g Y_k}{\dr^2_{tt} \g Y_k}_\eta\ d\theta\right) \cos\left( \frac{\pi s}{\ell_k} \right) \\
			&- \frac{\pi}{\ell_k}\frac{\dr_t \alpha_k}{\alpha_k^2} \left( \int_{\s^1} \scal{\g Y_k}{\dr_t \g Y_k}_\eta d\theta \right)\cos\left( \frac{\pi s}{\ell_k} \right) + O \left( \frac{ \alpha_k}{\ell_k^2 } \right)
		\end{align*}
		We estimate the second term thanks to \eqref{derivative_alpha} and \eqref{derivative2_alpha}. We estimate the first term thanks to the $\ve$-regularity :
		\begin{align}
			|\dr^3_{ttt} f_k|&\leq \frac{C}{\ell_k \alpha_k}\|\g Y_k\|_{L^2_\eta(Q(t,2))}^2 + \frac{C}{\ell_k \alpha_k^3} \left( \int_{\s^1} \scal{\g Y_k}{\dr_t\g Y_k}_\eta d\theta \right)^2 + O \left( \frac{ \alpha_k}{\ell_k^2 } \right) \nonumber \\
			&\leq \frac{C}{\ell_k \alpha_k}\|\g Y_k\|_{L^2_\eta(Q(t,2))}^2 + O \left( \frac{ \alpha_k}{\ell_k^2 } \right) \label{derivative3_f_alone}
		\end{align}
		Using \eqref{derivative1_f}, we obtain
		\begin{align}\label{derivative3_f}
			|\g f_k| |\g^3 f_k| &\leq C\ell_k^{-2}\|\g Y_k\|_{L^2_\eta (Q(t,2))}^2.
		\end{align}
		In the same manner, we estimate the third derivative of $\rho_k$. Thanks to \eqref{derivative_alpha}, it holds
		\begin{align*}
			\dr^3_{ttt}\rho_k &= \alpha_k \dr_s \left[ \left( \int_{\s^1} \scal{\g Y_k}{\dr_s \g Y_k}_\eta \right) (\dr_s \rho_k) + \alpha_k^2 (\dr^2_{ss} \rho_k) \right]\\
			&= \alpha_k \dr_s \left( \int_{\s^1} \scal{\g Y_k}{\dr_s \g Y_k}_\eta \right)(\dr_s \rho_k) + \alpha_k\left( \int_{\s^1} \scal{\g Y_k}{\dr_s \g Y_k}_\eta \right)(\dr^2_{ss} \rho_k) \\
			&+ 2\alpha_k\left( \int_{\s^1} \scal{\g Y_k}{\dr_s \g Y_k}_\eta \right) (\dr^2_{ss}\rho_k) + \alpha_k^3 (\dr^3_{sss}\rho_k).
		\end{align*}
		As for $\dr^3_{ttt} f_k$, we obtain 
		\begin{align}\label{derivative3_rho_alone}
			|\dr^3_{ttt}\rho_k| &\leq \frac{C}{\alpha_k} \|\g Y_k\|_{L^2_\eta(Q(t,2))}^2 \mathbf{1}_{[0,1]\cup[\ell_k-1,\ell_k]}(s).
		\end{align}
		Using \eqref{derivative1_rho}, we conclude that
		\begin{align}\label{derivative3_rho}
			|\dr_t \rho_k| |\dr^3_{ttt}\rho_k| &\leq C\|\g Y_k\|_{L^2_\eta(Q(t,2))}^2 \mathbf{1}_{[0,1]\cup[\ell_k-1,\ell_k]}(s).
		\end{align}
		
		\underline{Estimate of the term $\int_{\Cr_k} |\g E_k|^2_\eta$ in \eqref{second_variation_dirichlet_Ek} :}\\
		
		Given $(t_k,\theta_k)\in\Cr_k$, we estimate the term $|\g E_k(t_k,\theta_k)|^2_\eta $ using the isometry of $\s^{3,1}$ given by the \cref{control_normal_by_arond}. We compute $|\g E_k|^2_\eta$ :	
		\begin{align*}
			|\g E_k|^2_\eta =& \Big|\g(\rho_k f_k) E - \g(\rho_k f_k)\scal{E}{Y_k}_\eta Y_k - \rho_k f_k \scal{E}{\g Y_k}_\eta Y_k - \rho_k f_k \scal{E}{Y_k}_\eta \g Y_k \\
			& + (\g d_k)\nu_k + d_k(\g \nu_k) + (\g e^i_k) \dr_i \nu_k + e^i_k \g (\dr_i \nu_k) \Big|^2_\eta.
		\end{align*}		
		The key observation is that the leading terms inside the norm are $\g(\rho_k f_k) E$ since $E$ is constant, and possibly $(\g d_k)\nu_k$, since we don't have nice estimates on the third derivatives of $\rho_k$ and $f_k$. However $(\g d_k)\nu_k$ is light-like, so in the computation of $|\g E_k|^2_\eta$, the bad part vanishes. We have nice estimates for all the other terms. A second observation, is that when we will estimate derivatives of the product $\rho_k f_k$, we will keep track of $f_k$ when all the derivatives hit $\rho_k$, so we will often obtains terms of the form $f_k \mathbf{1}_{[0,1]\cup[\ell_k-1,\ell_k]}(s)$. This will allow us to use the assymptotic expansion $f_k \ust{k\to\infty}{=}C\ell_k^{-1} + o(\ell_k^{-1})$ near the boundary.\\
		
		We now developp $|\g E_k|^2_\eta$, using the orthogonality relations $Y_k,\nu_k\perp \g Y_k$ and $Y_k,\g Y_k,\nu_k,\g \nu_k \perp \nu_k$ :
		\begin{align}
			|\g E_k|^2_\eta =& |\g (\rho_k f_k)|^2  - 2 |\g (\rho_k f_k)|^2 \scal{E}{Y_k}_\eta^2 -4\rho_k f_k \g (\rho_k f_k) \scal{E}{\g Y_k}_\eta \scal{E}{Y_k}_\eta \nonumber\\
			&+ 2\g (\rho_k f_k) (\g d_k) \scal{E}{\nu_k}_\eta + 2\g (\rho_k f_k) d_k \scal{E}{\g \nu_k}_\eta + 2\g (\rho_k f_k) (\g e^i_k) \scal{E}{\dr_i \nu_k}_\eta + 2\g (\rho_k f_k) e^i_k \scal{E}{\g \dr_i \nu_k}_\eta \nonumber\\
			&+ |\g (\rho_k f_k)|^2 \scal{E}{Y_k}_\eta^2 +2\rho_k f_k \g (\rho_k f_k) \scal{E}{Y_k}_\eta \scal{E}{\g Y_k}_\eta - 2 e^i_k\g (\rho_k f_k) \scal{E}{Y_k}_\eta \scal{Y_k}{\g \dr_i \nu_k}_\eta \nonumber\\
			&+ (\rho_k f_k)^2 \scal{E}{\g Y_k}_\eta^2 - 2\rho_k f_k e^i_k \scal{E}{\g Y_k}_\eta \scal{Y_k}{\g \dr_i \nu_k}_\eta \nonumber\\
			&+ (\rho_k f_k)^2 \scal{E}{Y_k}_\eta^2 |\g Y_k|^2_\eta -2\rho_k f_k d_k \scal{E}{Y_k}_\eta \scal{\g Y_k}{\g \nu_k}_\eta -2\rho_k f_k (\g e^i_k) \scal{E}{Y_k}_\eta \scal{\g Y_k}{\dr_i \nu_k}_\eta \nonumber\\
			&- 2\rho_k f_k e^i_k \scal{E}{Y_k}_\eta \scal{\g Y_k}{\g \dr_i \nu_k}_\eta \nonumber\\
			&+ 2(\g d_k) e^i_k \scal{\nu_k}{\g \dr_i \nu_k}_\eta \nonumber\\
			&+ d_k^2 |\g \nu_k|^2_\eta + 2d_k (\g e^i_k) \scal{\g \nu_k}{\dr_i \nu_k}_\eta + 2d_k e^i_k \scal{\g \nu_k}{\g\dr_i \nu_k}_\eta \nonumber\\
			&+ |(\g e^i_k)\dr_i \nu_k|^2_\eta + 2 \scal{(\g e^i_k) \dr_i \nu_k }{ e^j_k \g \dr_j \nu_k}_\eta \nonumber\\
			&+ |e^i_k\g \dr_i \nu_k|^2_\eta. \label{eq:norm_gE_k}
		\end{align}
		We start by focusing on the scalar products. Thanks to \eqref{derivatives_Y}, it holds $\scal{\dr_i Y_k}{\dr_j \nu_k}_\eta = -(\Arond_k)_{ij}$. In particular, $\scal{\g Y_k}{\g \nu_k}_\eta = 0$. The terms involving second derivatives of $\nu_k$ can be controlled by $\|\g N_{\Theta_k\circ\Psi_k}\|_{L^2(Q(t,2))} = o(1)$, thanks to the $\ve$-regularity as in the proof of the \cref{existence_gauge}, where $\Theta_k\in\Conf(\s^3)$ corresponds to the isometry in $\s^{3,1}$ we are working with. As well for the term $\scal{\g Y_k}{\g \dr_i \nu_k}_\eta$ which can be estimated by $\ve$-regularity and \eqref{gradientY_bound_wrong_scale} 
		\begin{align*}
			|\scal{\g Y_k}{\g \dr_i \nu_k}_\eta| &\leq |\g Y_k|_\xi |\g^2 (\Theta_k\circ \Psi_k)|_\xi \\
			& \leq C|\g (\Theta_k\circ \Psi_k)|_\xi \|\g  N_{\Theta_k\circ \Psi_k}\|_{L^2(Q(t,2))}^2\\
			&\ust{k\to\infty}{=}o(1).
		\end{align*}
		Now we focus on the coefficient of the scalar products in \eqref{eq:norm_gE_k}. We estimate $e_k$ by its definition \eqref{component_gnu} and \eqref{derivative1_f}-\eqref{derivative1_rho} :
		\begin{align}\label{eq:est_ek}
			|e_k| &\leq C|\scal{E}{\nu_k}_\eta|\left( f_k\alpha_k\mathbf{1}_{[0,1]\cup[\ell_k-1,\ell_k]}(s) + \frac{\alpha_k}{\ell_k}\right).
		\end{align}
		We estimate the derivatives of $e_k$ in the following way : by its definition \eqref{component_gnu}, it holds
		\begin{align*}
			|\g e_k| &\leq C|\g^2(\rho_k f_k)||\scal{E}{\nu_k}_\eta| + C|\g (\rho_k f_k)| |\scal{E}{\g \nu_k}_\eta|.
		\end{align*}
		Thanks to \eqref{derivative1_f}, \eqref{derivative1_rho}, \eqref{derivative2_f} and \eqref{derivative2_rho}, we obtain
		\begin{align}\label{eq:est_gek}
			|\g e_k| &\leq C|\scal{E}{\nu_k}_\eta|\left( f_k\mathbf{1}_{[0,1]\cup[\ell_k-1,\ell_k]}(s) +\frac{1}{\ell_k}\right) \|\g Y_k\|_{L^2(Q(t,2))}  +C|\scal{E}{\g \nu_k}_\eta|\left(  f_k\alpha_k\mathbf{1}_{[0,1]\cup[\ell_k-1,\ell_k]}(s) + \frac{\alpha_k}{\ell_k}\right). 
		\end{align}
		We estimate $d_k$ by its definition \eqref{component_nu} :
		\begin{align*}
			|d_k|&\leq C|\scal{\nu_k}{E}_\eta| \left( |\g^2(\rho_kf_k)| + |\g Y|^2_\eta \rho_k f_k\right) + |\scal{E}{\g \nu_k}_\eta| |\g(\rho_k f_k)|.
		\end{align*}
		Thanks to \eqref{derivative1_f}, \eqref{derivative1_rho}, \eqref{derivative2_f} and \eqref{derivative2_rho}, we obtain
		\begin{align}
			|d_k| &\leq C|\scal{E}{\nu_k}_\eta|\left( \|\g Y_k\|_{L^2(Q(t,2))} f_k\mathbf{1}_{[0,1]\cup[\ell_k-1,\ell_k]}(s) +\frac{1}{\ell_k}\|\g Y_k\|_{L^2(Q(t,2))} + |\g Y_k|^2_\eta \rho_k f_k \right) \label{eq:est_dk1}\\
			& + C|\scal{E}{\g \nu_k}_\eta| \left(f_k \mathbf{1}_{[0,1]\cup[\ell_k-1,\ell_k]}(s) + \frac{1}{\ell_k} \right) \alpha_k. \label{eq:est_dk2}
		\end{align}
		We estimate the derivatives of $d_k$ in the following way : by its definition \eqref{component_nu}, it holds
		\begin{align*}
			|\g d_k| \leq &C|\scal{\g \nu_k}{E}_\eta| |\g^2(\rho_k f_k)| \\
			&+C|\scal{\nu_k}{E}_\eta|\left( |\g^3(\rho_k f_k)| + \rho_k f_k |\scal{\g Y_k}{\g^2 Y_k}_\eta| + |\g Y_k|^2_\eta |\g (\rho_k f_k)| \right)\\
			&+ C|\g^2(\rho_k f_k)| |\scal{E}{\nu_k}_\eta| + |\g (\rho_k f_k)| |\scal{E}{\g^2\nu_k}_\eta|.
		\end{align*}
		Thanks to \eqref{derivative1_f}, \eqref{derivative1_rho}, \eqref{derivative2_f}, \eqref{derivative2_rho}, \eqref{derivative3_f_alone}, \eqref{derivative3_rho_alone},  and the $\ve$-regularity, we obtain
		\begin{align*}
			|\g d_k| \leq &C|\scal{\g \nu_k}{E}_\eta|\left( f_k\mathbf{1}_{[0,1]\cup[\ell_k-1,\ell_k]}(s) +\frac{1}{\ell_k}\right) \|\g Y_k\|_{L^2(Q(t,2))}\\
			&+C|\scal{\nu_k}{E}_\eta|\left( \frac{f_k\mathbf{1}_{[0,1]\cup[\ell_k-1,\ell_k]}(s)}{\alpha_k} +\frac{1}{\ell_k \alpha_k}+ f_k\rho_k \right) \|\g Y_k\|_{L^2(Q(t,2))}^2\\
			&+C|\scal{\nu_k}{E}_\eta||\g Y|^2_\eta \left( f_k\mathbf{1}_{[0,1]\cup[\ell_k-1,\ell_k]}(s)+\frac{1}{\ell_k}\right)\alpha_k\\
			&+C|\scal{\nu_k}{E}_\eta|\|\g Y_k\|_{L^2(Q(t,2))}\left( f_k\mathbf{1}_{[0,1]\cup[\ell_k-1,\ell_k]}(s)+\frac{1}{\ell_k}\right)\\
			&+ C\left( f_k\mathbf{1}_{[0,1]\cup[\ell_k-1,\ell_k]}(s)+\frac{1}{\ell_k}\right)\alpha_k |\scal{E}{\g^2\nu_k}_\eta|.
		\end{align*}
		The term $\g d_k$ appears twice in \eqref{eq:norm_gE_k}. We estimate the term $(\g d_k)e_k$ term by the above estimate and \eqref{eq:est_ek}:
		\begin{align*}
			|\g d_k| |e_k| &\leq C|\scal{\nu_k}{E}_\eta| |\scal{\g \nu_k}{E}_\eta|\left( f_k\mathbf{1}_{[0,1]\cup[\ell_k-1,\ell_k]}(s) +\frac{1}{\ell_k}\right)^2\alpha_k \|\g Y_k\|_{L^2(Q(t,2))}\\
			&+C|\scal{\nu_k}{E}_\eta|^2 \left( f_k\mathbf{1}_{[0,1]\cup[\ell_k-1,\ell_k]}(s) +\frac{1}{\ell_k }+ \alpha_k f_k\rho_k \right)\left( f_k\mathbf{1}_{[0,1]\cup[\ell_k-1,\ell_k]}(s)+\frac{1}{\ell_k}\right) \|\g Y_k\|_{L^2(Q(t,2))}^2\\
			&+C|\scal{\nu_k}{E}_\eta|^2|\g Y|^2_\eta \left( f_k\mathbf{1}_{[0,1]\cup[\ell_k-1,\ell_k]}(s)+\frac{1}{\ell_k}\right)^2\alpha_k^2\\
			&+C|\scal{\nu_k}{E}_\eta|^2 \|\g Y_k\|_{L^2(Q(t,2))}\left( f_k\mathbf{1}_{[0,1]\cup[\ell_k-1,\ell_k]}(s)+\frac{1}{\ell_k}\right)^2 \alpha_k \\
			&+ C|\scal{\nu_k}{E}_\eta| \left( f_k\mathbf{1}_{[0,1]\cup[\ell_k-1,\ell_k]}(s)+\frac{1}{\ell_k}\right)^2\alpha_k^2 |\scal{E}{\g^2\nu_k}_\eta|.
		\end{align*}		
		We estimate the term $(\g d_k)\g(\rho_k f_k)$ by \eqref{derivative1_f}-\eqref{derivative1_rho} :
		\begin{align*}
			|\g d_k||\g(\rho_kf_k)|&\leq C|\scal{\g \nu_k}{E}_\eta|\left( f_k\mathbf{1}_{[0,1]\cup[\ell_k-1,\ell_k]}(s) +\frac{1}{\ell_k}\right)^2 \alpha_k \|\g Y_k\|_{L^2(Q(t,2))}\\
			&+C|\scal{\nu_k}{E}_\eta|\left( f_k\mathbf{1}_{[0,1]\cup[\ell_k-1,\ell_k]}(s) +\frac{1}{\ell_k }+ f_k\rho_k\alpha_k \right)\left( f_k\mathbf{1}_{[0,1]\cup[\ell_k-1,\ell_k]}(s)+\frac{1}{\ell_k}\right)  \|\g Y_k\|_{L^2(Q(t,2))}^2\\
			&+C|\scal{\nu_k}{E}_\eta||\g Y|^2_\eta \left( f_k\mathbf{1}_{[0,1]\cup[\ell_k-1,\ell_k]}(s)+\frac{1}{\ell_k}\right)^2\alpha_k^2\\
			&+C|\scal{\nu_k}{E}_\eta|\|\g Y_k\|_{L^2(Q(t,2))}\left( f_k\mathbf{1}_{[0,1]\cup[\ell_k-1,\ell_k]}(s)+\frac{1}{\ell_k}\right)^2 \alpha_k\\
			&+ C\left( f_k\mathbf{1}_{[0,1]\cup[\ell_k-1,\ell_k]}(s)+\frac{1}{\ell_k}\right)^2\alpha_k^2 |\scal{E}{\g^2\nu_k}_\eta|.
		\end{align*}
		Using the estimates $\scal{E_k}{\nu_k}_\eta \ust{k\to\infty}{=}o(1)$ and $\scal{E_k}{Y_k}_\eta \ust{k\to\infty}{=}o(1)$, from the above estimates and \eqref{eq:norm_gE_k}, we conclude :
		\begin{align*}
			\left| |\g E_k|^2_\eta - |\g(\rho_k f_k)|^2 \right| \leq &C(\rho_k f_k)^2 |\scal{E}{\g Y_k}_\eta|^2 \\
			&+ C|\scal{\g \nu_k}{E}_\eta|^2\left( f_k\mathbf{1}_{[0,1]\cup[\ell_k-1,\ell_k]}(s) +\frac{1}{\ell_k^2}\right)\left( \alpha_k^2 + \|\g Y_k\|_{L^2(Q(t,2))}^2 \right)\\
			&+o\left(|\g(\rho_k f_k)|^2 + \left( f_k\mathbf{1}_{[0,1]\cup[\ell_k-1,\ell_k]}(s) +\frac{1}{\ell_k^2}\right)\left( \alpha_k^2 + \|\g Y_k\|_{L^2(Q(t,2))}^2 \right) + f_k\rho_k |\g Y_k|^2_\eta \right)
		\end{align*}
		Therefore,
		\begin{align*}
			\left| \int_{\Cr_k} |\g E_k|^2_\eta -|\g(\rho_k f_k)|^2 \ dtd\theta \right| \leq &C\int_{\Cr_k} f_k^2 |\scal{E}{\g Y_k}_\eta|^2 dtd\theta \\
			&+ C\int_{\Cr_k} |\scal{\g \nu_k}{E}_\eta|^2\left( f_k\mathbf{1}_{[0,1]\cup[\ell_k-1,\ell_k]}(s) +\frac{1}{\ell_k^2}\right)\left( \alpha_k^2 + \|\g Y_k\|_{L^2(Q(t,2))}^2 \right) dtd\theta \\
			&+ o\left( \int_{\Cr_k}|\g(\rho_k f_k)|^2 + f_k\rho_k |\g Y_k|^2_\eta \ dtd\theta + \frac{\Ar(Y_k;\Cr_k)}{\ell_k^2} \right).
		\end{align*}
		Each integral is invariant by conformal transformation on $\Psi_k$. Therefore, we can use \eqref{as:gnu} to obtain
		\begin{align}\label{gradient_part_second_variation_Ek}
			\left| \int_{\Cr_k} |\g E_k|^2_\eta -|\g(\rho_k f_k)|^2 \ dtd\theta \right| \leq &  C\int_{\Cr_k} f_k^2 |\scal{E}{\g Y_k}_\eta|^2 dtd\theta \nonumber \\
			&+ o\left( \int_{\Cr_k}|\g(\rho_k f_k)|^2 + f_k\rho_k |\g Y_k|^2_\eta \ dtd\theta + \frac{\Ar(Y_k;\Cr_k)}{\ell_k^2} \right).
		\end{align}
		We compute the term $\int f_k^2 |\scal{E}{\g Y_k}_\eta|^2$. By integration by parts, it holds
		\begin{align*}
			\int_{\Cr_k} f_k^2 |\scal{E}{\g Y_k}_\eta|^2 dtd\theta &= -\int_{\Cr_k} f_k^2 \scal{E}{\lap Y_k}_\eta \scal{E}{Y_k}_\eta + 2f_k(\g f_k)\scal{E}{\g Y_k}\scal{E}{Y_k}\ dtd\theta.
		\end{align*}
		By Young's inequality
		\begin{align*}
			\int_{\Cr_k} f_k^2 |\scal{E}{\g Y_k}_\eta|^2 dtd\theta &\leq \int_{\Cr_k} f_k^2 |\g Y_k|^2_\eta \scal{E}{Y_k}_\eta^2 +C|\g f_k|^2 \scal{E}{Y_k}_\eta^2 + \frac{1}{4} f_k^2 |\scal{E}{\g Y_k}_\eta|^2 \ dtd\theta.
		\end{align*}
		Using $\scal{E}{Y_k}_\eta \ust{k\to\infty}{=}o(1)$, we obtain
		\begin{align}\label{integral_scalar_E_gYk}
			\int_{\Cr_k} f_k^2 |\scal{E}{\g Y_k}_\eta|^2 dtd\theta &\ust{k\to\infty}{=} o\left( \int_{\Cr_k} f_k^2 |\g Y_k|^2_\eta + |\g f_k|^2 \ dtd\theta\right).
		\end{align}
		Therefore, \eqref{gradient_part_second_variation_Ek} reduces to
		\begin{align}\label{integral_gEk}
			\int_{\Cr_k} |\g E_k|^2_\eta dtd\theta &\ust{k\to\infty}{=} \int_{\Cr_k} |\g (\rho_k f_k)|^2 dtd\theta + o\left( \int_{\Cr_k} |\g f_k|^2 +f_k^2 |\g Y_k|^2_\eta dtd\theta + \frac{\Ar(Y_k;\Cr_k)}{\ell_k^2} \right).
		\end{align}
		
		\underline{Estimate of the term $\int_{\Cr_k}|\scal{E_k}{\g Y_k}_\eta|^2 $ in \eqref{second_variation_dirichlet_Ek} :}\\
		We compute the scalar product in the isometry of $\s^{3,1}$ given by the \cref{bubbling_CGM} with $\delta=1$ : thanks to \eqref{definition_jacobi_field} and \eqref{derivatives_Y},
		\begin{align*}
			\scal{E_k}{\g Y_k}_\eta &= \scal{\rho_k f_k E}{ \g Y_k }_\eta - e_k \Arond_k.
		\end{align*}
		Thanks to \eqref{component_gnu}, \eqref{integral_scalar_E_gYk} and \eqref{derivative1_f} together with $\scal{E}{\nu_k}_\eta \ust{k\to\infty}{=}o(1)$, we obtain
		\begin{align}\label{scal_part_second_variation_Ek}
			\int_{\Cr_k} |\scal{E_k}{\g Y_k}_\eta |^2 dtd\theta \ust{k\to\infty}{=} o\left( \int_{\Cr_k}|\g f_k|^2 + f_k^2 |\g Y_k|^2_\eta \ dtd\theta \right).
		\end{align}
		
		\underline{Estimate of the term $\int_{\Cr_k} |E_k|^2_\eta |\g Y_k|^2_\eta$ in \eqref{second_variation_dirichlet_Ek} :}\\
		By a direct computation in the isometry of $\s^{3,1}$ given by the \cref{bubbling_CGM} with $\delta=1$, it holds
		\begin{align*}
			|E_k|^2_\eta &= (\rho_k f_k)^2 -2\rho_k^2 f_k^2\scal{E}{Y_k}_\eta^2 +2\rho_k f_k d_k \scal{E}{\nu_k}_\eta + 2 \rho_k f_k e^i_k \scal{E}{\dr_i \nu_k}_\eta \\
			&+ (\rho_k f_k)^2 \scal{E}{Y_k}_\eta^2 +|e^i_k \dr_i \nu_k|^2_\eta.
		\end{align*}
		By Young inequality,
		\begin{align*}
			\left| 2\rho_k f_k d_k \scal{E}{\nu_k}_\eta + 2\rho_k f_k e^i_k \scal{E}{\dr_i \nu_k}_\eta\right| &\leq \left( |\scal{E}{\nu_k}_\eta| + |\scal{E}{\g \nu_k}_\eta| \right) \left( \rho_k^2 f_k^2 + d_k^2 + |e_k|^2 \right).
		\end{align*}
		By the definitions \eqref{component_gnu}-\eqref{component_nu} together with the estimates \eqref{derivative1_f} and \eqref{derivative2_f}, we obtain
		\begin{align*}
			\left| 2\rho_k f_k d_k \scal{E}{\nu_k}_\eta + 2\rho_k f_k e^i_k \scal{E}{\dr_i \nu_k}_\eta\right| \ust{k\to\infty}{=} o\left( \rho_k^2 f_k^2 + |\g (\rho_k f_k)|^2 + \frac{\|\g Y_k\|_{L^2_\eta(Q(t,2))}^2 }{\ell_k^2} \right).
		\end{align*}
		Using \eqref{derivative1_rho}, we obtain the asymptotic expansion
		\begin{align}\label{norm_part_second_variation_Ek}
			\int_{\Cr_k} |E_k|^2_\eta |\g Y_k|^2_\eta dtd\theta &\ust{k\to\infty}{=} \int_{\Cr_k} \rho_k^2 f_k^2 |\g Y_k|^2_\eta dtd\theta + o \left( \int_{\Cr_k} |\g f_k|^2+ f_k^2 |\g Y_k|^2_\eta dtd\theta + \frac{\Ar(Y_k;\Cr_k)}{\ell_k^2 } \right).
		\end{align}
		Thanks to \eqref{assumption_small_energy}, we have $o\left(\frac{\Ar(Y_k;\Cr_k)}{\ell_k^2} \right) = o(\ell_k^{-2})$. We conclude by \eqref{integral_gEk}-\eqref{scal_part_second_variation_Ek}-\eqref{norm_part_second_variation_Ek}.
	\end{proof}
	
	Now, we show that we can get rid of $\rho_k$ in the expression $\left( \int_{\Cr_k} |\g (\rho_k f_k)|^2 -\rho_k^2 f_k^2 |\g Y_k|^2_\eta\ dtd\theta \right)$. To do so, we compute it with the change of variable \eqref{change_variable}. 
	\begin{lemma}
		It holds 
		\begin{align*}
			\int_{\Cr_k} |\g (\rho_k f_k)|^2 -\rho_k^2 f_k^2 |\g Y_k|^2_\eta\ dtd\theta &\leq \int_1^{\ell_k-1} \alpha_k \left( 4\pi |\dr_s f_k|^2_\eta -  f_k^2 \right) ds + o\left( \frac{1}{\ell_k^2} \right).
		\end{align*}
	\end{lemma}
	
	\begin{proof}
		We first show the following :
		\begin{align}\label{vanishing_integral_g_rho}
			\int_{\Cr_k} |\g \rho_k|^2 f_k^2 dtd\theta &\ust{k\to\infty}{=} o\left( \frac{1}{\ell_k^2} \right).
		\end{align}
		We consider the change of variable \eqref{change_variable}. It holds
		\begin{align*}
			\int_{\Cr_k} |\g \rho_k|^2 f_k^2 dtd\theta &= 2\pi \int_0^{\ell_k} \alpha_k (\dr_s \rho_k)^2 f_k(s)^2 ds \\
			&= 2\pi \int_{[0,1]\cup[\ell_k-1,\ell_k]} \alpha_k \sin\left( \frac{\pi s}{\ell_k} \right)^2 ds.
		\end{align*}
		Using $\sin(u)\ust{u\to 0}{=} u+o(u)$ and $\sin(\pi-t)=\sin(t)$, we obtain
		\begin{align*}
			\int_{\Cr_k} |\g \rho_k|^2 f_k^2 dtd\theta &\ust{k\to\infty}{=} \frac{2\pi^3}{\ell_k^2} \int_{[0,1]\cup[\ell_k-1,\ell_k]} \alpha_k ds.
		\end{align*}
		Since $\int_{[0,1]\cup[\ell_k-1,\ell_k]} \alpha_k ds\leq \|\g Y_k\|_{L^\infty_\eta(\Cr_k)}$, we obtain \eqref{vanishing_integral_g_rho}.\\
		Using that $\rho_k(s) = 1$ if $s\in[1,\ell_k-1]$, we obtain 
		\begin{align*}
			\int_{\Cr_k} |\g (\rho_k f_k)|^2 -\rho_k^2 f_k^2 |\g Y_k|^2_\eta\ dtd\theta \leq &\int_{\Cr_k} 2|(\g \rho_k) f_k|^2 + 2|\rho_k \g f_k|^2  -\rho_k^2 f_k^2 |\g Y_k|^2_\eta\ dtd\theta  \\
			\leq &\int_{a_k}^{b_k} \rho_k^2\Big( 4\pi (\dr_t f_k)^2 - f_k^2 \alpha_k^2 \Big) dt + o\left( \frac{1}{\ell_k^2} \right) \\
			\leq &\int_{0}^{\ell_k} \rho_k^2 \alpha_k \Big( 4\pi (\dr_s f_k)^2 - f_k^2  \Big) ds + o\left( \frac{1}{\ell_k^2} \right) \\
			\leq &\int_1^{\ell_k-1} \alpha_k \Big( 4\pi (\dr_s f_k)^2 - f_k^2  \Big) ds + o\left( \frac{1}{\ell_k^2} \right) \\
			&+ \int_{[0,1]\cup[\ell_k-1,\ell_k]}  \rho_k^2 \alpha_k \Big( 4\pi (\dr_s f_k)^2 - f_k^2  \Big) ds.
		\end{align*}
		Using $\dr_s f_k = \frac{\pi}{\ell_k} \cos\left( \frac{\pi s}{\ell_k} \right)$, we obtain an estimate for the last term :
		\begin{align*}
			\int_{[0,1]\cup[\ell_k-1,\ell_k]}  \rho_k^2 \alpha_k \Big( 4\pi (\dr_s f_k)^2 - f_k^2  \Big) ds &\ust{k\to\infty}{=} \int_{[0,1]}  \rho_k^2 \alpha_k \Big( \frac{4\pi^3}{\ell_k^2} - \frac{\pi^2 s^2}{\ell_k^2}  \Big) ds + \int_{[\ell_k-1,\ell_k]} \rho_k^2 \alpha_k \Big( \frac{4\pi^3}{\ell_k^2} - \frac{\pi^2 (\ell_k-s)^2}{\ell_k^2}  \Big) ds \\
			&\ust{k\to\infty}{=} o\left( \frac{1}{\ell_k^2} \right).
		\end{align*}
	\end{proof}
	
	To conclude the proof of the proposition \ref{existence_jacobi_field}, we prove that the quantity $\Big( \int 2|\g f_k|^2 - |\g Y_k|^2_\eta f_k^2 dtd\theta \Big)$ is exactly of size $\frac{1}{\ell_k^2}$.
	
	\begin{lemma}\label{bound_energy_fk}
		There exists $k_0 \in\N$ such that for $k\geq k_0$,
		\begin{align*}
			\int_1^{\ell_k-1} \alpha_k \left( 4\pi |\dr_s f_k|^2_\eta -  f_k^2 \right) ds <- \frac{ \pi^3 \mu}{2 \ell_k^2 }.
		\end{align*}
	\end{lemma}
	
	\begin{proof}
		We now compute the integral : 
		\begin{align*}
			\int_1^{\ell_k-1} \alpha_k \left( 4\pi |\dr_s f_k|^2_\eta -  f_k^2 \right) ds &= \int_1^{\ell_k-1} \alpha_k \left[ \frac{4\pi^3}{\ell_k^2} \cos\left( \frac{\pi s}{\ell_k} \right)^2 - \sin\left( \frac{\pi s}{\ell_k} \right)^2 \right]\ ds \\
			&= \int_1^{\ell_k-1} \alpha_k \left[ \frac{4\pi^3}{\ell_k^2} - \left( 1+  \frac{4\pi^3}{\ell_k^2}  \right)\sin\left( \frac{\pi s}{\ell_k} \right)^2 \right]\ ds. 
		\end{align*}
		We separate the integral on the domains $[2\pi ,\ell_k-2\pi]$ and $[1,2\pi]\cup [\ell_k-2\pi,\ell_k-1]$ :
		\begin{align*}
			\int_1^{\ell_k-1} \alpha_k \left( 4\pi |\dr_s f_k|^2_\eta -  f_k^2 \right) ds =& \int_{2\pi}^{\ell_k-2\pi} \alpha_k \left[ \frac{4\pi^3}{\ell_k^2} - \left( 1+  \frac{4\pi^3}{\ell_k^2}  \right)\sin\left( \frac{\pi s}{\ell_k} \right)^2 \right]\ ds \\
			&+ \int_1^{2\pi} \alpha_k \left[ \frac{4\pi^3}{\ell_k^2} - \left( 1+  \frac{2\pi^3}{\ell_k^2}  \right)\sin\left( \frac{\pi s}{\ell_k} \right)^2 \right]\ ds \\
			&+ \int_{\ell_k-2\pi}^{\ell_k-1} \alpha_k \left[ \frac{4\pi^3}{\ell_k^2} - \left( 1+  \frac{4\pi^3}{\ell_k^2}  \right)\sin\left( \frac{\pi s}{\ell_k} \right)^2 \right]\ ds.
		\end{align*}
		Since $\int_a^b \alpha_k(s) ds = \Ar(Y_k; [a,b]\times \s^1)$, we obtain
		\begin{align*}
			\int_1^{\ell_k-1} \alpha_k \left( 4\pi |\dr_s f_k|^2_\eta -  f_k^2 \right) ds \leq & \left[ \frac{4\pi^3}{\ell_k^2} - \left( 1+  \frac{4\pi^3}{\ell_k^2}  \right)\sin\left(\frac{2\pi^2}{\ell_k}\right)^2 \right] \Ar\Big(Y_k;[2\pi,\ell_k-2\pi]\times \s^1\Big) \\
			&+ \frac{4\pi^3}{\ell_k^2} \Ar\Big(Y_k;([0,2\pi]\cup[\ell_k-2\pi,\ell_k])\times \s^1\Big).
		\end{align*}
		Since $Y_k$ has average lorentz length $4\pi$ on $([0,2\pi]\cup[\ell_k-2\pi,\ell_k])\times \s^1$, we obtain
		\begin{align*}
			\Ar\Big(Y_k;([0,2\pi]\cup[\ell_k-2\pi,\ell_k])\times \s^1\Big)\leq C\|\g Y_k\|_{L^\infty_\eta (\Cr_k) } \xrightarrow[k\to\infty]{}{0}.
		\end{align*}
		Thanks to \eqref{assumption_small_energy}, for $k$ large enough, it holds
		\begin{align*}
			& \Ar\Big(Y_k;([0,2\pi]\cup[\ell_k-2\pi,\ell_k])\times \s^1\Big) \leq \frac{\mu}{10^{1000}},\\
			& \Ar\Big(Y_k;[2\pi,\ell_k - 2\pi]\times \s^1\Big) \geq \frac{\mu}{2}.
		\end{align*}
		Using the asymptotic expansion $\sin(u) \ust{u\to 0}{=} u+o(u)$, we obtain
		\begin{align*}
			\int_1^{\ell_k-1} \alpha_k \left( 4\pi |\dr_s f_k|^2_\eta -  f_k^2 \right) ds &\leq \left[ \frac{4\pi^3}{\ell_k^2} - \left( 1+  \frac{4\pi^3}{\ell_k^2}  \right) \left( \frac{2\pi^2}{\ell_k}  \right)^2  + o\left( \frac{1}{\ell_k^2} \right) \right] \frac{\mu}{2} + \frac{4\pi^3}{\ell_k^2} \frac{\mu}{10^{1000} } \\
			&\leq \left( 2\pi^3 -2\pi^4 + \frac{2\pi^3}{10^{1000}} \right) \frac{\mu}{\ell_k^2 } + o\left( \frac{1}{\ell_k^2} \right) \\
			&\leq -\frac{\pi^3 \mu}{\ell_k^2 } + o\left( \frac{1}{\ell_k^2} \right) .
		\end{align*}
	\end{proof}

	We conclude the prooof of the \cref{existence_jacobi_field} by choosing $V_k = E_k$. Thanks to the \cref{asympotic_second_derivative_dirichlet_Ek} and the  \cref{bound_energy_fk}, we obtain $\delta^2\Dr_{Y_k}(V_k) <0$ for $k$ large enough.

	\subsection{Lower bound on the Willmore index}\label{section_explosion_index}
	
	In this section, we prove that $\Ind_\Er(\Psi_k) \xrightarrow[k\to\infty]{}{+\infty}$.\\
	
	Thanks to \eqref{existence_energy}, for any $J\in\N$ and $k\in\N$, there exists a subdivision of $\Cr_k$ in $J$ subcylinders $\Cr_k^1 ,..., \Cr_k^J$ such that for any $j\in\inter{1}{J}$, it holds $\Ar(Y_k;\Cr_k^j) \in [ \frac{\lambda}{2J} , \frac{\lambda}{J} ]$. Thanks to the lemma \ref{existence_jacobi_field}, for each $j\in\inter{1}{J}$, there exists $k_{\lambda,j}\in\N$ such that for $k\geq k_{\lambda,j}$, there exists a vector field $V_k^j$ supported on $\Cr_k^j$ such that $\delta^2 \Dr_{Y_k}(V_k^j)<0$. Since $Y_k$ is a critical point of both $\Ar$ and $\Dr$, with $\Ar \leq \Dr$, it holds $\delta^2 \Ar_{Y_k}(V_k^j) \leq \delta^2 \Dr_{Y_k}(V_k^j) <0$. Since each $V_k^j$ comes from a variation through conformal Gauss maps, to each $V_k^j$ corresponds a Jacobi field for $\Psi_k$. Hence for $k\geq \max\{k_{\lambda,1},...,k_{\lambda,J}\}$, it holds $\Ind_\Er(\Psi_k) \geq J$. So for any $J\in\N$, we obtain $\liminf_{k\to\infty} \Ind_\Er(\Psi_k) \geq J$.

	\appendix
	
	\section{Quatization for $\Er$ is equivalent to quantization for $W$}\label{quantization_E_W}
	
	Recall that for any immersion $\Phi : \Sigma \to \R^3$, with $g_\Phi = \Phi^*\xi$, we have $|\Arond_\Phi|^2_{g_\Phi} = 2H^2_\Phi - 2 K_\Phi$. Therefore, to prove quantization for the lagragian $W$ knowing the quantization for the lagragian $\Er$, we only have to prove the quantization for $\Phi \mapsto \int_\Sigma K_\Phi d\vol_{g_\Phi}$.\\
	
	We recall \cite[Lemma V.1]{bernard2014} :
	\begin{lemma}
		There exists $\eta_0>0$ and $C>0$ such that for any dyadic annuli $ B_R(0)\setminus B_r(0) \subset \R^2$ with $R>4r$, $\eta \in (0,\eta_0)$, and any conformal immersion $\Phi : B_R(0)\setminus B_r(0) \to \R^3$ with $L^2$-bounded second fundamental form satisfying :
		\begin{align*}
			\|\g \vec{n}_\Phi \|_{L^{2,\infty}(B_R\setminus B_r)} + \int_{\dr B_r(0)} |\g \vec{n}_\Phi| + \|\g \vec{n}_\Phi \|_{L^2\big( [B_R\setminus B_{R/2}]\cup [B_{2r}\setminus B_r] \big)} &\leq \eta.
		\end{align*}
		Then, its Gauss curvature $K_\Phi$ satisfies
		\begin{align*}
			\left| \int_{B_R(0) \setminus B_r(0)} K_\Phi d\vol_{g_\Phi} \right| &\leq C\eta.
		\end{align*}
	\end{lemma}
	
	Consider a neck or a collar region $\Cr_k := B_{R_k} \setminus B_{r_k}$ with $\frac{r_k}{R_k} \xrightarrow[k\to \infty]{}{0}$, developped by a sequence $(\Phi_k)_{k\in\N}$ of Willmore immersions from $\Sigma\to \R^3$. Using the estimate \eqref{apriori_necks} and (VI.5)-(VI.6) in \cite{bernard2014}, we have that
	\begin{align*}
		\|\g \vec{n}_{\Phi_k} \|_{L^{2,\infty}(B_{4R_k/5}\setminus B_{5r_k/5})} &\leq C \sup_{\rho\in[r_k,R_k/2]} \|\g \vec{n}_{\Phi_k} \|_{L^2(B_{2\rho}\setminus B_\rho)}.
	\end{align*}
	So that
	\begin{align*}
		\left| \int_{\Cr_k} K_{\Phi_k} d\vol_{g_{\Phi_k}} \right| \ust{k\to \infty}{=} o(1).
	\end{align*}
	Assume that $(\Sigma,[g_{\Phi_k}])$ converges to a nodal surfaces $(\tilde{\Sigma},\tilde{h})$. Let $(\tilde{\Sigma}^l)_{l\in\inter{1}{q}}$ be the connected components of $\tilde{\Sigma}$. Assume that we have quantization : there exists $q$ branched immersions $\Phi^l_\infty : \tilde{\Sigma}^l \to \R^3$ and a finite number of immersions $\omega_j :\s^2\to \R^3$ which are Willmore away from possibly finitely many points, with a possible finite number of ends and branch points, such that, up to a subsequence,
	\begin{align*}
		\lim_{k\to \infty} \Er(\Phi_k) &= \sum_{l=1}^q \Er(\tilde{\Phi}^l_\infty) + \sum_{j=1}^p \Er(\omega_j).
	\end{align*}
	Then,
	\begin{align*}
		\lim_{k\to \infty} W(\Phi_k) &= \lim_{k\to \infty} \int_\Sigma \frac{1}{2}\left|\Arond_{\Phi_k} \right|^2_{g_{\Phi_k}} + K_{\Phi_k}\ d\vol_{g_{\Phi_k}} \\
		&= \sum_{l=1}^q \int_{\tilde{\Sigma}^l} \frac{1}{2}\left|\Arond_{\Phi_\infty^l} \right|^2_{g_{\Phi_\infty^l}} + K_{\Phi_\infty^l}\ d\vol_{g_{\Phi_\infty^l}} + \sum_{j=1}^p \int_{\s^2} \frac{1}{2}\left|\Arond_{\omega_j} \right|^2_{g_{\omega_j}} + K_{\omega_j}\ d\vol_{g_{\omega_j}} \\
		&= \sum_{l=1}^q W(\Phi^l_\infty) + \sum_{j=1}^p W(\omega_j).
	\end{align*}
	
	\section{Technical results for the \cref{gauge_neck_region} }\label{technical_lemmas_gauge}

	The two following proposition can be proved just as the corollaries 2.1 and 2.2 from \cite{bernard2020}, by changing the domains $B_1$ and $B_{1/2}$ by $B_2\setminus B_1$ and $B_{11/6}\setminus B_{7/6}$.
	\begin{proposition}\label{control_h_a}
		Let $\Phi : B_2(0)\setminus B_1(0) \to \R^3$ be a conformal Willmore immersion. There exists $\ve_1>0$ such that, if
		\begin{align*}
			\int_{B_2\setminus B_1} |\g \vec{n}|^2 < \ve_1,\ \ \ \ \ \| \g \lambda \|_{L^{2,\infty}(B_2\setminus B_1)} \leq C_0.
		\end{align*}
		for some $C_0>0$, then there exists $C=C(C_0)>0$ such that
		\begin{align*}
			\left\| \left( H - \fint_{B_{11/6}\setminus B_{7/6}} H \right) e^\lambda \right\|_{L^2(B_{11/6}\setminus B_{7/6})} &\leq C\left\| \Arond e^{-\lambda} \right\|_{L^2(B_2 \setminus B_1)}.
		\end{align*}
	\end{proposition}
	
	\begin{lemma}\label{average_spacelike}
		Let $\Phi : B_2(0)  \setminus B_1(0) \to \R^3$ be a conformal Willmore immersion such that its conformal factor $\lambda_\Phi$ satisfies $C^{-1} < e^{\lambda_\Phi} < C$, and $\|\Phi\|_{W^{1,\infty}(B_2(0)  \setminus B_1(0))} < C$.\
		There exists $\ve_2>0$ such that if
		\begin{align*}
			\int_{B_2(0)  \setminus B_1(0)} \left| \Arond_\Phi e^{-\lambda_\Phi} \right|^2 <\ve_2,
		\end{align*}
		Then, its conformal Gauss map $Y_\Phi$ satisfy
		\begin{align*}
			\left| \fint_{B_2(0)  \setminus B_1(0)} Y_\Phi \right|^2_\eta &\geq \frac{1}{2}.
		\end{align*}
	\end{lemma}
	
	The proofs rely on two main ingredients. First, the Gauss-Codazzi equations, which are independant of the domain. Second, on \cite[Theorem 3']{bourgain2003} stating that we can solve $\di Y = f$ on any Lipschitz, connected, bounded open set $\Omega\subset \R^2$, with the bounds $\|Y\|_{L^\infty(\Omega)} + \|Y\|_{W^{1,2}(\Omega)} \leq C(\Omega)\|f\|_{L^2(\Omega)}$. So all the arguments still hold if the domain is not simply connected.

	\section{Umbilic circles}\label{few_facts_CGM}
	
	\begin{lemma}\label{no_umbilic_circle}
		Let $\ve\in(0,1)$ and $k\in\N$. Consider a Willmore immersion $\Psi : [0,T]\times \s^1 \to \s^3$ not totally umbilic. There exists a diffeomorphism $f : [0,T]\times \s^1\to [0,T]\times \s^1$, such that $\Psi\circ f$ has no circle of umbilic points of the form $\{t\}\times \s^1$. We can choose $f$ such that $\|f-\mathrm{id}\|_{W^{k,\infty}([0,T]\times \s^1)}\leq \ve$.
	\end{lemma}
	\begin{proof}
		Indeed, assume that $\Arond_{\Psi} = 0$ on $\{t\}\times \s^1$ and consider a small neighbourhood $[t-s,t+s]\times \s^1$, for some $s\in(0,\frac{1}{2})$ such that $\Arond_\Psi$ doesn't vanish on $([t-s,t+s]\times \s^1) \setminus ( \{t\}\times \s^1)$. Then, choose $f : [t-s,t+s]\times \s^1 \to [t-s,t+s]\times \s^1$ defined by
		\begin{align*}
			\forall (\tau,\theta) \in [t-s,t+s]\times \s^1,\ \ \ f(\tau,\theta) = (\tau + a \eta(\tau)\sin(\theta), \theta),
		\end{align*}
		where $\eta : [t-s,t+s]\to[0,1]$ is a cutoff function such that $\eta = 0$ on $[t-s,t-2s/3]\cup[t+2s/3,t+s]$, $\eta = 1$ on $[t-s/3,t+s/3]$, and for any $i\in\inter{1}{k}$, $s^i|\g^i \eta| \leq C$, and $a>0$ is small enough in order to have $f : [t-s,t+s]\times \s^1 \to [t-s,t+s]\times \s^1$. Then there is no more $\tau \in[t-s,t+s]$ such that $\Arond_{\Psi\circ f} = \Arond_{\Psi}\circ f$ satisfy $(\Arond_{\Psi\circ f})_{|\{\tau\}\times \s^1} = 0$. We can estimate :
		\begin{align*}
			\|f-\mathrm{id}\|_{W^{k,\infty}} &\leq a \|\eta(\tau) \sin(\theta)\|_{W^{k,\infty}} \leq C a \left(  \sum_{i=0}^k \frac{1}{s^i} \right).
		\end{align*}
		Up to reduce again $a$, it holds $\|f-\mathrm{id}\|_{W^{k,\infty}}\leq \ve$.
	\end{proof}

	\bibliographystyle{plain}
	\bibliography{biblio1}

\begin{thebibliography}{10}

\bibitem{akivis1996}
M.~A. Akivis and V.~V. Gol'dberg.
\newblock {\em Conformal Differential Geometry and Its Generalizations}.
\newblock Pure and Applied Mathematics. {Wiley}, {New York}, 1996.

\bibitem{bernard2016}
Yann Bernard.
\newblock Noether's theorem and the {{Willmore}} functional.
\newblock {\em Advances in Calculus of Variations}, 9(3):217--234, 2016.

\bibitem{bernard2020}
Yann Bernard, Paul Laurain, and Nicolas Marque.
\newblock Energy {{Estimates}} for the {{Tracefree Curvature}} of {{Willmore
  Surfaces}} and {{Applications}}, September 2020.

\bibitem{bernard2014}
Yann Bernard and Tristan Rivi{\`e}re.
\newblock Energy quantization for {{Willmore}} surfaces and applications.
\newblock {\em Annals of Mathematics. Second Series}, 180(1):87--136, 2014.

\bibitem{bernard2019}
Yann Bernard and Tristan Rivi{\`e}re.
\newblock Uniform regularity results for critical and subcritical surface
  energies.
\newblock {\em Calculus of Variations and Partial Differential Equations},
  58(1):Art. 10, 39, 2019.

\bibitem{blaschke1955}
Wilhelm Blaschke.
\newblock {\em Vorlesungen \"Uber {{Integralgeometrie}}}.
\newblock {Deutscher Verlag der Wissenschaften, Berlin}, 1955.

\bibitem{bourgain2003}
Jean Bourgain and Ha{\"i}m Brezis.
\newblock On the equation div {{Y}}=f and application to control of phases.
\newblock {\em Journal of the American Mathematical Society}, 16(2):393--426,
  2003.

\bibitem{bryant1984}
Robert~L. Bryant.
\newblock A duality theorem for {{Willmore}} surfaces.
\newblock {\em Journal of Differential Geometry}, 20(1):23--53, 1984.

\bibitem{chen2012}
Li~Chen, Yuxiang Li, and Youde Wang.
\newblock The {{Refined Analysis}} on the {{Convergence Behavior}} of
  {{Harmonic Map Sequence}} from {{Cylinders}}.
\newblock {\em Journal of Geometric Analysis}, 22(4):942--963, October 2012.

\bibitem{eschenburg}
Jost-Hinrich Eschenburg.
\newblock Willmore surfaces and {{Moebius Geometry}}.

\bibitem{hertrich-jeromin2003}
Udo {Hertrich-Jeromin}.
\newblock {\em Introduction to {{M\"obius}} Differential Geometry}, volume 300
  of {\em London {{Mathematical Society Lecture Note Series}}}.
\newblock {Cambridge University Press, Cambridge}, 2003.

\bibitem{hirsch2019}
J.~Hirsch and E.~{Mader-Baumdicker}.
\newblock On the {{Index}} of {{Willmore}} spheres.
\newblock {\em arXiv: Differential Geometry}, May 2019.

\bibitem{hirsch2021}
Jonas Hirsch, Rob Kusner, and Elena {M{\"a}der-Baumdicker}.
\newblock Geometry of complete minimal surfaces at infinity and the
  {{Willmore}} index of their inversions, November 2021.

\bibitem{hummel1997}
Christoph Hummel.
\newblock {\em Gromov's Compactness Theorem for Pseudo-Holomorphic Curves},
  volume 151 of {\em Progress in {{Mathematics}}}.
\newblock {Birkh\"auser Verlag, Basel}, 1997.

\bibitem{kuwert2001}
Ernst Kuwert and Reiner Sch{\"a}tzle.
\newblock The {{Willmore Flow}} with {{Small Initial Energy}}.
\newblock {\em Journal of Differential Geometry}, 57(3):409--441, March 2001.

\bibitem{kuwert2012}
Ernst Kuwert and Reiner Sch{\"a}tzle.
\newblock The {{Willmore}} functional.
\newblock In Giuseppe Mingione, editor, {\em Topics in {{Modern Regularity
  Theory}}}, {{CRM Series}}, pages 1--115, {Pisa}, 2012. {Edizioni della
  Normale}.

\bibitem{laurain2014}
Paul Laurain and Tristan Rivi{\`e}re.
\newblock Angular energy quantization for linear elliptic systems with
  antisymmetric potentials and applications.
\newblock {\em Analysis \& PDE}, 7(1):1--41, May 2014.

\bibitem{laurain2018a}
Paul Laurain and Tristan Rivi{\`e}re.
\newblock Energy quantization of {{Willmore}} surfaces at the boundary of the
  moduli space.
\newblock {\em Duke Mathematical Journal}, 167(11):2073--2124, August 2018.

\bibitem{laurain2018}
Paul Laurain and Tristan Rivi{\`e}re.
\newblock Optimal estimate for the gradient of {{Green}}'s function on
  degenerating surfaces and applications.
\newblock {\em Communications in Analysis and Geometry}, 26(4):887--913, July
  2018.

\bibitem{li2017}
Yuxiang Li, Lei Liu, and Youde Wang.
\newblock Blowup behavior of harmonic maps with finite index.
\newblock {\em Calculus of Variations and Partial Differential Equations},
  56(5):146, October 2017.

\bibitem{marque2021}
Nicolas Marque.
\newblock Conformal {{Gauss}} map geometry and application to {{Willmore}}
  surfaces in model spaces.
\newblock {\em Potential Analysis. An International Journal Devoted to the
  Interactions between Potential Theory, Probability Theory, Geometry and
  Functional Analysis}, 54(2):227--271, 2021.

\bibitem{marques2014b}
Fernando~C. Marques and Andr{\'e} Neves.
\newblock Min-max theory and the {{Willmore}} conjecture.
\newblock {\em Annals of Mathematics. Second Series}, 179(2):683--782, 2014.

\bibitem{michelat2021}
Alexis Michelat.
\newblock On the {{Morse}} index of branched {{Willmore}} spheres in 3-space.
\newblock {\em Calculus of Variations and Partial Differential Equations},
  60(4):126, June 2021.

\bibitem{michelat2016}
Alexis Michelat and Tristan Rivi{\`e}re.
\newblock A {{Viscosity}} method for the min-max construction of closed
  geodesics.
\newblock {\em ESAIM: Control, Optimisation and Calculus of Variations},
  22(4):1282--1324, October 2016.

\bibitem{palmer1991}
Bennett Palmer.
\newblock The conformal {{Gauss}} map and the stability of {{Willmore}}
  surfaces.
\newblock {\em Annals of Global Analysis and Geometry}, 9(3):305--317, January
  1991.

\bibitem{parker1996}
Thomas~H. Parker.
\newblock Bubble tree convergence for harmonic maps.
\newblock {\em Journal of Differential Geometry}, 44(3):595--633, 1996.

\bibitem{pinkall1985}
U.~Pinkall.
\newblock Hopf tori in {$\mathbb{S}^3$}.
\newblock {\em Inventiones Mathematicae}, 81(2):379--386, 1985.

\bibitem{riviere2008}
Tristan Rivi{\`e}re.
\newblock Analysis aspects of {{Willmore}} surfaces.
\newblock {\em Inventiones mathematicae}, 174(1):1--45, October 2008.

\bibitem{riviere2014}
Tristan Rivi{\`e}re.
\newblock Variational principles for immersed surfaces with {$L^2$}-bounded
  second fundamental form.
\newblock {\em Journal f\"ur die Reine und Angewandte Mathematik. [Crelle's
  Journal]}, 695:41--98, 2014.

\bibitem{riviere2013a}
Tristan Rivi\'{e}re.
\newblock Weak immersions of surfaces with {$L^2$}-bounded second fundamental
  form.
\newblock In {\em Geometric analysis}, volume~22 of {\em IAS/Park City Math.
  Ser.}, pages 303--384. Amer. Math. Soc., Providence, RI, 2016.

\bibitem{riviere2017}
Tristan Rivi{\`e}re.
\newblock A viscosity method in the min-max theory of minimal surfaces.
\newblock {\em Publications math\'ematiques de l'IH\'ES}, 126(1):177--246,
  November 2017.

\bibitem{riviere2020}
Tristan Rivi{\`e}re.
\newblock Willmore minmax surfaces and the cost of the sphere eversion.
\newblock {\em Journal of the European Mathematical Society}, 23(2):349--423,
  October 2020.

\bibitem{riviere2021}
Tristan Rivi{\`e}re.
\newblock Lower {{Semi-continuity}} of the {{Index}} in the {{Viscosity
  Method}} for {{Minimal Surfaces}}.
\newblock {\em International Mathematics Research Notices}, 2021(8):5651--5675,
  April 2021.

\bibitem{sacks1981}
J.~Sacks and K.~Uhlenbeck.
\newblock The {{Existence}} of {{Minimal Immersions}} of 2-{{Spheres}}.
\newblock {\em Annals of Mathematics}, 113(1):1--24, 1981.

\bibitem{schatzle2017}
Reiner~M. Sch{\"a}tzle.
\newblock The umbilic set of {{Willmore}} surfaces, October 2017.

\bibitem{willmore1965}
T.~J. Willmore.
\newblock Note on embedded surfaces.
\newblock {\em Analele \c{S}tiin\c{t}ifice ale Universit\u{a}\c{t}ii "Al. I.
  Cuza" din Iasi. Sec\c{t}iunea I a. Matematic\u{a}. Serie Nou\u{a}},
  11B:493--496, 1965.

\bibitem{zhu2008}
Miaomiao Zhu.
\newblock Harmonic maps from degenerating {{Riemann}} surfaces.
\newblock {\em Mathematische Zeitschrift}, 264(1):63, November 2008.

\bibitem{zhu2013}
Miaomiao Zhu.
\newblock Regularity for harmonic maps into certain pseudo-{{Riemannian}}
  manifolds.
\newblock {\em Journal de Math\'ematiques Pures et Appliqu\'ees},
  99(1):106--123, January 2013.

\end{thebibliography}

\end{document}